\documentclass[11pt,oneside,reqno]{amsart}
\usepackage{amsmath,amsfonts,amsthm,amssymb}
\usepackage{tikz}
\usetikzlibrary{cd,arrows,decorations.pathmorphing,decorations.markings,backgrounds,positioning,fit,calc,shapes,patterns}
\usepackage{setspace}
\usepackage{fullpage}
\usepackage{mathtools}
\usepackage[shortlabels]{enumitem}
\usepackage[subrefformat=parens, labelfont=up]{subcaption}
\usepackage{hyperref}

\usepackage{etoolbox}
\usepackage{environ}
\newtoggle{preliminarydraft}

\iftoggle{preliminarydraft}{
    \usepackage[inline,draft]{showlabels}
}

\newcommand{\draftnewpage}{
    \iftoggle{preliminarydraft}{
        \newpage
    }
}

\newtheorem{thm}{Theorem}[section]

\newtheorem{innercustomthm}{Theorem}
\newenvironment{mythm}[1]
  {\renewcommand\theinnercustomthm{#1}\innercustomthm}
  {\endinnercustomthm}

\newtheorem{lemma}[thm]{Lemma}

\newtheorem{coro}[thm]{Corollary}
\newtheorem{prop}[thm]{Proposition}
\theoremstyle{definition}
\newtheorem{defn}[thm]{Definition}

\theoremstyle{remark}

\theoremstyle{remark}

\usepackage{manfnt}
\newenvironment{ex}{\refstepcounter{thm}\begin{proof}[Example \emph{\thethm}]}{\end{proof}}
\newenvironment{rem}{\refstepcounter{thm}\begin{proof}[Remark \emph{\thethm}]}{\end{proof}}
\newenvironment{conv}{\refstepcounter{thm}\begin{proof}[Convention \emph{\thethm}]}{\end{proof}}
\newenvironment{warn}{\refstepcounter{thm}\begin{proof}[Warning \emph{\thethm}]}{\end{proof}}

\newenvironment{cons}{\refstepcounter{thm}\begin{proof}[Construction \emph{\thethm} ]}{\end{proof}}
\numberwithin{equation}{section}

\newcommand{\gvec}[1]{\vec{\mathsf{#1}}}

\colorlet{dark purple}{red!35!blue}
\colorlet{dark green}{green!70!black}
\colorlet{dark red}{red!80!black}
\colorlet{dark blue}{blue!80!black!80!cyan}

\tikzstyle{mutable}=[inner sep=0.9mm,rounded corners,draw,minimum size=2mm]
\tikzstyle{frozen}=[inner sep=.9mm,rectangle,draw]
\tikzstyle{dot} = [fill=black!25,inner sep=0.5mm,circle,draw,minimum size=.5em]
\tikzstyle{marked}=[inner sep=0.5mm,circle,draw,blue!75!black,fill=blue!50]

\tikzset{bpoint/.style = {shape=circle,fill=black,draw,minimum size=.5em, inner sep=0}}
\tikzset{wpoint/.style = {shape=circle,thick,fill=white,draw,minimum size=.5em, inner sep=0}}

\def\D{\mathsf{D}}

\def\Hom{\operatorname{Hom}}
\def\Spec{\operatorname{Spec}}
\def\CA{\mathcal{A}}
\def\kk{\Bbbk}
\def\SS{\Sigma}
\def\MM{M}
\def\Sk{\mathrm{Sk}}
\def\cV{\mathcal{V}}

\title{Deep Points of Cluster Algebras}  

\author{James Beyer}

\author{Greg Muller}


\begin{document}

\maketitle 

\begin{abstract}

We initiate a systematic study of the \textbf{deep points} of a cluster algebra; that is, the points in the associated variety which are not in any cluster torus. We describe the deep points of cluster algebras of type A, rank 2, Markov, and unpunctured surface type.
\end{abstract}

\section{Introduction}

\subsection{Cluster varieties and cluster tori}
Cluster algebras are commutative domains with distinguished elements, called \emph{cluster variables}, and distinguished sets of cluster variables, called \emph{clusters}. 
Cluster algebras were introduced in \cite{FZ02} to axiomatize patterns found in the coordinate rings of Lie groups, and since then they have been found in the coordinate rings of many important varieties, such as Grassmannians, Teichm{\"u}ller spaces, and affine log Calabi-Yau varieties.

We may associate a geometric object to an arbitrary cluster algebra $\CA$ by choosing a field $\kk$ and considering the associated \textbf{cluster variety (of $\kk$-points of $\CA$)}:
\[ V(\CA,\kk) \coloneqq \Hom_{\mathrm{Ring}}(\CA,\kk) \coloneqq \{\text{ring homomorphisms $p:\CA\rightarrow \kk$}\} \]
The examples mentioned above can be recovered in this way, as well as many other varieties.



One of the foundational properties of a cluster algebra is the \emph{Laurent phenomenon}, which says that the localization of a cluster algebra $\CA$ at the cluster variables in any cluster $\{x_1,x_2,\dotsc ,x_n\}$ is the ring of integral Laurent polynomials in those variables; that is,
\[ 
\CA[x_1^{-1},x_2^{-1},\dotsc ,x_n^{-1}] = \mathbb{Z} [x_1^{\pm1},x_2^{\pm1},\dotsc ,x_n^{\pm1}] \]
Applying $\mathrm{Hom}_{\text{Ring}}(-,\kk)$ to the map $\CA\hookrightarrow \mathbb{Z} [x_1^{\pm1},x_2^{\pm1},\dotsc ,x_n^{\pm1}]$ gives an open inclusion of varieties:
\[
(\kk^\times)^n \simeq \Hom_{\mathrm{Ring}}(\mathbb{Z} [x_1^{\pm1} ,x_2^{\pm1},\dotsc ,x_n^{\pm1}],\kk)\hookrightarrow \Hom_{\mathrm{Ring}}(\CA,\kk) \eqqcolon V(\CA,\kk)
\]
The image of $(\kk^\times)^n$ under this inclusion is called the \textbf{cluster torus} of the cluster $\{x_1,x_2,\dotsc ,x_n\}$; explicitly, it is the subset of $V(\CA,\kk)$ on which the none of the cluster variables $x_1,x_2,\dotsc ,x_n$ vanish. 

%

One might hope that the cluster tori cover the cluster variety; however, this is false even in some of the nicest examples. Worse, it is only \emph{sometimes} false, as the following example shows.



\begin{ex}
The cluster variety of type $A_n$ (see Section \ref{section: polygons}) is covered by its cluster tori, unless 
\begin{itemize}
	\item $n\equiv 3$ mod $4$, or
	\item $n\equiv 1$ mod $4$ and $\operatorname{char}(\kk)=2$.
\end{itemize}
In these cases, the union of the cluster tori misses a single point.
\end{ex}

\subsection{Deep points}

We therefore turn our attention to the set of points \emph{not} covered by the cluster tori, called the \textbf{deep points} of the cluster variety. The set of deep points, called the \textbf{deep locus}, forms a closed (possibly empty) subvariety of $V(\CA,\kk)$ (see Propostion \ref{prop: deepideal}).


\begin{ex}
The cluster variety of type $A_n$ (with no frozens) has the following well-known description. Let $\Delta_{n+3}$ denote a convex $(n+3)$-gon, and consider the set of \emph{diagonals} of $\Delta_{n+3}$ connecting non-adjacent vertices. 
Then 
a point $p\in  V(\CA(A_{n}),\kk)$) is equivalent to a map
\[ p:\{\text{diagonals in $\Delta_{n+3}$}\} \rightarrow \kk \]
such that the \emph{Ptolemy relations} hold; that is, the product of two crossing diagonals are equal to the sum of the products of opposite sides (Figure \ref{fig: introptolemy}). The edges of $\Delta_{n+3}$ implicitly have value $1$.
In this model, each cluster corresponds to a triangulation of $\Delta_{n+3}$, and the corresponding cluster torus consists of points which have non-zero values on each diagonal in that triangulation.

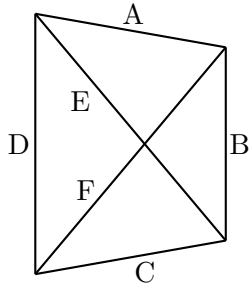
\begin{figure}[h!t]
\begin{tikzpicture}[scale=2]
\draw[thick] (40:1)--(120:1) {}; 
\draw[thick] (120:1)--(240:1) {}; 
\draw[thick] (240:1)--(320:1) {}; 
\draw[thick] (320:1)--(40:1) {}; 
\draw[thick] (40:1)--(240:1) {}; 
\draw[thick] (320:1)--(120:1) {};
\node at (80:0.88) {A};
\node at (180:0.61) {D};
\node at (285:0.88) {C};
\node at (0:0.86) {B};
\node at (125:0.35) {E};
\node at (242:0.35) {F};
\end{tikzpicture}
\caption{A quadralateral with Ptolemy relation $EF=AC+BD$}
\label{fig: introptolemy}
\end{figure}

A deep point in $V(\CA(A_n),\kk)$ is therefore a function on the set of diagonals of $\Delta_{n+3}$ which 
\begin{itemize}
    \item satisfies the Ptolemy relations, and
    \item kills at least one diagonal in every triangulation. 
\end{itemize}
We encourage the reader to verify the following facts for themselves.
\begin{itemize}
    \item There are no deep points in $V(\CA(A_2),\kk)$. That is, it is impossible to assign values to the five diagonals in a pentagon so that the Ptolemy relations hold and each triangulation contains a diagonal with value 0. As a consequence, $V(\CA(A_2),\kk)$ is covered by its 5 cluster tori.
    \item There is a unique deep point in $V(\CA(A_3),\kk)$, given by the values in Figure \ref{fig: deepA3}. \qedhere
\end{itemize}

\end{ex}

\begin{figure}[h!t]
\[
\begin{tikzpicture}[xscale=-1,scale=.75]
    \draw[fill=black!10,thick] 
    (0*60:2)
    to (1*60:2)
    to (2*60:2)
    to (3*60:2)
    to (4*60:2)
    to (5*60:2)
    to (6*60:2)
    ;
    \draw[thick] (5*60:2) to (3*60:2);
    \node at (245:1.27) {$0$};
    \draw[thick] (3*60:2) to (6*60:2);
    \node at (140:0.4) {$-1$};
    \draw[thick] (6*60:2) to (2*60:2);
    \node at (70:1.31) {$0$};
\end{tikzpicture}
\hspace{1cm}
\begin{tikzpicture}[xscale=-1,scale=.75]
    \draw[fill=black!10,thick] 
    (0*60:2)
    to (1*60:2)
    to (2*60:2)
    to (3*60:2)
    to (4*60:2)
    to (5*60:2)
    to (6*60:2)
    ;
    \draw[thick] (4*60:2) to (2*60:2);
    \node at (180:1.22) {$0$};
    \draw[thick] (2*60:2) to (5*60:2);
    \node at (90:0.9) {$-1$};
    \draw[thick] (5*60:2) to (1*60:2);
    \node at (0:1.22) {$0$};
\end{tikzpicture}
\hspace{1cm}
\begin{tikzpicture}[xscale=-1,scale=.75]
    \draw[fill=black!10,thick] 
    (0*60:2)
    to (1*60:2)
    to (2*60:2)
    to (3*60:2)
    to (4*60:2)
    to (5*60:2)
    to (6*60:2)
    ;
    \draw[thick] (3*60:2) to (1*60:2);
    \node at (110:1.31) {$0$};
    \draw[thick] (1*60:2) to (4*60:2);
    \node at (30:0.9) {$-1$};
    \draw[thick] (4*60:2) to (0*60:2);
    \node at (295:1.29) {$0$};
\end{tikzpicture}
\]
\caption{The unique deep point in the $A_3$ cluster variety (with no frozens)}
\label{fig: deepA3}
\end{figure}
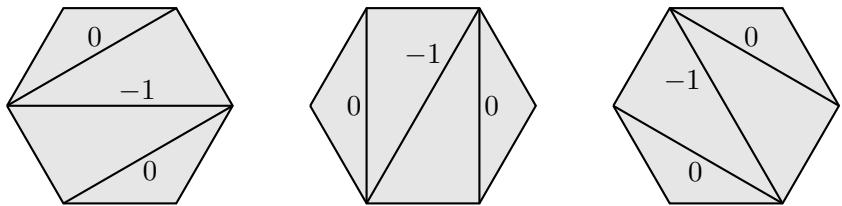

\subsection{Deep points and defects}

To stoke the reader's interest, we propose the following maxim.


\begin{center}
\fbox{\textit{
Any bad property of a cluster algebra can be blamed on its deep points.
}}
\end{center}
As a simple example, any singular points of $V(\CA,\kk)$ must be deep points (Proposition \ref{prop: deepsingular}).\footnote{
The simplest examples of deep points coincide with the singular points, but not every deep point is singular; see Remarks \ref{rem: polygonsingular} and \ref{rem: kronecker}.
}
%

As a contrapositive to the maxim, cluster algebras without deep points should be exceedingly well-behaved. The following theorem collects several notable instances of this phenomenon.

\begin{mythm}{\ref{thm: properties no deep points}}
Let $\CA$ be a cluster algebra with no deep points for all fields $\kk$. Then

\begin{enumerate}
    \item $\CA$ is non-singular.
    \item $\CA$ is finitely generated.
    \item $\CA$ is equal to its upper cluster algebra.
    \item $\CA$ is a locally acyclic cluster algebra.
\end{enumerate}
\end{mythm}

\begin{warn}
The converse to the maxim is generally false, and cluster algebras with deep points can be perfectly well-behaved. For example, all four properties in Theorem \ref{thm: properties no deep points} hold for the cluster algebra of the hexagon (with boundary coefficients), which has deep points by Theorem \ref{thm: deeppolygon}.
\end{warn}

\subsection{Summary of results}

The bulk of this note is dedicated to classifying and parametrizing the deep points of several concrete cluster algebras, with results summarized below. 

\textbf{For simplicity, we assume $\kk$ is algebraically closed and characteristic $0$ below.}
\begin{itemize}
    \item The $A_n$ cluster algebra (with no coefficients) has no deep points, except when $n\equiv 3\bmod{4}$, in which case it has a unique deep point.
    \item The rank 2 cluster algebra of type $(b,c)$ (with no coefficients) has a discrete set of 
    \[
    \left\{
    \begin{array}{cc}
    0 & \text{if $c=1$} \\
    b & \text{if $c>1$}
    \end{array}
    \right\}
    +
    \left\{
    \begin{array}{cc}
    0 & \text{if $b=1$} \\
    c & \text{if $b>1$}
    \end{array}
    \right\}
    \]
    many deep points; e.g.~types $(1,1)$, $(2,1)$, and $(2,2)$ have $0, 1,$ and $4$ deep points, respectively.

    \item The Markov cluster algebra has six deep lines which all intersect at a single point. By contrast, the Markov upper cluster algebra has seven deep lines which all intersect at a single point, and the map $V(\mathcal{U},\kk)\rightarrow V(\CA,\kk)$ collapses one of these lines to a point.
    
    
    \item Let $\SS$ be a connected, unpunctured, triangulable marked surface with genus $g$, $b$-many boundary components, and $m$-many marked points; and let $\CA(\SS)$ be the cluster algebra of $\SS$ (with boundary coefficients).
    \begin{itemize}
        \item If $m$ is even, then the deep locus of $\CA(\SS)$ consists of $2^{2g+b-1}$-many disjoint copies of the algebraic torus $(\kk^\times)^{2g+b+m-2}$. These components are in bijection with many interesting sets (see appendix).
        \item If $m$ is odd, then $\CA(\SS)$ has no deep points.
    \end{itemize}
    E.g.~the annulus with 2 marked points on each boundary has $2$ deep tori of dimension $4$.
    
\end{itemize}

\subsection{Related work and acknowledgements}

While this paper was in preparation, the authors became aware of a related work in preparation by Castronovo, Gorsky, Simental, and Speyer \cite{CGSS24arxiv}. They conjecture that, in many cases, deep points may be characterized as the points in a cluster variety with non-trivial stabilizer for a certain group action. We have since had several productive conversations with them, and both papers are better for it.

While completing this work, J.B. was partially supported by a David and Judi Proctor Department of Mathematics Graduate Fellowship and a Robberson Travel Grant, both from the University of Oklahoma.

\setcounter{tocdepth}{1}
\tableofcontents

\draftnewpage

\section{Deep points in general cluster varieties}

From this point on, we assume the reader is familiar with a basic introduction to cluster algebras, such as the recent book of Fomin-Williams-Zelevinsky (\cite{FWZaa,FWZbb,FWZcc,FWZdd}). Unless otherwise specified, we will work with skew-symmetrizable cluster algebras over $\mathbb{Z}$; frozen variables are allowed but the cluster algebra must include their inverses. 


\subsection{Cluster varieties}


Given a cluster algebra $\CA$ and a field $\kk$, the \textbf{cluster variety} is the set
\[ V(\CA,\kk) \coloneqq \Hom_{\mathrm{Ring}}(\CA,\kk) \coloneqq \{\text{ring homomorphisms $p:\CA\rightarrow \kk$}\} \]
This has the \emph{Zariski topology} in which a closed set consists of homomorphisms which send a fixed ideal to 0. 
Each element $a\in \CA$ determines a function $f_a:V(\CA,\kk)\rightarrow \kk$ by the rule that $f_a(p) \coloneqq p(a)$. The notation $f_a(p)$ is more useful when regarding $V(\CA,\kk)$ as a geometric object with points $\{p\}$, and the notation $p(a)$ is more useful when regarding it as a set of homomorphisms $\{p\}$.






\begin{rem}
If $\CA$ is finitely generated, then $V(\CA,\kk)$ may be identified with an affine $\kk$-variety. Explicitly, a finite presentation $ \CA \simeq \mathbb{Z}[g_1,g_2,\dotsc ,g_m] / \langle p_1,p_2,\dotsc ,p_n\rangle$ induces a homeomorphism between $V(\CA,\kk)$ and the solution set of the system $\{p_1=0,p_2=0,\dotsc ,p_n=0\}$ inside $\kk^m$. 

In general, $V(\CA,\kk)$ is the set of $\kk$-points of the affine scheme $\Spec(\CA)$. While this scheme need not be finite-type, we still refer to $V(\CA,\kk)$ as a `variety' for simplicity. Many of the results of this paper can be reformulated for schemes, but we use varieties to appeal to a wider audience.
\end{rem}


By the Laurent phenemenon, each cluster $\{x_1,x_2,\dotsc ,x_n\}\subset \mathcal{A}$
gives an inclusion of rings
\[ \CA \hookrightarrow \CA[x_1^{-1},x_2^{-1},\dotsc ,x_n^{-1}] =\mathbb{Z} [x_1^{\pm1} ,x_2^{\pm1},\dotsc ,x_n^{\pm1}] \]
A homomorphism $p:\CA\rightarrow \kk$ factors through this localization map iff $p(x_1),p(x_2), \dotsc ,p(x_n)$ are non-zero. Therefore, we have an open inclusion
\[
(\kk^\times)^n \simeq \Hom_{\mathrm{Ring}}(\mathbb{Z} [x_1^{\pm1} ,x_2^{\pm1},\dotsc ,x_n^{\pm1}],\kk)\hookrightarrow \Hom_{\mathrm{Ring}}(\CA,\kk) \eqqcolon V(\CA,\kk)
\]
whose image is the set of homomorphisms $p:\CA\rightarrow \kk$ such that $p(x_1),p(x_2), \dotsc ,p(x_n)$ are non-zero; or equivalently, the open subset of $V(\CA,\kk)$ where none of the functions $f_{x_1},f_{x_2},\dotsc ,f_{x_n}$ vanish. 
This image is called the \textbf{cluster torus} associated to the cluster $\{x_1,x_2,\dotsc ,x_n\}$.\footnote{The reader is cautioned that a cluster torus is not topological torus; rather, it is a \emph{split algebraic torus}.} 


\begin{warn}\label{warn: clustervariety}
The phrase `cluster variety' has been used in at least two other ways in the literature:
\begin{itemize}
    \item The union of the cluster tori (e.g.~\cite{FG09}).
    \item A certain \emph{blowup at infinity} of an algebraic torus (e.g.~\cite{GHKBirat15}).
\end{itemize}
These references generally only consider varieties `up to birational equivalence' or `up to codimension 2' (e.g.~\cite[Section 3.2]{GHKBirat15}), in which case they are equivalent to our definition. By contrast, we are interested in the special points and fine geometry of $V(\CA,\kk)$, so we need to be more explicit and careful about which definition we are using.
\end{warn}

\subsection{Deep points}


The goal of this paper is the study of the points in a cluster variety which are not in any cluster torus, which we call \emph{deep points}.


\begin{defn}
A point $p\in V(\CA,\kk)$ is a \textbf{deep point} if, for all clusters $\{x_1,\dotsc ,x_n\}\subset \CA$, there is a cluster variable $x_i\in \{x_1,\dotsc ,x_n\}$ such that $p(x_i)=0$.
\end{defn}

\noindent The set of all deep points in $V(\CA,\kk)$ is called the \textbf{deep locus} of $V(\CA,\kk)$.





\begin{prop}\label{prop: deepsingular}
A singular point in a cluster variety must be deep.
\end{prop}

\begin{proof}
If $p\in V(\CA,\kk)$ is not deep, then it is contained in a cluster torus, and so there is an open neighborhood of $p$ which is non-singular.
\end{proof}

Deep points can be translated to the algebraic side. 
Define the \textbf{deep ideal} $\D(\CA)$ of a cluster algebra $\CA$ to be the ideal generated by the products of the elements in each cluster; 
that is,
\[ \D(\CA) \coloneqq \sum_{\text{clusters }\{x_1,x_2,\dotsc ,x_n\}} \CA x_1x_2 \dotsm x_n\]
As the following proposition shows, the deep locus is defined by the vanishing of this ideal.

\begin{prop}\label{prop: deepideal}
A point $p\in V(\CA,\kk)$ is deep if and only if $p(\D(\CA))=0$. Therefore, the deep locus of $V(\CA,\kk)$ is the Zariski closed subset defined by the vanishing of the deep ideal $\D(\CA)$.

\end{prop}
\begin{proof}
The first claim follows from equivalences between the the following statements.
%
%
%
\begin{enumerate}
    \item $p$ is deep; that is, for each cluster $\{x_1,x_2, \dotsc ,x_n\}$, there exists $x_i$ such that $p(x_i)=0$.
    \item For each cluster $\{x_1,x_2, \dotsc ,x_n\}$, $p(x_1x_2\dotsm x_n)=0$.
    \item $p(\D(\CA))=0$.
\end{enumerate}
These all follow from basic properties of homomorphisms; note that $(2)\Rightarrow (1)$ uses that $\kk$ is a field.
The second claim follows from the definition of the Zariski topology.
\end{proof}

A cluster algebra may have deep points over some fields but not others (e.g.~Example \ref{ex: type 1 2 deep points}), so a lack of deep points over a specific field generally doesn't say much about the algebra. Rather, it is the absence of deep points over all fields which will have the strongest algebraic consequences.

\begin{prop}\label{prop: nodeep}
Given a cluster algebra $\CA$, 
the following are equivalent.
\begin{enumerate}
    \item The deep ideal $\D(\CA)$ of $\CA$ is trivial; that is, $1\in \D(\CA)$.
    \item $V(\CA,\kk)$ has no deep points, for all fields $\kk$.
    \item $V(\CA,\kk)$ is covered by finitely many cluster tori, for all fields $\kk$.
\end{enumerate}

\end{prop}
\begin{proof}
If $\D(\CA)$ is non-trivial, then it is contained is some maximal ideal $\mathfrak{m}$. The quotient map
$ \CA\rightarrow \CA/\mathfrak{m} $
is a deep point in $V(\CA,\CA/\mathfrak{m})$, and so $(2)\Rightarrow (1)$.

Conversely, if the deep ideal $\D(\CA)$ is trivial, then we may write $1$ as
\[ 1 = \sum_{\{x_1,x_2,\dotsc ,x_n\}\in S} a_{\{x_1,x_2,\dotsc ,x_n\}} x_1x_2 \dotsm x_n\]
where each $a_{\{x_1,x_2,\dotsc ,x_n\}}\in \CA$ and $S$ is a finite set of clusters.
Applying $p\in V(\CA,\kk)$ to this equality,
\[ 1 =\sum_{\{x_1,x_2,\dotsc ,x_n\}\in S} p(a_{\{x_1,x_2,\dotsc ,x_n\}} )p(x_1)p(x_2) \dotsm p(x_n)\]
Since the right-hand side cannot vanish, there must be some cluster in $\{x_1,x_2,\dotsc ,x_n\}\in  S$ such that $p(x_1)p(x_2)\dotsm p(x_n)\neq0$, and so the point $p$ is in the cluster torus corresponding to $\{x_1,x_2,\dotsc ,x_n\}$. Therefore, the cluster tori indexed by $S$ cover $V(\CA,\kk)$, so $(1)\Rightarrow (3)$. Finally, $(3)\Rightarrow (2)$ is clear.
\end{proof}


\begin{thm}\label{thm: properties no deep points}
Let $\CA$ be a cluster algebra with no deep points for all $\kk$ (or equivalently, the deep ideal $\D(\CA)$ is trivial). Then

\begin{enumerate}
    \item $\CA$ is non-singular.
    \item $\CA$ is finitely generated.
    \item $\CA$ is equal to its upper cluster algebra.
    \item $\CA$ is a locally acyclic cluster algebra.
\end{enumerate}
\end{thm}

\begin{proof}
The non-singularity of $\CA$ (or equivalently, of $V(\CA,\kk)$ for all $\kk$) follows from Proposition \ref{prop: deepsingular}.

In the terminology of \cite{MulLA}, the localization 
$ \CA\rightarrow \CA[x_1^{-1},x_2^{-1},\dotsc ,x_n^{-1}]$
is the \emph{cluster localization} corresponding to freezing an entire cluster; the corresponding cluster has no mutable variables and is therefore vacuously acyclic. If the deep ideal of $\CA$ is trivial, then finitely many of these cluster localizations cover the cluster variety (Proposition \ref{prop: nodeep}) and so $\CA$ is locally acyclic. Local acyclicity implies finite generation and $\mathcal{A}=\mathcal{U}$ by \cite{MulLA}.
\end{proof}

\subsection{Deep points and upper cluster algebras}

The story up to this point can be repeated for 
the upper cluster algebra $\mathcal{U}$ of a cluster algebra $\mathcal{A}$. The \textbf{upper cluster variety} is the set 
\[ V(\mathcal{U},\kk) \coloneqq \Hom_{\mathrm{Ring}}(\mathcal{U},\kk) \coloneqq \{\text{ring homomorphisms $p:\mathcal{U}\rightarrow \kk$}\} \]
endowed with the Zariski topology. Each cluster $\{x_1,x_2,\dotsc ,x_n\}$ determines an open inclusion 
\[
(\kk^\times)^n \simeq \Hom_{\mathrm{Ring}}(\mathbb{Z} [x_1^{\pm1} ,x_2^{\pm1},\dotsc ,x_n^{\pm1}],\kk)\hookrightarrow \Hom_{\mathrm{Ring}}(\mathcal{U},\kk) \eqqcolon V(\mathcal{U},\kk)
\]
whose image we call a \textbf{cluster torus} in $V(\mathcal{U},\kk)$. The points in $V(\mathcal{U},\kk)$ which are not in any cluster torus are the \textbf{deep points}, which collectively form the \textbf{deep locus}, which is defined by the vanishing of a \textbf{deep ideal} $\D(\mathcal{U})$ in $\mathcal{U}$.

%
%

We can relate this back to the cluster algebra $\CA$ using the inclusion $\CA\hookrightarrow \mathcal{U}$. Applying the functor $\mathrm{Hom}_{\text{Ring}}(-,\kk)$ gives a morphism of varieties $V(\mathcal{U},\kk)\rightarrow V(\CA,\kk)$.
%
%

\begin{prop}
The map $V(\mathcal{U},\kk)\rightarrow V(\CA,\kk)$ restricts to an isomorphism on each cluster torus, and therefore restricts to an isomorphism between the complements of the deep loci.
%
\end{prop}

\begin{proof}
The localizations of $\CA$ and $\mathcal{U}$ at a cluster equal the Laurent ring, so they are equal.
\[ \mathcal{U}[x_1^{-1},x_2^{-1},\dotsc ,x_n^{-1}] = \mathbb{Z}[x_1^{\pm1},x_2^{\pm1},\dotsc ,x_n^{\pm1}] = \CA[x_1^{-1},x_2^{-1},\dotsc ,x_n^{-1}] \]
Therefore, the map $V(\mathcal{U},\kk)\rightarrow V(\CA,\kk)$ restricts to an isomorphism between the cluster tori
\[ 
V(\mathcal{U}[x_1^{-1},x_2^{-1},\dotsc ,x_n^{-1}],\kk) \xrightarrow{\sim} 
V(\CA[x_1^{-1},x_2^{-1},\dotsc ,x_n^{-1}],\kk) 
\]
Since this holds for each cluster, the morphism restricts to an isomorphism between the union of the cluster tori, which are the complements of the deep loci.
%
\end{proof}

The proposition tells us that any difference between a cluster algebra $\CA$ and its upper cluster algebra $\mathcal{U}$ corresponds to a difference between the deep loci. An example of this is the Markov cluster algebra (Section \ref{section: Markov}), where the map
$ V(\mathcal{U},\kk)\rightarrow V(\CA,\kk) $
collapses a line in the deep locus of $V(\mathcal{U},\kk)$ to a single deep point in $V(\CA,\kk)$, and is otherwise an isomorphism (see Remark \ref{rem: Markov Upper line to point}).

\begin{rem}
More generally, one can consider any \emph{intermediate cluster algebra} $\CA\subseteq R\subseteq \mathcal{U}$; examples include marked skein algebras \cite{Mul16} and the span of the theta functions \cite{GHKK}. We may define deep points in the intermediate cluster variety $V(R,\kk)$, and we have morphisms
\[ V(\mathcal{U},\kk)\rightarrow V(R,\kk) \rightarrow V(\CA,\kk) \]
which are isomorphisms away from the respective deep loci. 
\end{rem}

\draftnewpage

\section{Deep points of cluster algebras of polygons}\label{section: polygons}

We now turn to describing the deep locus of a number of specific cluster algebras, starting with cluster algebras of polygons.



\subsection{The cluster algebra(s) of a polygon}

We start with an algebra which is not quite a cluster algebra (by our definition), but can be made into a cluster algebra in two closely related ways.

For any integer $n\geq3$, define the \textbf{$(2,n)$-th Pl\"ucker algebra} as follows:
\[ P_{2,n}\coloneqq\mathbb{Z}[x_{i,j}, \forall 1\leq i<j\leq n] /\langle x_{i,k}x_{j,l}-(x_{i,j}x_{k,l}+x_{i,l}x_{j,k}), \forall1\leq i<j<k<l\leq n\rangle \]
%
The generators and relations of this algebra may be visualized as follows.
Let $\Delta_n$ be a convex $n$-gon with the vertices indexed clockwise from 1 to $n$. For each pair $1\leq i<j\leq n$, we identify the generator $x_{i,j}$ with the line segment between vertices $i$ and $j$. Inspired by this realization, we call $x_{i,j}$ an \textbf{edge} if $j=i+1$ or $(i,j)=(1,n)$; otherwise, $x_{i,j}$ is a \textbf{diagonal}. 
The relations in the algebra are parameterized by pairs of crossing diagonals, and can be visualized as in Figure \ref{fig: Ptolemy cluster variables}.\footnote{These relations are a special case of the \emph{Pl\"ucker relations} of a Grassmannian, the \emph{Ptolemy relations} in convex geometry, and the \emph{Kauffman skein relations} in knot theory.}

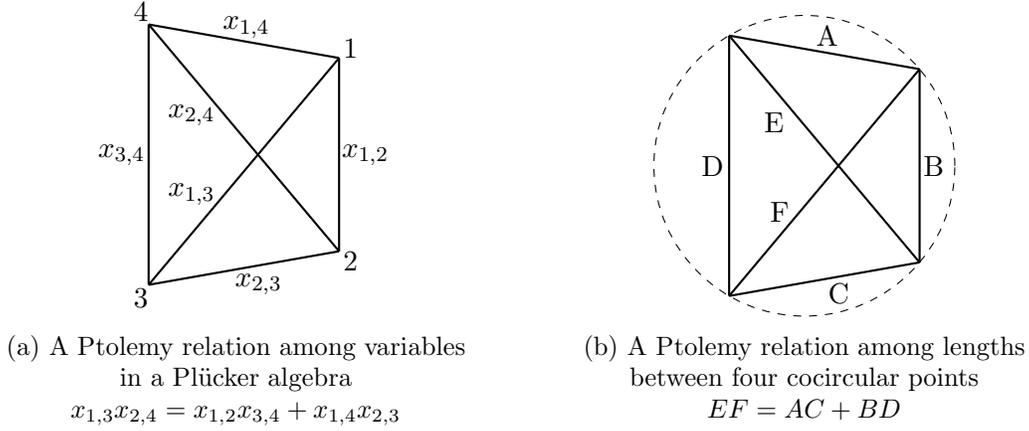
\begin{figure}[h!t]
\captionsetup[subfigure]{justification=centering}
\begin{subfigure}[b]{0.45\textwidth}
\centering
\begin{tikzpicture}[scale=2]
\draw[draw=none] (0,0) circle (1);
\node at (40:1.1) {1};
\node at (320:1.1) {2};
\node at (240:1.1) {3};
\node at (120:1.1) {4};
\draw[thick] (40:1)--(120:1) {}; 
\draw[thick] (120:1)--(240:1) {}; 
\draw[thick] (240:1)--(320:1) {}; 
\draw[thick] (320:1)--(40:1) {}; 
\draw[thick] (40:1)--(240:1) {}; 
\draw[thick] (320:1)--(120:1) {};
\node at (80:0.86) {$x_{1,4}$};
\node at (180:0.68) {$x_{3,4}$};
\node at (285:0.89) {$x_{2,3}$};
\node at (0:0.94) {$x_{1,2}$};
\node at (129:0.34) {$x_{2,4}$};
\node at (231:0.34) {$x_{1,3}$};
\end{tikzpicture}
\subcaption{
A Ptolemy relation among variables\\in a Pl\"ucker algebra
\\$x_{1,3}x_{2,4}=x_{1,2}x_{3,4}+x_{1,4}x_{2,3}$
}
\label{fig: Ptolemy cluster variables}
\end{subfigure}
\begin{subfigure}[b]{0.45\textwidth}
\centering
\begin{tikzpicture}[scale=2]
\draw[thin,dashed] (0,0) circle (1);
\draw[thick] (40:1)--(120:1) {}; 
\draw[thick] (120:1)--(240:1) {}; 
\draw[thick] (240:1)--(320:1) {}; 
\draw[thick] (320:1)--(40:1) {}; 
\draw[thick] (40:1)--(240:1) {}; 
\draw[thick] (320:1)--(120:1) {};
\node at (80:0.88) {A};
\node at (180:0.61) {D};
\node at (285:0.88) {C};
\node at (0:0.86) {B};
\node at (125:0.35) {E};
\node at (242:0.35) {F};
\end{tikzpicture}
\subcaption{
A Ptolemy relation among lengths\\between four cocircular points
\\$EF=AC+BD$
}
\label{fig: Ptolemy cocircular}
\end{subfigure}

\caption{The Ptolemy relations}
\label{fig: Ptolemy}
\end{figure}

The algebra $P_{2,n}$ is \emph{almost} a cluster algebra of type $A_{n-3}$. The issue lies in the edge elements; they behave like frozen variables, but our definition of cluster algebra requires that frozen variables have inverses. There are two natural ways to resolve this, both of which will be of interest.

\begin{itemize}
    \item The \textbf{cluster algebra of the $n$-gon (with boundary coefficients)} is the localization
    \[ \CA(\Delta_n)\coloneqq P_{2,n}[x_{i,j}^{-1}, \forall \text{ edges }x_{i,j}] \]
    \item The \textbf{cluster algebra of type $A_{n-3}$ (with no frozens)} is the quotient
    \[ \CA(A_{n-3}) \coloneqq P_{2,n}/ \langle x_{i,j}-1, \forall \text{ edges }x_{i,j} \rangle \]
\end{itemize}  
These two algebras are related by the quotient isomorphism
\[ \CA(\Delta_n)/ \langle x_{i,j}-1, \forall \text{ edges }x_{i,j} \rangle \simeq \CA(A_{n-3})\]
As the names imply, these algebras are both cluster algebras.

\begin{thm}[{\cite[Section 12]{FZ03}}]
The algebras $\CA(\Delta_n)$ and $\CA(A_{n-3})$ are both cluster algebras of type $A_{n-3}$. In both, the mutable cluster variables are the diagonals in $\Delta_n$ and the clusters are the triangulations of $\Delta_n$. The frozen variables in $\CA(\Delta_n)$ are the edges, while $\CA(A_{n-3})$ has no frozens.
\end{thm}

Since these cluster algebras are defined in terms of generators and relations, the associated cluster variety may be characterized in terms of these presentations.

\begin{coro}\label{coro: polygonvalues}
%
%
%
A point $p\in V(\CA(\Delta_n),\kk)$ (respectively, $V(\CA(A_{n-3}),\kk)$) is equivalent to a map
\[ p:\{\text{diagonals and edges in $\Delta_n$}\} \rightarrow \kk \]
such that the Ptolemy relations hold and edges are not sent to $0$ (respectively, edges are sent to $1$). Such a point is deep iff every triangulation of $\Delta_n$ contains a diagonal which is sent to $0$.
\end{coro}


\begin{rem}
If we choose a specific polygon $\Delta_n$ in the plane whose vertices lie on a common circle as in Figure \ref{fig: Ptolemy cocircular} (called a \emph{cocircular $n$-gon}), then sending each arc in $\Delta_n$ to its length satisfies the Ptolemy relations by \emph{Ptolemy's Theorem}. In this way, we may identify the set of cocircular $n$-gons up to congruence with a proper subset of the cluster variety $V(\CA(\Delta_n),\kk)$. 
\end{rem}

\begin{rem}
If we choose a $2\times n$ matrix $\mathsf{M}$ such that cyclically adjacent columns are linearly independent, then sending each $x_{i,j}$ to the determinant of the $i$th and $j$th columns of $\mathsf{M}$ satisfies the Ptolemy relations by the \emph{Pl\"ucker embedding}. In this way, we may identify the set of such matrices (up to left multiplication by $\mathrm{SL}(2,\kk)$) with the cluster variety $V(\CA(\Delta_n),\kk)$.
\end{rem}

\subsection{Vanishing lemmas}

By a \textbf{triangle} in $\Delta_n$, we mean a triple of the form $\{x_{i,j},x_{i,k},x_{j,k}\}$; that is, the line segments bounding a triangle of vertices in $\Delta_n$.
The following lemma demonstrates that a deep point of $\CA(\Delta_n)$ must kill an odd number of diagonals in each triangle in $\Delta_n$.





\begin{lemma}\label{lemma: trianglepoly}
Let $p\in V(\CA(\Delta_n),\kk)$ and let $\{x_{i,j},x_{i,k},x_{j,k}\}$ be a triangle in $\Delta_n$.
\begin{enumerate}
    \item If $p$ kills any two of $\{x_{i,j},x_{i,k},x_{j,k}\}$, then $p$ also kills the third.
    \item If $p$ is deep, then $p$ kills at least one of $\{x_{i,j},x_{i,k},x_{j,k}\}$.
\end{enumerate}
\end{lemma}

\begin{proof}
(1) Assume that $p(x_{i,j})=p(x_{i,k})=0$. Since $p$ cannot vanish on an edge, $x_{i,j}$ cannot be an edge and so $j\geq i+2$. 
%
Applying $p$ to the Ptolemy relation of the quadrilateral with vertices $i,i+1,j,k$ and using the fact that $p(x_{i,i+1})\neq0$,
\[ p(x_{j,k}) = \frac{p(x_{i,j})p(x_{i+1,k})-p(x_{i,k})p(x_{i+1,k})}{p(x_{i,i+1})} = 0 \]
An identical argument applies if $p$ kills any other pair in $\{x_{i,j},x_{i,k},x_{j,k}\}$.

(2) Assume, for contradiction, that $p$ is deep and there is some triangle in $\Delta_n$ on which it does not vanish. Therefore, the set of triangulations of $\Delta_n$ which contain a triangle on which $p$ does not vanish is non-empty. Among all such triangulations, choose a triangulation $S$ which contains the minimal possible number of diagonals killed by $p$.

Since $p$ kills diagonals in some triangles in $S$ but not others, there must be an adjacent pair of triangles in $S$ such that $p$ kills a diagonal in one triangle and no diagonals in the other. 
We focus on the pentagon consisting of these two triangles and the triangle on the other side of the diagonal killed by $p$.
Flipping and rotating indices in $\Delta_n$ as necessary,
we may choose indices $i<j<k<l<m$ for the pentagon so that $\overline{ijk}$ is the triangle on which $p$ is does not vanish, $\overline{ikl}$ is the adjacent triangle in $S$ containing a diagonal $x_{i,l}$ killed by $p$, and $\overline{ilm}$ is the other triangle in $S$ containing $x_{i,l}$ (Figure \ref{fig: first triangulation poly}).

By assumption, $p(x_{i,j}),p(x_{i,k})$, and $p(x_{j,k})$ are non-zero and $p(x_{i,l})=0$. 
We repeatedly apply part (1) to triangles in this pentagon.
\begin{itemize}
    \item Since $p(x_{i,j})\neq0$ and $p(x_{i,l})=0$, $p(x_{j,l})\neq0$.
    \item Since $p(x_{i,k})\neq0$ and $p(x_{i,l})=0$, $p(x_{k,l})\neq0$.
    \item Since $p(x_{j,k})\neq0$, at most one of $p(x_{j,m})$ and $p(x_{k,m})$ can be $0$.
\end{itemize}
Therefore, either $p(x_{j,l})$ and $p(x_{j,m})$ are non-zero, or $p(x_{i,k})$ and $p(x_{k,m})$ are non-zero (Figure \ref{fig: second triangulation poly}).

As a consequence, replacing $\{x_{i,k},x_{i,l}\}$ in $S$ with $\{x_{j,l},x_{j,m}\}$ or $\{x_{i,k},x_{k,m}\}$ produces a triangulation of $\Delta_n$ with strictly fewer arcs killed by $p$ and at least one triangle on which $p$ does not vanish. This contradicts the minimality of $S$.
\end{proof}

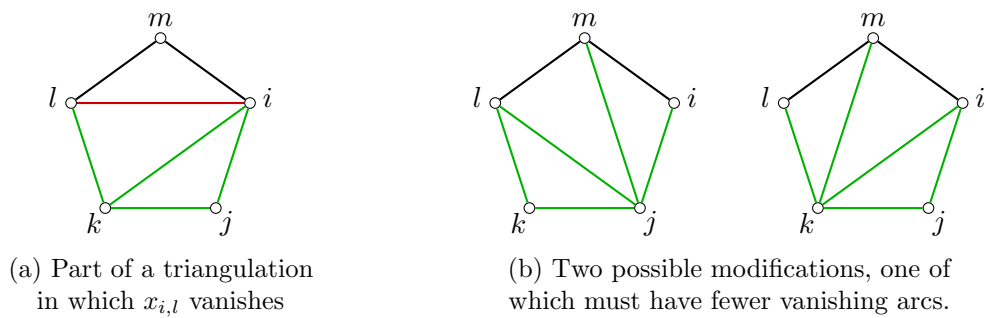
\begin{figure}[h!tb]
\captionsetup[subfigure]{justification=centering}
\centering
\begin{subfigure}[c]{0.35\linewidth}
\centering
\begin{tikzpicture}[inner sep=0.5mm,scale=1.25,auto]
	\node (1) at (90:1) [circle,draw] {};
	\node (2) at (90-72:1) [circle,draw] {};
	\node (3) at (90-2*72:1) [circle,draw] {};
	\node (4) at (90-3*72:1) [circle,draw] {};
	\node (5) at (90-4*72:1) [circle,draw] {};
	\node at (90:1.2) {$m$};
	\node at (90-72:1.2) {$i$};
	\node at (90-2*72:1.2) {$j$};
	\node at (90-3*72:1.2) {$k$};
	\node at (90-4*72:1.2) {$l$};
	\draw[dark green,thick] (2) to (3);
	\draw[dark green,thick] (3) to (4);
	\draw[dark green,thick] (2) to (4);
	\draw[dark green,thick] (4) to (5);
	\draw[dark red,thick] (2) to (5);
	\draw[thick] (5) to (1);
	\draw[thick] (1) to (2);
\end{tikzpicture}
\subcaption{Part of a triangulation\\in which $x_{i,l}$ vanishes}
\label{fig: first triangulation poly}
\end{subfigure}
\begin{subfigure}[c]{0.55\linewidth}
\centering
\begin{tikzpicture}[inner sep=0.5mm,scale=1.25,auto]
	\node (1) at (90:1) [circle,draw] {};
	\node (2) at (90-72:1) [circle,draw] {};
	\node (3) at (90-2*72:1) [circle,draw] {};
	\node (4) at (90-3*72:1) [circle,draw] {};
	\node (5) at (90-4*72:1) [circle,draw] {};
	\node at (90:1.2) {$m$};
	\node at (90-72:1.2) {$i$};
	\node at (90-2*72:1.2) {$j$};
	\node at (90-3*72:1.2) {$k$};
	\node at (90-4*72:1.2) {$l$};
	\draw[dark green,thick] (2) to (3);
	\draw[dark green,thick] (3) to (4);
	\draw[dark green,thick] (4) to (5);
	\draw[thick] (5) to (1);
	\draw[thick] (1) to (2);
	\draw[dark green,thick] (3) to (5);
	\draw[dark green,thick] (3) to (1);
\end{tikzpicture}
\hspace{.5cm}
\begin{tikzpicture}[inner sep=0.5mm,scale=1.25,auto]
	\node (1) at (90:1) [circle,draw] {};
	\node (2) at (90-72:1) [circle,draw] {};
	\node (3) at (90-2*72:1) [circle,draw] {};
	\node (4) at (90-3*72:1) [circle,draw] {};
	\node (5) at (90-4*72:1) [circle,draw] {};
	\node at (90:1.2) {$m$};
	\node at (90-72:1.2) {$i$};
	\node at (90-2*72:1.2) {$j$};
	\node at (90-3*72:1.2) {$k$};
	\node at (90-4*72:1.2) {$l$};
	\draw[dark green,thick] (2) to (3);
	\draw[dark green,thick] (3) to (4);
	\draw[dark green,thick] (4) to (5);
	\draw[thick] (5) to (1);
	\draw[thick] (1) to (2);
	\draw[dark green,thick] (4) to (1);
	\draw[dark green,thick] (4) to (2);
\end{tikzpicture}
\subcaption{Two possible modifications, one of\\which must have fewer vanishing arcs.}
\label{fig: second triangulation poly}
\end{subfigure}
\caption{Modifying a triangulation to reduce the number of vanishing arcs.}
\label{fig: triangulation modification poly}
\end{figure}


As a consequence, deep points of $\CA(\Delta_n)$ must kill a specific set of diagonals.

\begin{lemma}\label{lemma: evenrule}
Let $p\in V(\CA(\Delta_n),\kk)$. Then $p$ is deep iff, for all $i,j$,
\begin{equation}\label{eq: evenrule}
p(x_{i,j})=0\text{ iff $j-i$ is even}
\end{equation}
\end{lemma}

\begin{proof}
Since every triangulation of $\Delta_n$ contains a diagonal $x_{i,j}$ with $j-i$ even, a point satisfying \eqref{eq: evenrule} is deep. 
Next, we prove by induction on $j-i$ that a deep point $p$ must satisfy \eqref{eq: evenrule}.
If $j-i=1$, then $x_{i,j}$ is an edge of $\Delta_n$ and so $p(x_{i,j})\neq0$. If $j-i>1$, 
apply Lemma \ref{lemma: trianglepoly} to the triangle $\{x_{i,j-1},x_{i,j},x_{j-1,j}\}$. 
Since $x_{j-1,j}$ is an edge, $p(x_{j-1,j})\neq0$ and so exactly one of $p(x_{i,j-1})$ and $p(x_{i,j})$ is 0. If \eqref{eq: evenrule} holds for $p(x_{i,j-1})$, then it holds for $p(x_{i,j})$, completing the induction.
\end{proof}

\subsection{Deep points}

We can now classify the deep points of both $\CA(\Delta_n)$ and $\CA(A_n)$.

\begin{thm}\label{thm: deeppolygon}
Let $n \geq 3$.
\begin{enumerate}
    \item If $n$ is odd, then the cluster variety $V(\CA(\Delta_n),\kk)$ has no deep points.
    \item If $n$ is even, then a choice of values on the edges of $\Delta_n$
    \[ p:\{\text{edges in $\Delta_n$}\} \rightarrow \kk^\times\]
    extends to a deep point of $V(\CA(\Delta_n),\kk)$ iff
    \begin{equation}\label{eq: alternating product}
    \frac{p(x_{1,2})p(x_{3,4}) \dotsm p(x_{n-1,n})}{p(x_{1,n})p(x_{2,3}) \dotsm p(x_{n-2,n-1})} = (-1)^{\frac{n+2}{2}}
    \end{equation}
%
    and this extension is unique, with values on an arbitrary diagonal $x_{i,j}$ given by
    \begin{equation}\label{eq: arbitrary diagonal}
        p(x_{i,j}) = \left\{\begin{array}{cc}
        (-1)^{\frac{j-i-1}{2}} \frac{p(x_{i,i+1})p(x_{i+2,i+3})\dotsm p(x_{j-1,j})}{p(x_{i+1,i+2})\dotsm p(x_{j-2,j-1})} & \text{if $i \not\equiv j\bmod{2}$} \\
        0 & \text{if $i \equiv j\bmod{2}$}
        \end{array}\right\}
    \end{equation}   
    As a consequence, the deep locus of $V(\CA(\Delta_n),\kk)$ freely determined by its values on all but one edge, and is therefore isomorphic to $(\kk^\times)^{n-1}$.
\end{enumerate}
\end{thm}

\begin{proof}
Fix a deep point $p\in V(\CA(\Delta_n),\kk)$ and consider a number of cases.
%

If $n$ is odd, then $n-1$ is even, and so Lemma \ref{lemma: evenrule} forces $p(x_{1,n})=0$. However, $x_{1,n}$ is an edge in $\Delta_n$, and so this is impossible. Therefore, $V(\CA(\Delta_n),\kk)$ has no deep points when $n$ is odd.

If $n=4$, the two triangulations of $\Delta_4$ each contain a unique diagonal, and so $p$ is deep iff $p(x_{1,3})=p(x_{2,4})=0$. A deep point of $\CA(\Delta_4)$ is then determined by its non-zero values on the edges, subject to the Ptolemy relation $0= p(x_{1,2}) p(x_{3,4}) + p(x_{2,3}) p(x_{1,4})$. Rewriting yields  \eqref{eq: alternating product}:
\[ \frac{p(x_{1,2}) p(x_{3,4})}{p(x_{2,3}) p(x_{1,4})} = -1 = (-1)^{\frac{4+2}{2}}. \]

If $n>4$ is even, assume (for induction) that the theorem holds for all $\Delta_{n'}$ with $n'<n$, and consider a deep point $p\in V(\CA(\Delta_n),\kk)$. 
Lemma \ref{lemma: evenrule} forces $p(x_{i,j})=0$ iff $j-i$ is even. 
Consider $x_{i,j}$ with $j-i$ odd, and let $\Delta_{i,j}$ be the convex hull of the vertices $i,i+1,\dotsc,j-1,j$. The restriction of $p$ to edges and diagonals in $\Delta_{i,j}$ defines a point in $V(\CA(\Delta_{i,j}),\kk)$; furthermore, since $p$ kills at least one diagonal in every triangle (Lemma \ref{lemma: trianglepoly}), this restriction is deep. By the inductive hypothesis on the $(j-i-1)$-gon $\Delta_{i,j}$, we know that $p$ satisfies Condition \eqref{eq: alternating product} around the edges of $\Delta_{i,j}$:
\[ 
\frac{p(x_{i,i+1})p(x_{i+2,i+3}) \dotsm p(x_{j-1,j})}{p(x_{i,j})p(x_{i+1,i+2}) \dotsm p(x_{j-2,j-1})} = (-1)^{\frac{j-i+1}{2}}
\]
Multiplying both sides by $(-1)^{\frac{j-i-1}{2}}p(x_{i,j})$ gives Formula \eqref{eq: arbitrary diagonal} for $p(x_{i,j})$ when $j-i$ is odd. Since the values on the edges determines the value on each diagonal, $p$ is uniquely determined by its values on the edges.


To verify Condition \eqref{eq: alternating product} for $\Delta_n$, consider the convex hull $\Delta_{j,i}$ of the vertices $j,j+1,\dotsc,n,1,\dotsc,i$. By the same argument as before, $p$ restricts to a deep point in $V(\CA(\Delta_{j,i}),\kk)$. By another application of the inductive hypothesis, $p$ satisfies Condition \eqref{eq: alternating product} around the edges of $\Delta_{j,i}$:
\[ 
\frac{p(x_{j,j+1})p(x_{j+2,j+3}) \dotsm p(x_{i-1,i})}{p(x_{i,j})p(x_{j+1,j+2}) \dotsm p(x_{i-2,i-1})} = (-1)^{\frac{n-i+j+1}{2}}
\]
Multiplying both sides by $(-1)^{\frac{n-i+j+1}{2}}p(x_{i,j})$ gives another formula for $p(x_{i,j})$. Since the two formulas must be equal,
\[ 
(-1)^{\frac{j-i+1}{2}}\frac{p(x_{i,i+1})p(x_{i+2,i+3}) \dotsm p(x_{j-1,j})}{p(x_{i+1,i+2}) \dotsm p(x_{j-2,j-1})} 
=
(-1)^{\frac{n-i+j+1}{2}}\frac{p(x_{j,j+1})p(x_{j+2,j+3}) \dotsm p(x_{i-1,i})}{p(x_{j+1,j+2}) \dotsm p(x_{i-2,i-1})}
\]
Moving the values of $p$ to the left and the signs to the right proves that $p$ satisfies Condition \eqref{eq: alternating product} for the edges in the boundary of $\Delta_n$, completing the induction.
\end{proof}

\begin{ex}
Consider the hexagon with values $a_1,a_2,\dotsc,a_6$ on the six edges.
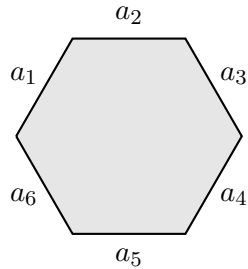
\[
\begin{tikzpicture}[scale=.75, xscale=-1]
    \draw[fill=black!10,thick] 
    (0*60:2)
    to (1*60:2)
    to (2*60:2)
    to (3*60:2)
    to (4*60:2)
    to (5*60:2)
    to (6*60:2)
    ;
    \node (a1) at (30+0*60:2.15) {$a_1$};
    \node (a2) at (30+1*60:2.15) {$a_2$};
    \node (a3) at (30+2*60:2.15) {$a_3$};
    \node (a4) at (30+3*60:2.15) {$a_4$};
    \node (a5) at (30+4*60:2.15) {$a_5$};
    \node (a6) at (30+5*60:2.15) {$a_6$};
\end{tikzpicture}
\]
By Theorem \ref{thm: deeppolygon}, these values extend to a (necessarily unique) deep point if and only if
\begin{equation}\label{eq: hexagon}
\frac{a_1a_3a_5}{a_2a_4a_6} = 1
\end{equation}
Assuming this holds, the values at this deep point on the diagonals are depicted in Figure \ref{fig: deephexagon}.

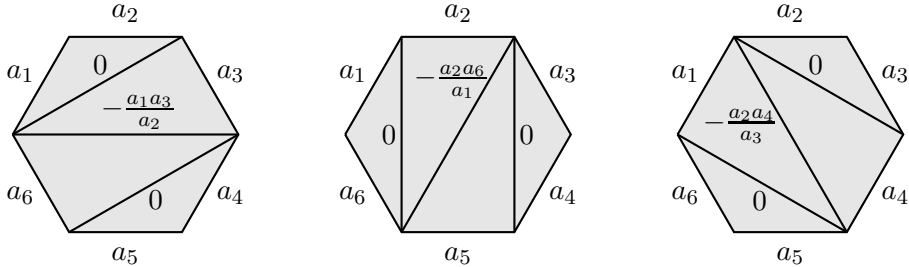
\begin{figure}[h!t]
\[
\begin{tikzpicture}[xscale=-1,scale=.75]
    \draw[fill=black!10,thick] 
    (0*60:2)
    to (1*60:2)
    to (2*60:2)
    to (3*60:2)
    to (4*60:2)
    to (5*60:2)
    to (6*60:2)
    ;
    \node (a1) at (30+0*60:2.15) {$a_1$};
    \node (a2) at (30+1*60:2.15) {$a_2$};
    \node (a3) at (30+2*60:2.15) {$a_3$};
    \node (a4) at (30+3*60:2.15) {$a_4$};
    \node (a5) at (30+4*60:2.15) {$a_5$};
    \node (a6) at (30+5*60:2.15) {$a_6$};
    \draw[thick] (5*60:2) to (3*60:2);
    \node at (245:1.27) {$0$};
    \draw[thick] (3*60:2) to (6*60:2);
    \node at (120:0.46) {$-\frac{a_1a_3}{a_2}$};
    \draw[thick] (6*60:2) to (2*60:2);
    \node at (70:1.31) {$0$};
\end{tikzpicture}
\hspace{1cm}
\begin{tikzpicture}[xscale=-1,scale=.75]
    \draw[fill=black!10,thick] 
    (0*60:2)
    to (1*60:2)
    to (2*60:2)
    to (3*60:2)
    to (4*60:2)
    to (5*60:2)
    to (6*60:2)
    ;
    \node (a1) at (30+0*60:2.15) {$a_1$};
    \node (a2) at (30+1*60:2.15) {$a_2$};
    \node (a3) at (30+2*60:2.15) {$a_3$};
    \node (a4) at (30+3*60:2.15) {$a_4$};
    \node (a5) at (30+4*60:2.15) {$a_5$};
    \node (a6) at (30+5*60:2.15) {$a_6$};
    \draw[thick] (4*60:2) to (2*60:2);
    \node at (180:1.22) {$0$};
    \draw[thick] (2*60:2) to (5*60:2);
    \node at (83:0.95) {$-\frac{a_2a_6}{a_1}$};
    \draw[thick] (5*60:2) to (1*60:2);
    \node at (0:1.22) {$0$};
\end{tikzpicture}
\hspace{1cm}
\begin{tikzpicture}[xscale=-1,scale=.75]
    \draw[fill=black!10,thick] 
    (0*60:2)
    to (1*60:2)
    to (2*60:2)
    to (3*60:2)
    to (4*60:2)
    to (5*60:2)
    to (6*60:2)
    ;
    \node (a1) at (30+0*60:2.15) {$a_1$};
    \node (a2) at (30+1*60:2.15) {$a_2$};
    \node (a3) at (30+2*60:2.15) {$a_3$};
    \node (a4) at (30+3*60:2.15) {$a_4$};
    \node (a5) at (30+4*60:2.15) {$a_5$};
    \node (a6) at (30+5*60:2.15) {$a_6$};
    \draw[thick] (3*60:2) to (1*60:2);
    \node at (110:1.31) {$0$};
    \draw[thick] (1*60:2) to (4*60:2);
    \node at (10:0.9) {$-\frac{a_2a_4}{a_3}$};
    \draw[thick] (4*60:2) to (0*60:2);
    \node at (295:1.29) {$0$};
\end{tikzpicture}
\]
\caption{A general deep point in the cluster algebra of a hexagon}
\label{fig: deephexagon}
\end{figure}

Note that Condition \eqref{eq: hexagon} implies that $-\frac{a_1a_3}{a_2} = -\frac{a_4a_6}{a_5}$. More generally, every diagonal is defined by two equivalent formulas corresponding to the edges on either side. 
Setting all boundary values $a_i=1$ gives the unique deep point in the coefficient-free cluster algebra of type $A_3$ (Figure \ref{fig: deepA3}).\qedhere
\end{ex}

A key tool of the proof and key takeaway from the proposition is that a point $p$ of $\CA(\Delta_n)$ is deep iff, for any diagonal which cuts $\Delta_n$ into two even-sided polygons, the restrictions of $p$ to each of those polygons are deep points in the respective cluster varieties. Stated precisely:

\begin{coro}\label{coro: polygoncut}
Let $x_{i,j}$ be a diagonal in $\Delta_n$ with $j-i$ odd, and let $\Delta_{i,j}$ and $\Delta_{j,i}$ be the convex hulls of $\{i,i+1,\dotsc ,j-1,j\}$ and $\{j,j+1,\dotsc ,n,1,\dotsc ,i-1,i\}$, respectively. Then $p\in V(\CA(\Delta_n),\kk)$ is deep iff the restrictions of $p$ to $\Delta_{i,j}$ and $\Delta_{j,i}$ define deep points in $V(\CA(\Delta_{i,j}),\kk)$ and $V(\CA(\Delta_{j,i}),\kk)$.
\end{coro}

\begin{rem}
As an application, the non-zero values of Formula \eqref{eq: arbitrary diagonal} follow from Condition \eqref{eq: alternating product} for $\Delta_{i,j}$, and Condition \eqref{eq: alternating product} for $\Delta_n$ follows from Condition \eqref{eq: alternating product} for $\Delta_{i,j}$ and $\Delta_{j,i}$.
\end{rem}

Theorem \ref{thm: deeppolygon} gives a classification of deep points in coefficient-free cluster algebras of type $A$, by specializing the values of the edges to $1$ and
shifting the indexing by $3$.

\begin{thm}\label{thm: deepAn}
If $n\equiv 3\bmod{4}$ (or $n\equiv 1\bmod{4}$ and $\operatorname{char}(\kk)=2$), then the cluster variety $V(\CA(A_{n}),\kk)$ has a unique deep point whose values on the diagonals in $\Delta_{n+3}$ are given by
\begin{equation}\label{eq: polygondiagonal}
	p(x_{i,j}) = \left\{\begin{array}{cc}
	(-1)^{\frac{j-i-1}{2}} & \text{if $i \not\equiv j\bmod{2}$} \\
	0 & \text{if $i \equiv j\bmod{2}$}
	\end{array}\right\}
\end{equation}  
For all other choices of $n$ and $\kk$, the cluster variety $V(\CA(A_{n}),\kk)$ has no deep points.
\end{thm}

\begin{proof}
Since the value of $p$ on each edge must be $1$, Condition \eqref{eq: alternating product} becomes $1=(-1)^{\frac{(n+3)+2}{2}}$, which holds iff $n\equiv 3\bmod{4}$ (or $n\equiv 1\bmod{4}$ and $\operatorname{char}(\kk)=2$). Formula \eqref{eq: arbitrary diagonal} reduces to \eqref{eq: polygondiagonal}.
\end{proof}

\begin{rem}\label{rem: polygonsingular}
Comparing with the classification of singularities in \cite[Theorem A]{BFMS23}, we see that the deep points of $\CA(A_n)$ are precisely the singularities. By contrast, $V(\CA(\Delta_n),\mathbb{C})$ is non-singular (\cite[Corollary 7.9]{MulLA}), and so the deep points in Theorem \ref{thm: deeppolygon} are not singularities.
\end{rem}



\draftnewpage

\section{Deep points of cluster algebras of rank 2}

In this section, we describe deep points of cluster algebras of rank 2, with arbitrary coefficients.

\subsection{Cluster algebras of rank 2}

We consider a cluster algebra $\CA$ with extended exchange matrix 
\begin{equation}\label{eq: rank2matrix}
\begin{bmatrix}
0 & b \\
-c & 0 \\
-\gvec{e}_2 & \gvec{e}_1 
\end{bmatrix}
\end{equation}
for some positive integers $b,c$ and some column vectors $\gvec{e}_1,\gvec{e}_2\in \mathbb{Z}^f$. We refer to this as a \textbf{cluster algebra of $(b,c)$-type}; or simply a \textbf{cluster algebra of rank $2$} when we don't wish to specify $b$ and $c$. Note that $f$ may be $0$, in which case $\gvec{e}_1,\gvec{e}_2$ are empty and there are no frozen variables.

Since $\CA$ is acyclic, there is a presentation due to \cite{BFZ05} 
\begin{equation}\label{eq: rank2pres}:
\CA \simeq  \mathbb{Z}[x_0,x_1,x_2,x_3,y_1^{\pm1},y_2^{\pm1},\dotsc ,y_f^{\pm1}]/
\langle x_0x_2 - y^{[\gvec{e}_1]_+} x_1^b - y^{[\gvec{e}_1]_-} , x_1x_3 - y^{[\gvec{e}_2]_+}x_2^c -y^{[\gvec{e}_2]_-}\rangle
\end{equation}
Note that $x_0$ and $x_3$ are being used to denote the mutations of $x_2$ and $x_1$, respectively. We extend this convention to recursively define $x_i$ for all $i\in \mathbb{Z}$ by the rule that the clusters are of the form $\{x_i,x_{i+1},y_1,y_2,\dotsc ,y_f\}$. The mutation relations become the family of equations (below left):
\[ 
x_{i-1}x_{i+1} = 
\left\{
\begin{array}{cc}
y^{[\gvec{e}_i]_+} x_i^b +y^{[\gvec{e}_i]_-} & \text{if $i$ is odd} \\
y^{[\gvec{e}_i]_+} x_i^c +y^{[\gvec{e}_i]_-} & \text{if $i$ is even} \\
\end{array}
\right\}
\hspace{.75cm}
\gvec{e}_{i+1}+\gvec{e}_{i-1} = 
\left\{
\begin{array}{cc}
-b[\gvec{e}_i]_- & \text{if $i$ is odd} \\
-c[\gvec{e}_i]_- & \text{if $i$ is even} \\
\end{array}
\right\}
\]
for a sequence of vectors $(\gvec{e}_i,i\in \mathbb{Z})$ in $\mathbb{Z}^f$ recursively defined by the equations (above right).


The presentation \eqref{eq: rank2pres} identifies the cluster variety of $\CA$ with the solutions to a pair of equations:
\[ 
\{ (x_0,x_1,x_2,x_3,y_1,y_2,\dotsc ,y_f)\in \kk^{4}\times (\kk^\times)^f
\mid x_0x_2 = y^{[\gvec{e}_1]_+} x_1^b + y^{[\gvec{e}_1]_-} , x_1x_3 = y^{[\gvec{e}_2]_+}x_2^c + y^{[\gvec{e}_2]_-} \} 
\]
Given a cluster $\{x_i,x_{i+1},y_1,y_2,\dotsc ,y_f\}$, the associated cluster torus is the subset of $V(\CA,\kk)$ where $p(x_i),p(x_{i+1})$ are not zero. Therefore, the deep locus consists of the subset of $V(\CA,\kk)$ such that at least one of $p(x_i),p(x_{i+1})$ vanishes for all $i\in \mathbb{Z}$.


\subsection{Vanishing lemmas}

Before we classify deep points in cluster algebras of $(b,c)$-type, we need a few lemmas which characterize which sets of cluster variables can vanish at the same time.

\begin{lemma}
Let $\CA$ be a cluster algebra of $(b,c)$-type, and let $p\in V(\CA,\kk)$. If $p(x_i)=0$, then $p(x_{i-1})\neq 0$ and $p(x_{i+1})\neq0$.
\end{lemma}
\begin{proof}
Assume $i$ is even, and apply $p$ to the mutation relation
$ x_{i-1}x_{i+1} = y^{[\gvec{g}_i]_+}x_i^c + y^{[\gvec{e}_i]_-} $ to get
\[ p(x_{i-1})p(x_{i+1}) = p(y)^{[\gvec{e}_i]_+}p(x_i)^c + p(y)^{[\gvec{e}_i]_-} \]
Since $p(x_i)=0$ and each coordinate of $p(y)$ is non-zero, the product $p(x_{i-1})p(x_{i+1})$ must be non-zero. Therefore, $p(x_{i-1})\neq0$ and $p(x_{i+1})\neq0$. The case of $i$ odd is the same.
\end{proof}

The lemma implies that at most one of $x_i$ and $x_{i+1}$ can vanish at any point in $V(\CA,\kk)$. Since $\{x_i,x_{i+1}\}$ is a cluster, at least one of them must vanish at each deep point. Thus, at a deep point, the sequence of cluster variables must alternate between vanishing and non-vanishing.


\begin{coro}\label{coro: ranktwodeeppoint}
Let $\CA$ be a cluster algebra of $(b,c)$-type.
Then a point $p\in V(\CA,\kk)$ is deep iff one of the following two mutually exclusive conditions hold.
\begin{enumerate}
    \item For all $i\in \mathbb{Z}$, $p(x_{2i})=0$ and $p(x_{2i+1})\neq0$.
    \item For all $i\in \mathbb{Z}$, $p(x_{2i})\neq 0$ and $p(x_{2i+1}) = 0$.
\end{enumerate}
\end{coro}

The vanishing of cluster variables is not independent. In most cases, the vanishing of $x_{i-2}$ and $x_i$ is enough to imply the vanishing of $x_{i+2}$. This is easiest to state precisely in terms of ideals rather than points.

\begin{lemma}\label{lemma:x0x2x4}
Let $\CA$ be a cluster algebra of $(b,c)$-type. For all $i\in \mathbb{Z}$,
\begin{enumerate}
    \item If $c>1$, then $ x_{2i+2}$ is in the $\CA$-ideal $\langle x_{2i-2},x_{2i}\rangle$ generated by $x_{2i-2}$ and $x_{2i}$.
    \item If $c=1$, then the $\CA$-ideals $\langle x_{2i-2},x_{2i},x_{2i+2}\rangle$ and $\langle x_{2i-2},x_{2i},b\rangle$ are equal.
\end{enumerate}
\end{lemma}

\begin{proof}
Define a polynomial $g(x_{2i})$ in $x_{2i}$ with coefficients in $\mathbb{Z}[y_1^{\pm1},y_2^{\pm1},\dotsc ,y_f^{\pm1}]$ by the rule
\[ x_{2i-1}^b x_{2i+1}^b = (y^{[\gvec{e}_{2i}]_+}x_{2i}^c+y^{[\gvec{e}_{2i}]_-})^b =  x_{2i}^c g(x_{2i})  + y^{b[\gvec{e}_{2i}]_-}\]
Note that the constant term of $g(x_{2i})$ is $b y^{[\gvec{e}_{2i}]_+ + (b-1)[\gvec{e}_{2i}]_-}$.
Then
\[
\begin{aligned}
x_{2i-1}^b x_{2i} x_{2i+2}
&= x_{2i-1}^b \left( y^{[\gvec{e}_{2i+1}]_+}x_{2i+1}^b+y^{[\gvec{e}_{2i+1}]_-} \right) 
= y^{[\gvec{e}_{2i+1}]_+} ( x_{2i-1}^b x_{2i+1}^b ) + y^{[\gvec{e}_{2i+1}]_-}x_{2i-1}^b \\
&=  y^{[\gvec{e}_{2i+1}]_+} ( x_{2i}^c g(x_{2i}) + y^{b[\gvec{e}_{2i}]_-} ) + y^{[\gvec{e}_{2i+1}]_-}x_{2i-1}^b \\
&=  y^{[\gvec{e}_{2i+1}]_+}  x_{2i}^c g(x_{2i}) + y^{b[\gvec{e}_{2i}]_-+[\gvec{e}_{2i+1}]_+}  + y^{[\gvec{e}_{2i+1}]_-}x_{2i-1}^b \\
\end{aligned}
\]
Since $\gvec{e}_{2i-1} +\gvec{e}_{2i+1} = -b[\gvec{e}_{2i}]_-$, this can be rewritten as
\[
\begin{aligned}
x_{2i-1}^b x_{2i} x_{2i+2}
&= y^{[\gvec{e}_{2i+1}]_+}  x_{2i}^c g(x_{2i}) + y^{[\gvec{e}_{2i+1}]_--\gvec{e}_{2i-1}}  + y^{[\gvec{e}_{2i+1}]_-}x_{2i-1}^b \\
&= y^{[\gvec{e}_{2i+1}]_+}  x_{2i}^c g(x_{2i}) + y^{[\gvec{e}_{2i+1}]_--[\gvec{e}_{2i-1}]_+}(y^{[\gvec{e}_{2i-1}]_-}  + y^{[\gvec{e}_{2i-1}]_+}x_{2i-1}^b) \\
&= y^{[\gvec{e}_{2i+1}]_+}  x_{2i}^c g(x_{2i}) + y^{[\gvec{e}_{2i+1}]_--[\gvec{e}_{2i-1}]_+}x_{2i-2}x_{2i} \\
\end{aligned}
\]
Canceling $x_{2i}$ from both sides yields
\begin{align*}
x_{2i-1}^b x_{2i+2} 
&= y^{[\gvec{e}_{2i+1}]_+}  x_{2i}^{c-1} g(x_{2i}) + y^{[\gvec{e}_{2i+1}]_--[\gvec{e}_{2i-1}]_+}x_{2i-2}
\end{align*}
Using the fact that $y^{-b[\gvec{e}_{2i}]_-}(x_{2i-1}^bx_{2i+1}^b-x_{2i}^cg(x_{2i}))=1$, we get a formula for $x_{2i+2}$.
\[
\begin{aligned}
x_{2i+2} 
&=
y^{-b[\gvec{e}_{2i}]_-}(x_{2i-1}^bx_{2i+1}^b-x_{2i}^cg(x_{2i})) x_{2i+2} 
\\ &
=
y^{-b[\gvec{e}_{2i}]_-}(x_{2i-1}^bx_{2i+1}^bx_{2i+2}-x_{2i}^cx_{2i+2}g(x_{2i})) \\
&=
y^{-b[\gvec{e}_{2i}]_-}(y^{[\gvec{e}_{2i+1}]_+}  x_{2i}^{c-1}x_{2i+1}^b g(x_{2i}) + y^{[\gvec{e}_{2i+1}]_--[\gvec{e}_{2i-1}]_+}x_{2i-2}x_{2i+1}^b-x_{2i}^cx_{2i+2}g(x_{2i})) \\
&=
y^{-b[\gvec{e}_{2i}]_-}(
x_{2i}^{c-1}g(x_{2i})
(y^{[\gvec{e}_{2i+1}]_+}  x_{2i+1}^b -x_{2i}x_{2i+2})
+ y^{[\gvec{e}_{2i+1}]_--[\gvec{e}_{2i-1}]_+}x_{2i-2}x_{2i+1}^b
) \\
&=
y^{-b[\gvec{e}_{2i}]_-}(
-y^{[\gvec{e}_{2i+1}]_-}
x_{2i}^{c-1}g(x_{2i})
+ y^{[\gvec{e}_{2i+1}]_--[\gvec{e}_{2i-1}]_+}x_{2i-2}x_{2i+1}^b
) \\
\end{aligned}
\]
If $c>1$, then $x_{2i+2}$ is in the ideal generated by $x_{2i-2}$ and $x_{2i}$. If $c=1$, then $x_{2i+2}$ is congruent to
\[ 
-y^{-b[\gvec{e}_{2i}]_-+[\gvec{e}_{2i+1}]_-}
g(x_{2i})
\equiv -y^{-b[\gvec{e}_{2i}]_-+[\gvec{e}_{2i+1}]_-}
(b y^{[\gvec{e}_{2i}]_+ + (b-1)[\gvec{e}_{2i}]_-})
= -b y^{[\gvec{e}_{2i+1}]_-+\gvec{e}_{2i}}
\bmod \langle x_{2i-2}, x_{2i} \rangle. 
\]
Since $-y^{[\gvec{e}_{2i+1}]_-+\gvec{e}_{2i}}$ is invertible in $\CA$, we have the equality $\langle x_{2i-2},x_{2i},x_{2i+2}\rangle=\langle x_{2i-2},x_{2i},b\rangle$.
\end{proof}

\begin{lemma}\label{lemma:ranktwoideals}
Let $\CA$ be a cluster algebra of $(b,c)$-type.
\begin{enumerate}
    \item The $\CA$-ideal $\langle \dotsc ,x_{-2},x_0,x_2,x_4,\dotsc \rangle$ generated by the even cluster variables is equal to
    \[
    \left\{
    \begin{array}{cc}
    \langle x_0,x_2,b\rangle & \text{if }c=1 \\
    \langle x_0,x_2\rangle & \text{if }c>1 \\
    \end{array}
    \right\}
    \]
    \item The $\CA$-ideal $\langle \dotsc,x_{-1},x_1,x_3,x_5,\dotsc \rangle$ generated by the odd cluster variables is equal to
    \[
    \left\{
    \begin{array}{cc}
    \langle x_1,x_3,c\rangle & \text{if }b=1 \\
    \langle x_1,x_3\rangle & \text{if }b>1 \\
    \end{array}
    \right\}
    \]\end{enumerate}
\end{lemma}
\begin{proof}
If $c>1$, then $x_4\in \langle x_0,x_2\rangle$ by Lemma \ref{lemma:x0x2x4}, and so $\langle x_0,x_2,x_4\rangle = \langle x_0,x_2\rangle$. Iterating this logic, $\langle x_0,x_2,x_4,x_6,\dotsc \rangle=\langle x_0,x_2\rangle$. By symmetry, $x_{-2} \in \langle x_0,x_2\rangle$ by Lemma \ref{lemma:x0x2x4}; iterating this logic in the negative direction, 
$ \langle \dotsc ,x_{-2},x_0,x_2,x_4,\dotsc \rangle = \langle x_0,x_2\rangle $. The case of $c=1$ is identical.  

The odd cluster variables case then follows from a `reindexing' isomorphism $\CA\rightarrow \CA'$ to a cluster algebra of $(c,b)$-type which sends $x_i$ to $x_{i+1}$.
\end{proof}

\subsection{Deep points}

We can now parametrize the deep locus of cluster algebras of rank 2.

\begin{thm}
Let $\CA$ be a cluster algebra of $(b,c)$-type.
\begin{itemize}
	\item Assume either $c>1$ or $\operatorname{char}(\kk)$ divides $b$. For each $\alpha \in \kk^\times$ and each $\beta\in (\kk^\times)^f$ such that $\alpha^b+\beta^{-\gvec{e}_1}=0$, there is a deep point in $V(\CA,\kk)$ defined by
	\[ (x_0,x_1,x_2,x_3,y_1,y_2,\dotsc ,y_f) \mapsto (0,\alpha, 0 , \alpha^{-1}\beta^{[\gvec{e}_2]_-} , \beta_1,\beta_2, \dotsc ,\beta_f) \]

	\item Assume either $b>1$ or $\operatorname{char}(\kk)$ divides $c$. For each $\alpha \in \kk^\times$ and each $\beta\in (\kk^\times)^f$ such that $\alpha^c+\beta^{-\gvec{e}_2}=0$, there is a deep point in $V(\CA,\kk)$ defined by 
	\[ (x_0,x_1,x_2,x_3,y_1,y_2, \dotsc ,y_f) \mapsto (\alpha^{-1}\beta^{[\gvec{e}_1]_-} , 0,\alpha, 0, \beta_1,\beta_2,\dotsc ,\beta_f) \]
\end{itemize}
Every deep point in $V(\CA,\kk)$ is one of these two mutually exclusive types.
\end{thm}

\begin{proof}
Consider $p\in V(\CA,\kk)$. By Corollary \ref{coro: ranktwodeeppoint}, $p$ is deep iff its kernel either contains every even cluster variable, or every odd cluster variable. By Lemma \ref{lemma:ranktwoideals}, $p$ is deep iff 
\begin{enumerate}
	\item its kernel contains $x_0$ and $x_2$ (and $b$, if $c=1$), or 
	\item its kernel contains $x_1$ and $x_3$ (and $c$, if $b=1$).
\end{enumerate}
The kernel of $p$ contains an integer iff the characteristic of the target field $\kk$ divides that integer, so the parenthetical conditions above may be replaced by the assumptions in the theorem.

The presentation \eqref{eq: rank2pres} implies that $p$ is determined by the values on $x_0,x_1,x_2,x_3,y_1,y_2,\dotsc ,y_f$, subject to the two equations
\[ 
p(x_0)p(x_2) = p(y)^{[\gvec{e}_1]_+}p(x_1)^b + p(y)^{[\gvec{e}_1]_-}
\text{ and }
p(x_1)p(x_3) = p(y)^{[\gvec{e}_2]_+}p(x_2)^b + p(y)^{[\gvec{e}_2]_-}
\]
If we set $p(x_0)=p(x_2)=0$, the remaining values must satisfy the equations
\[ 
0 = p(y)^{[\gvec{e}_1]_+}p(x_1)^b + p(y)^{[\gvec{e}_1]_-}
\text{ and }
p(x_1)p(x_3) = p(y)^{[\gvec{e}_2]_-}
\]
Note that both equations force $p(x_1)\neq0$. We may rewrite these equations as
\[ 
p(x_1)^b + p(y)^{-\gvec{e}_1} =0
\text{ and }
p(x_3) = p(y)^{[\gvec{e}_2]_-}p(x_1)^{-1}
\]
Using the second equation as the definition of $p(x_3)$, we see that a deep point which kills the even cluster variables is equivalent to a choice of $p(x_1)$ and $p(y)$ satisfying the first equation above; this gives the first half of the theorem. The second half of the theorem is identical.
\end{proof}

It is worth considering the case when there are no frozens, as the results are particularly simple.


\begin{coro}
Let $\CA$ be the cluster algebra of $(b,c)$-type with no frozens. 
\begin{itemize}
	\item Assume either $c>1$ or $\operatorname{char}(\kk)$ divides $b$. For each $\alpha \in \kk^\times$ such that $\alpha^b=-1$, there is a deep point in $V(\CA,\kk)$ defined by
	\[
	p(x_i) = 
	\left\{
	\begin{array}{cc}
	0 & \text{if $i$ is even} \\
	\alpha & \text{if }i\equiv 1 \bmod{4} \\
	\alpha^{-1} & \text{if }i\equiv 3 \bmod{4} \\
	\end{array}
	\right\}
	\]
	\item Assume either $b>1$ or $\operatorname{char}(\kk)$ divides $c$. For each $\alpha \in \kk^\times$ such that $\alpha^c=-1$, there is a deep point in $V(\CA,\kk)$ defined by
	\[
	p(x_i) = 
	\left\{
	\begin{array}{cc}
	0 & \text{if $i$ is odd} \\
	\alpha & \text{if }i\equiv 2 \bmod{4} \\
	\alpha^{-1} & \text{if }i\equiv 0 \bmod{4} \\
	\end{array}
	\right\}
	\]
\end{itemize}
Every deep point in $V(\CA,\kk)$ is one of these two mutually exclusive types.
\end{coro}


\begin{ex}
The  cluster algebra of type $A_2$ has extended exchange matrix
\[
\begin{bmatrix}
0 & 1 \\
-1 & 0
\end{bmatrix}
\]
Since $b=c=1$, neither of the conditions in the theorem are possible. Thus, the cluster algebra of type $A_2$ has no deep points, recovering a special case of Theorem \ref{thm: deepAn}.
\end{ex}

\begin{ex}\label{ex: type 1 2 deep points}
Consider a cluster algebra with no frozens and exchange matrix
\[
\begin{bmatrix}
0 & 1 \\
-c & 0
\end{bmatrix}
\]
for any choice of $c>1$. 
Then the first condition holds and the equation $\alpha^b+1=0$ implies that $\alpha=-1$. Therefore, there is a unique deep point of the first kind, which is defined by
	\[
	p(x_i) = 
	\left\{
	\begin{array}{cc}
	0 & \text{if $i$ is even} \\
	-1 & \text{if $i$ is odd} \\
	\end{array}
	\right\}
	\]
There are only deep points of the second kind when $\operatorname{char}(\kk)$ divides $c$, in which case there is a unique deep point of the second kind, which is defined by
	\[
	p(x_i) = 
	\left\{
	\begin{array}{cc}
	0 & \text{if $i$ is odd} \\
	-1 & \text{if $i$ is even} \\
	\end{array}
	\right\}
	\qedhere
	\]
\end{ex}

\begin{ex}\label{ex: kronecker}
The $2$-Kronecker cluster algebra (with no frozens) has extended exchange matrix
\[
\begin{bmatrix}
0 & 2 \\
-2 & 0
\end{bmatrix}
\]
Each square root $\alpha$ of $-1$ in $\kk$ gives a deep point of the first type, defined by 
	\[
	p(x_i) = 
	\left\{
	\begin{array}{cc}
	0 & \text{if $i$ is even} \\
	\alpha & \text{if }i\equiv 1 \bmod{4} \\
	\alpha^{-1} & \text{if }i\equiv 3 \bmod{4} \\
	\end{array}
	\right\}
	\]
Similarly, each square root $\alpha$ of $-1$ in $\kk$ gives a deep point of the second type, defined by 
	\[
	p(x_i) = 
	\left\{
	\begin{array}{cc}
	0 & \text{if $i$ is odd} \\
	\alpha & \text{if }i\equiv 2 \bmod{4} \\
	\alpha^{-1} & \text{if }i\equiv 0 \bmod{4} \\
	\end{array}
	\right\}
	\]
So, over $\mathbb{C}$ there are 4 deep points, over $\mathbb{R}$ there are no deep points, and over a field of characteristic 2 there are 2 deep points.
\end{ex}

\begin{rem}\label{rem: kronecker}
Over $\mathbb{C}$, the 2-Kronecker cluster algebra (with no frozens) is non-singular (e.g.~by \cite[Corollary 7.9]{MulLA})). Therefore, Example \ref{ex: kronecker} provides examples of non-singular deep points.
\end{rem}

\draftnewpage

\section{Deep points of the Markov cluster algebra}\label{section: Markov}

In this section, we classify the deep points of the Markov cluster algebra and its upper cluster algebra, and consider the relation between the two.

\subsection{The Markov cluster algebra}

One of the most well-known and well-studied cluster algebras is the \textbf{Markov cluster algebra}, the cluster algebra associated to the \textbf{Markov quiver} (Figure \ref{fig: MarkovQuiver}). 

\begin{figure}[h!t]
	\begin{tikzpicture}
		\node[mutable] (x) at (90:1) {$x$};
		\node[mutable] (y) at (210:1) {$y$};
		\node[mutable] (z) at (-30:1) {$z$};
		\draw[-angle 90,relative, out=15,in=165] (x) to node[inner sep=.1cm] (xy1m) {} (y);
		\draw[-angle 90,relative, out=-15,in=-165] (x) to node[inner sep=.1cm] (xy2m) {} (y);
		\draw[-angle 90,relative, out=15,in=165] (y) to node[inner sep=.1cm] (yz1m) {} (z);
		\draw[-angle 90,relative, out=-15,in=-165] (y) to node[inner sep=.1cm] (yz2m) {} (z);
		\draw[-angle 90,relative, out=15,in=165] (z) to node[inner sep=.1cm] (zx1m) {} (x);
		\draw[-angle 90,relative, out=-15,in=-165] (z) to node[inner sep=.1cm] (zx2m) {} (x);
    \end{tikzpicture}
    \caption{The Markov quiver $Q$ with cluster $\{x,y,z\}$.}
    \label{fig: MarkovQuiver}
\end{figure}
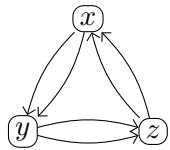

The Markov cluster algebra arises in many applications (as the cluster algebra of a once-punctured torus \cite[Example 4.6]{FST08}, in the study of \emph{Markov triples} \cite[Appendix B]{FG07}, etc), and is the smallest example of many pathological behaviors in cluster algebras; e.g.~it is not equal to its own cluster algebra \cite{BFZ05}, it is not finitely generated, and it is non-Noetherian \cite{MulLA}.

\subsection{Mutation-cyclic cluster algebras of rank 3}

We start with some general results about cluster algebras of rank 3 which are \textbf{mutation-cyclic}; that is, not acyclic.\footnote{A sharp characterization of mutation-cyclic quivers on three vertices appears in \cite{BBH11}.}
Ironically, the abundance of cycles in the mutation-equivalent quivers implies a lack of cycles in the exchange graph.


\begin{thm}[{\cite[Theorem 10.1]{War14}}]\label{thm: 3tree}
The exchange graph of a mutation-cyclic skew-symmetric cluster algebra of rank 3 is a 3-regular tree.
\end{thm}

\noindent Equivalently, there is a unique reduced\footnote{By \emph{reduced}, we mean that the sequence never mutates the same vertex twice in a row.} sequence of mutations between any two clusters.

The theorem allows us to partition the cluster variables in such a cluster algebra into three types. For the sake of terminology, let us choose an initial cluster in $\CA$ and denote the cluster variables $\{x,y,z\}$.
We recursively define the \textbf{$x$-type} cluster variables in $\CA$ to consist of $x$ itself and any mutation of an $x$-type cluster variable; \textbf{$y$-type} and \textbf{$z$-type} cluster variables are defined similarly.

\begin{prop}
Every cluster variable in a mutation-cyclic skew-symmetric cluster algebra of rank 3 is exactly one of the three types, and every cluster consists of one variable of each type.
\end{prop}
\begin{proof}
By Theorem \ref{thm: 3tree}, every cluster in $\CA$ can be obtained from the initial cluster by a unique reduced sequence of mutations. Each mutation replaces a single cluster variable by another variable of the same type, and so each cluster in the sequence consists of one variable of each type. 
\end{proof}

This has the following application for the deep points of the cluster algebra of any mutation-cyclic quiver with three vertices.

\def\x{\overline{x}}
\def\y{\overline{y}}
\def\z{\overline{z}}

\begin{lemma}\label{lemma: 3killtype}
Let $\CA$ be a mutation-cyclic skew-symmetric cluster algebra of rank 3. 
\begin{enumerate}
    \item If $p\in V(\CA,\kk)$ is deep, then either 
    \begin{enumerate}	
    	\item $p$ kills every cluster variable of one type and no other cluster variables, or 
	\item $p$ kills every cluster variable. 
	\end{enumerate}
    \item Two deep points in $p\in V(\CA,\kk)$ with the same values on one cluster must be equal.
\end{enumerate}
%
\end{lemma}
\begin{proof}
Let $\{\x,\y,\z\}$ be a cluster in $\CA$, in which $\x$ is the $x$-type variable, $\y$ is the $y$-type variable, and $\z$ is the $z$-type variable. 
Since $\CA$ is mutation-cyclic, the quiver associated to $\{\x,\y,\z\}$ is cyclic and the mutation relation for $\y$ will be of the form $\y\y' = \x^b + \z^c$ for some $b,c>0$.

Since $p$ is deep, it must kill at least one of $\x,\y,\z$.
Suppose $p$ kills at least two of $\x,\y,\z$; without loss of generality, assume $p(\x)=p(\y)=0$.  Applying $p$ to $\y\y' = \x^b + \z^c$ yields $p(\z)^c=0$, so $p(\z)=0$. Therefore, if $p$ kills two cluster variables in a cluster, it must kill the third. Since adjacent clusters have two variables in common with this cluster, $p$ must kill all three cluster variables in those clusters as well. 
Iteratively, $p$ must then kill all three cluster variables in every cluster.

Next, suppose $p$ kills only one of $\x,\y,\z$; without loss of generality, assume $p(\x)=0$ and $p(\y),p(\z)\neq0$. Since $p$ is deep, the mutated variable $\x'$ in the cluster $\{\x',\y,\z\}$ must also be sent to zero by $p$. Mutating at $\y$, we have $\y\y'=\x^b+\z^c$, and so $p(\y')= \frac{p(\z)^c}{p(\y)}\neq0$; the mutation of $\z$ is similarly not killed. Iteratively, we see that all of the $x$-type variables are killed by $p$, and none of the $y$- or $z$-type variables are killed by $p$. 

Let $p,p'$ be deep points with the same values on $\{\x,\y,\z\}$. If $p(\x)=p(\y)=p(\z)=0$, then $p$ and $p'$ kill every cluster variable; since this is a generating set of $\CA$, they must coincide. Suppose they both kill cluster variables of one type and no other cluster variables; without loss of generality, assume they both kill all $x$-type cluster variables. Mutating at $\y$, we have $\y\y'=\x^b+\z^c$, and so $p(\y')= \frac{p(\z)^c}{p(\y)}= \frac{p'(\z)^c}{p'(\y)} = p'(\y')$. Therefore, $p$ and $p'$ have the same values on all adjacent clusters as well. Iteratively, we see that $p$ and $p'$ have the same values on all cluster variables; since this is a generating set of $\CA$, they must coincide.
\end{proof}

\begin{warn}
It is not true that two (not necessarily deep) points in $V(\CA,\kk)$ with the same values on a cluster must be equal. 
For example, the Markov cluster algebra has multiple complex points which send the initial cluster $(x,y,z)$ to $(0,1,i)$, but only one of these points is deep.
\end{warn}

\begin{warn}
Lemma \ref{lemma: 3killtype} does not guarantee the existence of any deep points! General mutation-cyclic rank 3 cluster algebras are poorly behaved (e.g.~they are always non-Noetherian) and poorly understood (e.g.~no explicit presentation is known), so a parametrization of their deep points eludes us. We are only able to circumvent this for the Markov cluster algebra via its upper cluster algebra, which is much better behaved (see next section).
\end{warn}

\subsection{The Markov upper cluster algebra}

We now turn to the Markov upper cluster algebra.
Letting $x,y,z$ denote the initial cluster in the Markov cluster algebra $\CA$, we refer to  
\[ M \coloneqq \frac{x^2+y^2+z^2}{xyz}\]
as the \textbf{Markov element}. 
This element satisfies many remarkable properties. For example, it is independent on the choice of initial cluster; that is, for any other cluster $\{\x,\y,\z\}$, 
\[ M = \frac{\x^2+\y^2+\z^2}{\x\y\z}\]
As a consequence, $M$ is Laurent in every cluster, and so it is in the upper cluster algebra $\mathcal{U}$.

\begin{thm}[{\cite[Theorem 6.2.2]{MM15}}]
Let $\CA$ be the Markov cluster algebra, and let $\mathcal{U}$ be its upper cluster algebra. Then the initial cluster $\{x,y,z\}$ and the Markov element $M$ generate $\mathcal{U}$, with presentation
\[ \mathcal{U} \simeq \mathbb{Z}[x,y,z,M]/  \langle xyzM-x^2-y^2-z^2\rangle \]
\end{thm}

\noindent Thus, a homomorphism $\mathcal{U}\rightarrow \kk$ is determined by its values on $(x,y,z,M)$; conversely, any choice of values for $(x,y,z,M)$ extends to a homomorphism $\mathcal{U}\rightarrow \kk$ iff it kills $xyzM-x^2-y^2-z^2$. Therefore, $V(\mathcal{U},\kk)$ may be identified with the set of solutions to $xyzM=x^2+y^2+z^2$ inside $\kk^4$.

The proof of Theorem 6.2.2 in \cite{MM15} relied on a description of the deep locus of $\mathcal{U}$, albeit in a more algebraic language. We restate and reprove this description in the language of varieties.



\begin{thm}\label{thm: UMarkovdeep}
Let $\mathcal{U}$ be the Markov upper cluster algebra. 
\begin{itemize}
    \item[(A)] For each $\alpha\in \kk$, there is a deep point $p\in V(\mathcal{U},\kk)$ which sends 
    $(x,y,z,M) \mapsto (0,0,0,\alpha) $.
    \item[(X)] For each $\beta,\gamma\in \kk$ with $\beta^2+\gamma^2=0$, there is a deep point $p\in V(\mathcal{U},\kk)$ which sends 
    \[(x,y,z,M) \mapsto (0,\beta,\gamma,0) \]
    \item[(Y)] For each $\beta,\gamma\in \kk$ with $\beta^2+\gamma^2=0$, there is a deep point $p\in V(\mathcal{U},\kk)$ which sends 
    \[(x,y,z,M) \mapsto (\beta,0,\gamma,0) \]
    \item[(Z)] For each $\beta,\gamma\in \kk$ with $\beta^2+\gamma^2=0$, there is a deep point $p\in V(\mathcal{U},\kk)$ which sends 
    \[(x,y,z,M) \mapsto (\beta,\gamma,0,0) \]
\end{itemize}
Every deep point in $V(\mathcal{U},\kk)$ is one of these four types, which are pairwise disjoint except for the point which sends $(x,y,z,M)$ to $(0,0,0,0)$.
\end{thm}

\begin{proof}
For any cluster $\{\x,\y,\z\}$, the mutation $\x'$ satisfies $\x'=\y\z M-\x$. As a result, any point which kills $\x$ and either $M$ or $\y\z$ must also kill $\x'$. Iteratively applying this to the maps in the statement of the proposition shows that points of Type (A) kill every cluster variable, points of Type (X) kill every $x$-type variable, points of Type (Y) kill every $y$-type variable, and points of Type (Z) kill every $z$-type variable.

To show every deep point is of these types, consider a deep point $p\in V(\mathcal{U},\kk)$ with $p(x)=0$. The equation $xyzM=x^2+y^2+z^2$ implies that $p(y)^2+p(z)^2=0$. If $p(y)=0$, $p(z)=0$, so $p$ is of Type (A). If $p(y)\neq0$, then $p(z)\neq0$, and $p(x')=0$. The equation $x'=yzM-x$ implies that $p(yzM)=0$. Since $p(y) p(z)\neq0$, $p(M)=0$, and $p$ is of Type (X). By a similar argument, deep points which kill $y$ or $z$ but not $x$ must be of Types (Y) or (Z), respectively.
%
%
\end{proof}

\subsection{Deep points of the Markov cluster algebra}

We are now in a position to classify deep points in the Markov cluster algebra itself, with Theorem \ref{thm: UMarkovdeep} providing a source of deep points in $V(\CA,\kk)$ and and Lemma \ref{lemma: 3killtype} providing a bound.

\begin{prop}
Let $\CA$ be the Markov cluster algebra.
\begin{itemize}
    \item[(X)] For each $\beta,\gamma\in \kk$ with $\beta^2+\gamma^2=0$, there is a unique deep point $p\in V(\mathcal{A},\kk)$ which kills every $x$-type cluster variable and sends $(y,z)$ to $(\beta,\gamma)$.
    \item[(Y)] For each $\beta,\gamma\in \kk$ with $\beta^2+\gamma^2=0$, there is a unique deep point $p\in V(\mathcal{A},\kk)$ which kills every $y$-type cluster variable and sends $(x,z)$ to $(\beta,\gamma)$.
    \item[(Z)] For each $\beta,\gamma\in \kk$ with $\beta^2+\gamma^2=0$, there is a unique deep point $p\in V(\mathcal{A},\kk)$ which kills every $z$-type cluster variable and sends $(x,y)$ to $(\beta,\gamma)$.
\end{itemize}
Every deep point in $V(\mathcal{A},\kk)$ is one of these three types, which are pairwise disjoint except for the point which sends every cluster variable to $0$.
\end{prop}

\begin{proof}
A deep point in $V(\CA,\kk)$ which kills every $x$-type cluster variable must also kill $xx'=y^2+z^2$, and so it must be of Type (X).
To show such a deep point exists for all $\beta,\gamma$, observe that the deep point $p\in V(\mathcal{U},\kk)$ which sends $(x,y,z,M)$ to $(0,\beta,\gamma,0)$ (Type (X) in Theorem \ref{thm: UMarkovdeep}) restricts to a deep point in $V(\mathcal{A},\kk)$ which kills every $x$-type variable and sends $(y,z)$ to $(\beta,\gamma)$. Uniqueness follows from Lemma \ref{lemma: 3killtype}.2. Similar arguments hold for the other types.

By Lemma \ref{lemma: 3killtype}.1, every deep point of $\CA$ must kill every cluster variable of one type, and so it must be one of the three types stated above.
\end{proof}

\begin{rem}\label{rem: Markov Upper line to point}
We now consider the geometric relation between $V(\CA,\kk)$ and $V(\mathcal{U},\kk)$. The map
\[ V(\mathcal{U},\kk) \rightarrow V(\CA,\kk) \]
induces an isomorphism between deep points of Types (X), (Y), and (Z). 
However, since the deep points of Type (A) in $V(\mathcal{U},\kk)$ kill every cluster variable, they are all sent to the same point in $V(\CA,\kk)$: the point which sends every cluster variable to $0$.
\end{rem}




\begin{ex}
Over the complex numbers (i.e.~$\kk=\mathbb{C}$), the solution set of $\beta^2+\gamma^2=0$ breaks apart into two lines $\gamma=\pm i \beta$. As a consequence, the deep locus in $V(\mathcal{U},\mathbb{C})$ consists of seven lines, which are disjoint except for a common intersection point. Under the map
\[ V(\mathcal{U},\mathbb{C})\rightarrow V(\CA,\mathbb{C}) \]
six of these deep lines are sent bijectively to the six deep lines in $V(\CA,\mathbb{C})$, while the seventh is collapsed to the common intersection point.
\end{ex}

\draftnewpage

\section{Cluster algebras of unpunctured marked surfaces}

Another of the most well-understood classes of cluster algebras consists of those associated to a \emph{triangulable marked surface}. 
These cluster algebras were introduced in \cite{FST08} and their connection to \emph{Kauffman skein algebras} was explored in \cite{MW13,Mul16}. In this section, we will review the basics of commutative skein algebras of marked surfaces with boundary; for details on cluster algebras of marked surfaces with boundary, see \cite{FST08}.

\begin{rem}
We include the generality of marked surfaces with punctures when we can do so without much added difficulty. However, we will restrict to the unpunctured case when talking about skein algebras and cluster algebras, to avoid dealing with `tagged arcs' as in \cite{FST08}). 

This is not just an expository simplification; the deep locus of a punctured marked surface can behave substantially differently from the unpunctured case. For example, the deep locus of a once-punctured octagon is not equidimensional (contrasting the unpunctured case by Theorem \ref{thm: deepsurfacediss}); we hope to investigate the punctured case further in future work.
\end{rem}

\def\MM{\mathcal{M}}

\subsection{Marked surfaces}

A \textbf{marked surface} consists a smooth, oriented, compact surface-with-boundary $\SS$, together with a finite set of \textbf{marked points} $\MM$.
Marked points in the interior of $\SS$ are called \textbf{punctures}, and so a marked surface is \textbf{unpunctured} if $\MM\subset \partial\SS$. 
A connected marked surface is determined up to homeomorphism by a few numbers:
\begin{itemize}
    \item The genus $g$ of the underlying surface.
    \item The number $b$ of boundary components.
    \item The number $p$ of punctures.
    \item The unordered numbers $m_1,m_2,\dotsc,m_b$ of marked points on each boundary component.
\end{itemize}

A \textbf{marked curve} in a marked surface $\SS$ is an immersion of a compact, connected curve-with-boundary\footnote{Note that a compact, connected curve-with-boundary must be diffeomorphic to either an interval or a circle.} into $\SS$ such that
\begin{itemize}
    \item any endpoints map to $\MM$,
    \item the interior maps to $\SS\smallsetminus \MM$,
    \item the preimage of any point in $\SS\smallsetminus \MM$ is at most two points, and
    \item the set of \emph{crossings} (points in $\SS\smallsetminus \MM$ with at least two preimages) is finite.
\end{itemize}  
Intuitively, a marked curve is a curve drawn on the surface $\SS$ which may cross itself finitely many times (but only transversely) and either ends at marked points or forms a closed loop.
A \textbf{marked multicurve} is an immersion of finitely many compact curves-with-boundary satisfying these properties. We consider two marked curves (resp.~multicurves) \textbf{homotopic} if they are related by orientation-reversal and homotopies through the class of marked curves  (resp.~multicurves).\footnote{That is, homotopies which are marked curves (resp.~multicurves) at all times. In particular, the immersion condition prohibits homotopies which change the \emph{framing} in the sense of knot theory.}
%

Some additional terminology for marked curves:
\begin{itemize}
    \item A \textbf{marked arc} is a marked curve with two endpoints (i.e.~an immersed interval).
    \item A \textbf{loop} is a marked curve with no end points (i.e.~an immersed circle).
    \item A marked curve is \textbf{simple} if it has no crossings and is not contractible.
    \item A multicurve (i.e.~a finite set of marked curves) is \textbf{compatible} if there are no crossings and every curve is simple.
    \item A \textbf{boundary curve} is a marked curve homotopic to a marked curve contained in $\partial \SS$.\footnote{Since the interior of a marked curve cannot pass through a marked point, a boundary marked curve must either be a simple marked arc connecting adjacent marked points along the boundary, or a simple loop homotopic to an unmarked boundary component.}
\end{itemize}

\begin{ex}
A simple but important class of marked surfaces are the \textbf{(topological) polygons}. These are discs with no punctures and marked points along the boundary. Such a marked surface with $m$-many marked points is homeomorphic to any convex $m$-gon in the plane; under such a homemorphism, each edge and diagonal determines a marked arc.
\end{ex}

\begin{ex}
A less trivial example is the \textbf{$(2,2)$-annulus}, which is what we call the annulus with 2 marked points on each boundary component. It will be convenient to depict the annulus as a rectangle with two opposite sides identified (drawn dashed in Figure \ref{fig: annulus 2 2 simple marked arcs}). Up to homotopy, this marked surface has a unique simple loop and four boundary arcs.
\end{ex}

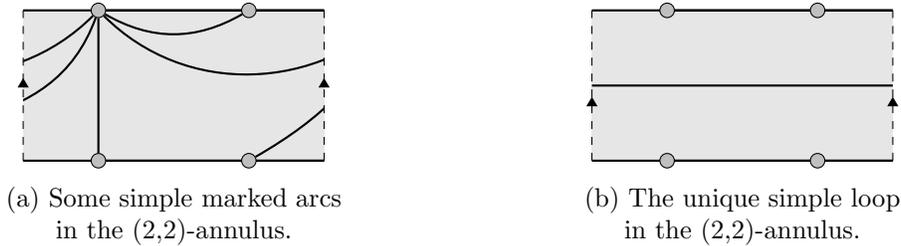
\begin{figure}[h!tb]
\captionsetup[subfigure]{justification=centering}
\centering
    

\begin{subfigure}[c]{0.45\linewidth}
\centering
\begin{tikzpicture}
   \path[fill=black!10] (-2,1) to (2,1) to (2,-1) to (-2,-1) to (-2,1);
    \draw[thick] (-2,1) to (2,1) (-2,-1) to (2,-1);
    \draw[dashed] (-2,-1) to (-2,1) (2,-1) to (2,1);
    \draw[-Triangle,thin] (-2,0) to (-2,0.1);
    \draw[-Triangle,thin] (2,0) to (2,0.1);
    
    \node[dot] (1) at (-1,1) {};
    \node[dot] (2) at (1,1) {};
    \node[dot] (3) at (1,-1) {};
    \node[dot] (4) at (-1,-1) {};
    \draw[thick] (1) to (2);
    \draw[thick] (2) to (2,1);
    \draw[thick] (3) to (4);
    \draw[thick] (3) to (2,-1);

    \draw[thick,out=-30,in=210] (1) to (2);
    \draw[thick] (1) to (4);
    \clip (-2,-1) rectangle (2,1);
    \draw[thick,out=-45,in=225] (1) to (3,1);
    \draw[thick,out=-45,in=225] (-5,1) to (1);
    \draw[thick,out=30,in=240] (3) to (3,1);
    \draw[thick,out=20,in=250] (-5,-1) to (1);
\end{tikzpicture}
\subcaption{Some simple marked arcs\\in the (2,2)-annulus.}
\label{fig: annulus 2 2 marked arcs}
\end{subfigure}
\begin{subfigure}[c]{0.45\linewidth}
\centering
\begin{tikzpicture}
   \path[fill=black!10] (-2,1) to (2,1) to (2,-1) to (-2,-1) to (-2,1);
    \draw[thick] (-2,1) to (2,1) (-2,-1) to (2,-1);
    \draw[dashed] (-2,-1) to (-2,1) (2,-1) to (2,1);
    \draw[-Triangle,thin] (-2,-.25) to (-2,-0.15);
    \draw[-Triangle,thin] (2,-.25) to (2,-0.15);
    
    \node[dot] (1) at (-1,1) {};
    \node[dot] (2) at (1,1) {};
    \node[dot] (3) at (1,-1) {};
    \node[dot] (4) at (-1,-1) {};
    \draw[thick] (1) to (2);
    \draw[thick] (2) to (2,1);
    \draw[thick] (3) to (4);
    \draw[thick] (3) to (2,-1);

    \draw[thick] (-2,0) to (2,0);
    
\end{tikzpicture}
\subcaption{The unique simple loop\\in the (2,2)-annulus.}
\label{fig: annulus 2 2 loop}
\end{subfigure}
\caption{Marked curves in the (2,2)-annulus.}
\label{fig: annulus 2 2 simple marked arcs}
\end{figure}

A \textbf{triangulation} of a marked surface $\SS$ is a simple marked multicurve which divides the surface into a union of topological triangles, together with the set of boundary arcs.\footnote{Including the boundary arcs in a triangulation is non-standard, but will simplify the connection to cluster algebras.}
 If a marked surface admits a triangulation, we say it is \textbf{triangulable}.

\begin{prop} [{\cite[Prop.~2.10]{FST08}}]\label{prop: triangulationcount}
Let $\SS$ be a connected marked surface with genus $g$, $b$-many marked points, $p$-many punctures, and $m$-many marked points on the boundary. Then $\SS$ is triangulable iff the set of marked points $\MM$ is not empty, every component of the boundary $\partial \SS$ contains at least one marked point, and the quantity
\[ n(\SS) \coloneqq 6g + 3b + 3p + 2m - 6 \]
is positive. Every triangulation of $\SS$ has $n(\SS)$-many marked arcs (counting boundary arcs).
\end{prop}



\begin{rem}
Every boundary marked arc is in every triangulation, and there are $m$-many of them. Therefore, every triangulation contains 
$ 6g + 3b + 3p + m - 6$
many non-boundary marked arcs; this matches the formula in \cite[Prop.~2.10]{FST08} and is the number of mutable cluster variables in the corresponding cluster algebra.
\end{rem}

\begin{ex}
Two triangulations of the $(2,2)$-annulus are given in Figure \ref{fig: annulus 2 2 triangulation}. Note that these triangulations each have 4 non-boundary arcs; together with the 4 boundary arcs, this matches
\[
n(\SS) = 6g + 3b + 3p +2m -6 = 6\cdot 0 + 3 \cdot 2 + 3 \cdot 0 + 2\cdot 4 -6 = 8
\qedhere
\]
\end{ex}

    


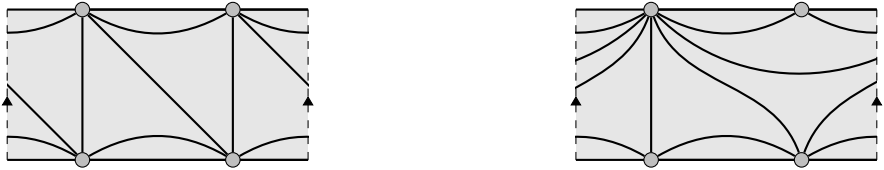
\begin{figure}[h!tb]
\captionsetup[subfigure]{justification=centering}
\centering
\begin{subfigure}[c]{0.45\linewidth}
\centering
\begin{tikzpicture}
    \path[fill=black!10] (-2,1) to (2,1) to (2,-1) to (-2,-1) to (-2,1);
    \draw[thick] (-2,1) to (2,1) (-2,-1) to (2,-1);
    \draw[dashed] (-2,-1) to (-2,1) (2,-1) to (2,1);
    \draw[-Triangle,thin] (-2,-.25) to (-2,-0.15);
    \draw[-Triangle,thin] (2,-.25) to (2,-0.15);
    
    \node[dot] (1) at (-1,1) {};
    \node[dot] (2) at (1,1) {};
    \node[dot] (3) at (1,-1) {};
    \node[dot] (4) at (-1,-1) {};
    \draw[thick] (1) to (2);
    \draw[thick] (2) to (2,1);
    \draw[thick] (3) to (4);
    \draw[thick] (3) to (2,-1);

    \draw[thick] (1) to (3);
    \draw[thick] (1) to (4);
    \draw[thick] (2) to (3);
    \clip (-2,-1) rectangle (2,1);
    \draw[thick] (2) to (3,-1);
    \draw[thick] (-3,1) to (4);
    
    \draw[thick,out=-30,in=210] (-3,1) to (1);
    \draw[thick,out=-30,in=210] (1) to (2);
    \draw[thick,out=-30,in=210] (2) to (3,1);
    \draw[thick,out=30,in=150] (-3,-1) to (4);
    \draw[thick,out=30,in=150] (4) to (3);
    \draw[thick,out=30,in=150] (3) to (3,-1);
\end{tikzpicture}
\end{subfigure}
\begin{subfigure}[c]{0.45\linewidth}
\centering
\begin{tikzpicture}
   \path[fill=black!10] (-2,1) to (2,1) to (2,-1) to (-2,-1) to (-2,1);
    \draw[thick] (-2,1) to (2,1) (-2,-1) to (2,-1);
    \draw[dashed] (-2,-1) to (-2,1) (2,-1) to (2,1);
    \draw[-Triangle,thin] (-2,-.25) to (-2,-0.15);
    \draw[-Triangle,thin] (2,-.25) to (2,-0.15);
    
    \node[dot] (1) at (-1,1) {};
    \node[dot] (2) at (1,1) {};
    \node[dot] (3) at (1,-1) {};
    \node[dot] (4) at (-1,-1) {};
    \draw[thick] (1) to (2);
    \draw[thick] (2) to (2,1);
    \draw[thick] (3) to (4);
    \draw[thick] (3) to (2,-1);

    \draw[thick,out=290,in=110] (1) to (3);
    \draw[thick] (1) to (4);
    \clip (-2,-1) rectangle (2,1);
    \draw[thick,out=-45,in=225] (1) to (3,1);
    \draw[thick,out=-45,in=225] (-5,1) to (1);
    \draw[thick,out=70,in=250] (3) to (3,1);
    \draw[thick,out=70,in=250] (-3,-1) to (1);

    \draw[thick,out=-30,in=210] (-3,1) to (1);
    \draw[thick,out=-30,in=210] (1) to (2);
    \draw[thick,out=-30,in=210] (2) to (3,1);
    \draw[thick,out=30,in=150] (-3,-1) to (4);
    \draw[thick,out=30,in=150] (4) to (3);
    \draw[thick,out=30,in=150] (3) to (3,-1);
\end{tikzpicture}
\end{subfigure}
\caption{Two marked triangulations of the $(2,2)$-annulus.}
\label{fig: annulus 2 2 triangulation}
\end{figure}

\subsection{The skein algebra of an unpunctured marked surface}

\def\Multi{\mathrm{Multi}}
\def\Sk{\operatorname{Sk}}

In \cite{GSV05, FG06, FST08}, it was shown how a triangulable marked surface $\SS$ determines a cluster algebra $\CA(\SS)$ with boundary coefficients. However, it will be more practical to work with the \emph{skein algebra} of the marked surface, a closely related algebra generated by marked curves with topologically-defined relations.



Given a marked surface $\SS$, let $\mathbb{Z}\Multi(\SS)$ denote the abelian group of $\mathbb{Z}$-linear combinations of marked multicurves in $\SS$ up to homotopy.
Define the \textbf{skein algebra} $\Sk_1(\SS)$ to be the quotient of $\mathbb{Z}\Multi(\SS)$ by the subgroup generated by the families of relations depicted in Figure \ref{fig: skeinrel}.\footnote{While we use the term `skein algebra' here for brevity, $\Sk_1(\SS)$ would more properly be called the \emph{marked Kauffman skein algebra of $(\SS,\MM)$ at $q=1$} to distinguish it from other variations.}  In each figure, the dashed circle represents a small contractible neighborhood in $\SS$, and all elements considered agree outside this circle.
We let $\langle\gamma\rangle$ denote the element of $\Sk_1(\SS)$ corresponding to a marked multicurve $\gamma$.

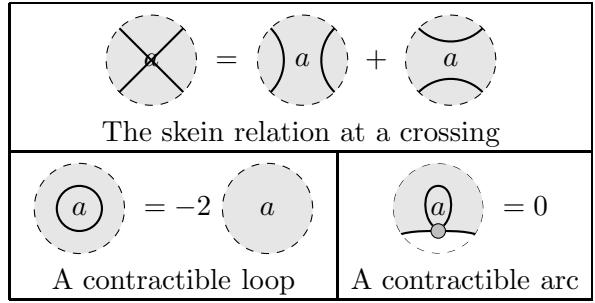
\begin{figure}[h!t]
\begin{tabular}{|c|c|}
\hline
\multicolumn{2}{|c|}{
$
\begin{tikzpicture}[scale=.15,baseline=(invisible.base)]
	\path[use as bounding box] (0,0) circle (5);
   \draw[fill=black!10,dashed] (0,0) circle (4);
    \draw[thick] (-2.83,-2.83) to (2.83,2.83);
    \draw[thick] (-2.83,2.83) to (2.83,-2.83);
    \node[opacity=0] (invisible) at (0,0) {$a$};
\end{tikzpicture}
=
\begin{tikzpicture}[scale=.15,baseline=(invisible.base)]
 	\path[use as bounding box] (0,0) circle (5);
   \draw[fill=black!10,dashed] (0,0) circle (4);
    \draw[thick] (-2.83,-2.83) to [out=45,in=-45] (-2.83,2.83);
    \draw[thick] (2.83,-2.83) to [out=135,in=-135] (2.83,2.83);
    \node[opacity=0] (invisible) at (0,0) {$a$};
\end{tikzpicture}
+
\begin{tikzpicture}[scale=.15,baseline=(invisible.base)]
	\path[use as bounding box] (0,0) circle (5);
    \draw[fill=black!10,dashed] (0,0) circle (4);
    \draw[thick] (-2.83,-2.83) to [out=45,in=135] (2.83,-2.83);
    \draw[thick] (-2.83,2.83) to [out=-45,in=-135] (2.83,2.83);
    \node[opacity=0] (invisible) at (0,0) {$a$};
\end{tikzpicture}
$} \\
\multicolumn{2}{|c|}{
The skein relation at a crossing
}
\\
\hline
$
\begin{tikzpicture}[scale=.15,baseline=(invisible.base)]
	\path[use as bounding box] (0,0) circle (5);
    \draw[fill=black!10,dashed] (0,0) circle (4);
    \draw[thick] (0,0) circle (2);
    \node[opacity=0] (invisible) at (0,0) {$a$};
\end{tikzpicture}
=-2
\begin{tikzpicture}[scale=.15,baseline=(invisible.base)]
	\path[use as bounding box] (0,0) circle (5);
    \draw[fill=black!10,dashed] (0,0) circle (4);
    \node[opacity=0] (invisible) at (0,0) {$a$};
\end{tikzpicture}
$
&
$
\begin{tikzpicture}[scale=.15,baseline=(invisible.base)]
	\path[use as bounding box] (0,0) circle (5);
	\clip (0,0) circle (4);
	\draw[fill=black!10,thick] (-5,-3) to [in=180,out=30] (0,-2) to [in=150,out=0] (5,-3) to [line to] (5,5) to (0,5) to (-5,5);
	\node (1) at (0,-2) [dot] {};
	\draw[thick] (1) to [out=45,in=0] (0,2) to [out=180,in=135] (1);
	\draw[dashed] (0,0) circle (4);
    \node[opacity=0] (invisible) at (0,0) {$a$};
\end{tikzpicture}
=0
$
\\
A contractible loop
&
A contractible arc
\\
\hline
\end{tabular}
\caption{The $q=1$ marked Kauffman skein relations 
}
\label{fig: skeinrel}
\end{figure}

Given two marked multicurves $\gamma$ and $\lambda$, their \textbf{superposition} $\gamma\cup \lambda$ is the union of the corresponding immersed curves. This may not be a marked multicurve, as the crossings may no longer be transverse or have more than two preimages. This can be resolved by replacing $\gamma$ and $\lambda$ with equivalent marked multicurves $\gamma'$ and $\lambda'$ whose superposition is a marked multicurve. While the resulting multicurve $\gamma' \cup \lambda'$ depends on the choice of $\gamma'$ and $\lambda'$, the corresponding element $\langle \gamma' \cup \lambda'\rangle $ in the skein algebra does not \cite[Proposition 3.5]{Mul16}. This gives a well-defined \textbf{superposition product} on the skein algebra $\Sk_1(\SS)$.



\begin{thm}[{\cite[Corollary 6.16]{Mul16}}]
    Let $\SS$ be an unpunctured, triangulable marked surface. Then the superposition product makes $\Sk_1(\SS)$ into a commutative domain. 
\end{thm}

\begin{ex}
If $\SS$ is a polygon with $m$-many marked points, we may index the marked points $1,2,\dotsc ,m$. This identifies the simple marked arcs with edges and diagonals in $\Delta_m$, and induces an isomorphism 
\[ \Sk_1(\Delta_n) \simeq P_{2,n} \]
As a consequence, the localization $\Sk_1(\Delta_n)[\partial \Sigma^{-1}]$ is a cluster algebra of type $A$, and its deep points are characterized by Theorem \ref{thm: deeppolygon}.
\end{ex}

\subsection{Cluster algebras of marked surfaces}



Generalizing how the cluster algebra of a polygon was defined as a localization of the Pl\"ucker algebra in Section \ref{section: polygons}, the cluster algebra of an unpunctured marked surface can be defined as a localization of the skein algebra.

\begin{thm}
Let $\SS$ be an unpunctured, triangulable marked surface. Then the localization 
\[ \Sk_1(\SS)[\partial\SS^{-1}] \]
of the skein algebra $\Sk_1(\SS)$ at the set of boundary arcs $\partial \SS$ is a cluster algebra. The frozen cluster variables are the boundary arcs, the mutable cluster variables are the simple non-boundary arcs, and the clusters are the triangulations of $\SS$.
%
%
\end{thm}

We call this \textbf{the cluster algebra of $\SS$ (with boundary coefficients)} and denote it by $\mathcal{A}(\SS)$.

\begin{proof}
Let $\mathcal{A}(\SS)$ be the Fomin-Shapiro-Thurston cluster algebra of $\SS$ and let $\mathcal{U}(\SS)$ be its upper cluster algebra.  
When $\lvert \MM \rvert \geq 2$, $\mathcal{A}(\SS)=\mathcal{U}(\SS)$ by \cite[Theorems 4.1 and 10.6]{MulLA}. When $\lvert \MM \rvert =1$, $\mathcal{A}(\SS)=\mathcal{U}(\SS)$ by \cite{CLS15}.
By \cite[Theorem 7.15]{Mul16},
there are containments 
\[\mathcal{A}(\SS)\subseteq \Sk_1(\SS)[\partial\SS^{-1}] \subseteq \mathcal{U}(\SS)\]
Therefore, all three algebras are equal.
%
\end{proof}

Since the skein algebra is generated by the simple marked curves and the relations are generated by the $q=1$ marked Kauffman skein relations, the cluster variety may be characterized as follows.

\begin{coro}
Let $\SS$ be an unpunctured, triangulable marked surface. A point $p\in V(\CA(\SS),\kk)$ is equivalent to a map 
\[ p:\{\text{simple marked curves in $\SS$ up to homotopy}\} \rightarrow \kk\]
such that the $q=1$ marked Kauffman skein relations hold (Figure \ref{fig: skeinrel}) and boundary arcs are not sent to $0$. Such a point is deep iff every triangulation contains an arc which is sent to $0$.
\end{coro}

From this point on, we identify points of $V(\CA(\SS),\kk)$ with such functions without comment.

\subsection{Cutting and gluing}\label{section: cutting}

Given a simple marked arc $s$ in a marked surface $\SS$, we may construct a new marked surface $\SS\smallsetminus s$ by \textbf{cutting} along $s$; that is, removing $s$ and adding two copies of $s$ (one to each side of the original $s$) (see Figure \ref{fig: annulus 2 2 cutting}). 
More generally, given a compatible collection of non-boundary marked arcs $S$, we construct $\SS\smallsetminus S$ by cutting along the arcs in $S$ in any order.

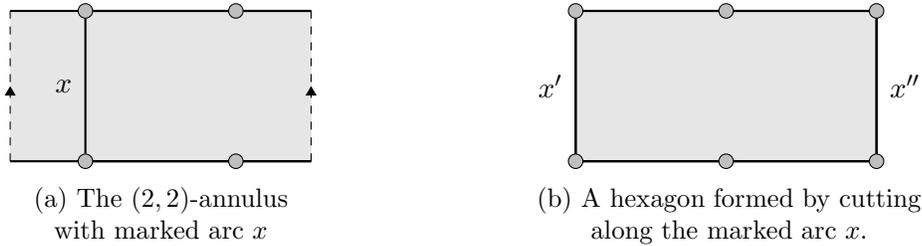
\begin{figure}[h!tb]
\captionsetup[subfigure]{justification=centering}
\centering
\begin{subfigure}[c]{0.45\linewidth}
\centering
\begin{tikzpicture}
    \path[fill=black!10] (-2,1) to (2,1) to (2,-1) to (-2,-1) to (-2,1);
    \draw[thick] (-2,1) to (2,1) (-2,-1) to (2,-1);
    \draw[dashed] (-2,-1) to (-2,1) (2,-1) to (2,1);
    \draw[-Triangle,thin] (-2,-.1) to (-2,0);
    \draw[-Triangle,thin] (2,-.1) to (2,0);
    
    \node[dot] (1) at (-1,1) {};
    \node[dot] (2) at (1,1) {};
    \node[dot] (3) at (1,-1) {};
    \node[dot] (4) at (-1,-1) {};
    \draw[thick] (1) to node [midway, left=1pt] {$x$} (4);
    \draw[thick] (1) to (2);
    \draw[thick] (2) to (2,1);
    \draw[thick] (3) to  (4);
    \draw[thick] (3) to (2,-1);

\end{tikzpicture}
\subcaption{The $(2,2)$-annulus\\with marked arc $x$}
\label{fig: annulus 2 2 identified sides}
\end{subfigure}
\begin{subfigure}[c]{0.45\linewidth}
\centering
\begin{tikzpicture}
    \draw[fill=black!10,thick] (-3,1) to (1,1) to (1,-1) to (-3,-1) to (-3,1);
    \node[dot] (1) at (-3,1) {};
    \node[dot] (2) at (-1,1) {};
    \node[dot] (3) at (1,1) {};
    \node[dot] (4) at (1,-1) {};
    \node[dot] (5) at (-1,-1) {};
    \node[dot] (6) at (-3,-1) {};
    \draw[thick] (2) to  (3);
    \draw[thick] (1) to  (2);
    \draw[thick] (4) to  (5);
    \draw[thick] (6) to  (5);
    \draw[thick] (6) to node [midway, left=1pt]  {$x'$} (1);
    \draw[thick] (4) to node [midway, right=1pt] {$x''$} (3);
\end{tikzpicture}
\subcaption{A hexagon formed by cutting \\along the marked arc $x$.}
\label{fig: annulus 2 2 poly diss}
\end{subfigure}
\caption{A cutting of the (2,2)-annulus.}
   \label{fig: annulus 2 2 cutting}
\end{figure}

The clusters algebras of $\SS$ and $\SS\smallsetminus S$ are related as follows.


\begin{prop}\label{prop: cutting algebras}
Let $\SS$ be a triangulable marked surface, and let $S$ be a compatible collection of non-boundary marked arcs in $\SS$. Then the map $\SS\smallsetminus S\rightarrow S$ induces an isomorphism
\[  \CA(\SS\smallsetminus S)/\langle s'-s'', \forall s\in S\rangle \xrightarrow{\sim} \CA(\SS)[S^{-1}] \]
Here, $s'$ and $s''$ denote the two preimages of $s$ in $\SS\smallsetminus S$.
\end{prop}

\begin{proof}
The topological map $\SS\smallsetminus S \rightarrow \SS$ is an immersion, and so it sends marked curves in $\SS\smallsetminus S$ to marked curves in $\SS$. Since the skein relations are local, this induces a morphism
\[ \Sk_1(\SS\smallsetminus S) \rightarrow \Sk_1(\SS) \]
For each $s\in S$, the two boundary arcs $s'$ and $s''$ in $\SS\smallsetminus S$ map to $s$. It follows that the above morphism factors through the quotient
\[ \Sk_1(\SS\smallsetminus S)/ \langle s'-s'', \forall s\in S\rangle \rightarrow \Sk_1(\SS) \]
By \cite[Corollary 4.13]{Mul16}, for each simple marked arc $x$ in $\SS$, there is some $k>0$ such that for any $ \langle y \rangle \in \Sk_1 (\SS) $, $\langle x^k \rangle \langle y \rangle$ does not intersect $x$. These exponents are multiplicative, so for any marked multiarc $\gamma$ in $\SS$, there is a monomial $S^m$ in $S$ such that $\langle S^m\rangle \langle \gamma\rangle$ is in the image of $\Sk_1(\SS\smallsetminus S)$. Therefore, the map
\[ \Sk_1(\SS\smallsetminus S)[\partial^{-1}]/ \langle s'-s'', \forall s\in S\rangle \rightarrow \Sk_1(\SS)[\partial^{-1},S^{-1}] \]
is an isomorphism. This translates to cluster algebras as the desired isomorphism
\[ \mathcal{A}(\SS\smallsetminus S)/ \langle s'-s'', \forall s\in S\rangle \xrightarrow{\sim} \mathcal{A}(\SS)[S^{-1}] \qedhere\]
\end{proof}


The isomorphism from the proposition fits into the following diagram, in which the map on the left is the localization map and the map on the right is the quotient map.
\begin{equation}\label{eq: cuttingalgebras}
\begin{tikzcd}
	{\CA(\SS)} & {\CA(\SS\smallsetminus S)} \\
	{\CA(\SS)[S^{-1}]} & {\CA(\SS\smallsetminus S)/\langle s'-s'', \forall s\in S\rangle}
	\arrow["\sim"', from=2-2, to=2-1]
	\arrow[two heads, from=1-2, to=2-2]
	\arrow[hook, from=1-1, to=2-1]
\end{tikzcd}
\end{equation}
The isomorphism can be translated to cluster varieties via the following terminology.
\begin{itemize}
    \item A point in $V(\CA(\SS),\kk)$ is \textbf{non-zero on $S$} if it is non-zero on each arc in $S$.
    \item A point in $V(\CA(\SS\smallsetminus S),\kk)$ is \textbf{gluable along $S$} if it has the same value on both preimages of each edge in $S$.
\end{itemize}
The first type of points are in bijection with homomorphisms $\CA(\SS)[S^{-1}]\rightarrow \kk$ and the second type of points are in bijection with homomorphisms $\CA(\SS\smallsetminus S)/\langle s'-s'', \forall s\in S\rangle\rightarrow \kk$. Therefore, Proposition \ref{prop: cutting algebras} implies the following.

\begin{prop}\label{prop: cutting subvarieties}
Let $\SS$ be a triangulable marked surface, and let $S$ be a compatible collection of non-boundary marked arcs in $\SS$.
Then restricting values from marked curves in $\SS$ to marked curves in $\SS\smallsetminus S$ gives an isomorphism between:
\begin{enumerate}
    \item The open subvariety of $V(\CA(\SS),\kk)$ consisting of points which are non-zero on $S$.
    \item The closed subvariety of $V(\CA(\SS\smallsetminus S),\kk)$ consisting of points which are {gluable along $S$}.
\end{enumerate}
\end{prop}

\noindent We refer to this as the \textbf{cutting isomorphism} between the two subvarieties:
\begin{equation}\label{eq: cutting}
\{ \text{points in $V(\CA(\SS),\kk)$ non-zero on $S$}\} 
\xrightarrow{\sim}
\{\text{points in $V(\CA(\SS\smallsetminus S),\kk)$ gluable along $S$} \}
\end{equation}
Additionally, we refer to the action of the inverse map as \textbf{gluing} points.

\draftnewpage

\section{Deep points of unpunctured marked surfaces}
\label{section: deepmarkedsurfaces}

In this section, we classify and parameterize the deep points of unpunctured and triangulable marked surfaces. Our strategy is to reduce to the polygonal case by cutting along collections of arcs called \emph{polygonal dissections}.




\subsection{Vanishing lemma}

The first step in characterizing deep points of $\CA(\SS)$ is to characterize which marked arcs can simultaneously vanish. 
A \textbf{triangle of arcs} in $\SS$ is a compatible triple of marked arcs $\{x,y,z\}$ which bound a disk in $\SS$. The following lemma shows that, for connected surfaces, a deep point must kill an odd number of arcs in each triangle, generalizing Lemma \ref{lemma: trianglepoly}.





\begin{lemma}\label{lemma: trianglesurf}
Let $\SS$ be a connected, unpunctured, triangulable marked surface, let $p\in V(\CA(\SS),\kk)$, and let $\{x,y,z\}$ be a triangle of arcs in $\SS$.
\begin{enumerate}
    \item If $p$ kills any two of $\{x,y,z\}$, then $p$ also kills the third.
    \item If $p$ is deep, then $p$ kills at least one of $\{x,y,z\}$.
\end{enumerate}
\end{lemma}

\begin{proof}
(1) Assume that $p(x)=p(y)=0$; note that neither $x$ nor $y$ can be boundary arcs. Let $a$ be a boundary arc with an endpoint at the shared vertex of $x$ and $y$. Then $a,x,z$ form three of the four edges of a quadrilateral (Figure \ref{fig: quadsurf}); let $b$ be the other edge, and let $c$ be the other diagonal.\footnote{Despite the figure, the arcs $a,b,x,z$ need not be distinct.}

\begin{figure}[h!tb]
\begin{center}
\begin{tikzpicture}[inner sep=0.5mm,scale=.5,auto]
    \path[use as bounding box] (-2.25,-2.25) rectangle (2.25,2.25);
	\node (1) at (2,2) [circle,draw] {};
	\node (2) at (2,-2) [circle,draw] {};
	\node (3) at (-2,-2) [circle,draw] {};
	\node (4) at (-2,2) [circle,draw] {};
	\draw[red,thick] (4) to node[above] {$x$} (1);
	\draw[red,thick] (2) to node[pos=0.75, below=2pt] {$y$} (4);
	\draw[thick] (1) to node[right] {$z$} (2);
	\draw[dark green,thick] (3) to node[left] {$a$} (4);
	\draw[thick] (2) to node[below] {$b$} (3);
	\draw[thick] (1) to node[pos=0.25, below=2pt] {$c$} (3);
\end{tikzpicture}
\end{center}
\caption{Extending a triangle to a quadrilateral in $\SS$}
\label{fig: quadsurf}
\end{figure}
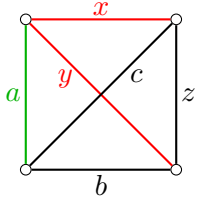

These arcs satisfy the skein relation $cy = az + bx$.
Applying $p$ to this relation gives 
\[ p(c)p(y) = p(a)p(z) + p(b)p(x) \]
Since $p(x)=p(y)=0$ and $p(a)\neq0$ (as $a$ is a boundary arc), this forces $p(z)=0$.

(2) Consider $p\in V(\CA(\SS),\kk)$ and assume that there is a triangle in $\SS$ with no vanishing arcs. 
We claim there is a triangulation of $T$ on which $p$ doesn't vanish. 
To show this, choose a triangulation $T$ of $\SS$ containing the non-vanishing triangle. Either $p$ does not vanish on all of $T$ (and we are done), or $p$ kills at least one arc in $T$.
%
%
%

Assume $p$ kills at least one arc in $T$.
Since $\SS$ is connected, we may find adjacent triangles $\{a,b,c\}$ and $\{c,d,e\}$ in $T$ such that $p$ does not kill any of $\{a,b,c\}$ and kills at least one of $\{c,d,e\}$. Since $p(c)\neq0$, part (1) implies that $p$ kills exactly one of $d$ and $e$; relabeling as needed, we assume that $p(d)\neq0$ and $p(e)=0$. 
Let $\{e,f,g\}$ be the other triangle in $T$ containing the arc $e$, and consider the pentagon in $\SS$ formed by these three triangles (Figure \ref{fig: first triangulation}); let $h,i,j$ be the other arcs in the figure (as in Figure \ref{fig: second triangulation}).
We apply part (1) to two triangles in this pentagon.
\begin{itemize}
    \item Consider the triangle $\{a,e,h\}$. Since $p(a)\neq0$ and $p(e)=0$, $p(h)\neq0$.
    \item Consider the triangle $\{b,i,j\}$. Since $p(b)\neq0$, at least one of $p(i),p(j)$ is not $0$.
\end{itemize}
Therefore, either $p(c) p(i)\neq0$ or $p(h) p(j)\neq0$. If we replace the arcs $\{c,e\}$ in $T$ with either $\{c,i\}$ or $\{h,j\}$, we obtain a triangulation of $\SS$ in which $p$ kills one fewer arc than $T$. Iterating this argument produces a triangulation of $\SS$ on which $p$ does not vanish, and so $p$ cannot be deep.
%
%
\end{proof}


\begin{figure}[h!tb]
\captionsetup[subfigure]{justification=centering}
\centering
\begin{subfigure}[c]{0.35\linewidth}
\centering
\begin{tikzpicture}[inner sep=0.5mm,scale=1.25,auto]
	\node (1) at (90:1) [circle,draw] {};
	\node (2) at (90-72:1) [circle,draw] {};
	\node (3) at (90-2*72:1) [circle,draw] {};
	\node (4) at (90-3*72:1) [circle,draw] {};
	\node (5) at (90-4*72:1) [circle,draw] {};
	\draw[dark green,thick] (2) to node[right] {$a$} (3);
	\draw[dark green,thick] (3) to node[below] {$b$} (4);
	\draw[dark green,thick] (2) to node[below right] {$c$} (4);
	\draw[dark green,thick] (4) to node[left] {$d$} (5);
	\draw[dark red,thick] (2) to node[above] {$e$} (5);
	\draw[thick] (5) to node[above left] {$f$} (1);
	\draw[thick] (1) to node[above right] {$g$} (2);
\end{tikzpicture}
\subcaption{Part of a triangulation\\in which $e$ vanishes}
\label{fig: first triangulation}
\end{subfigure}
\begin{subfigure}[c]{0.55\linewidth}
\centering
\begin{tikzpicture}[inner sep=0.5mm,scale=1.25,auto]
	\node (1) at (90:1) [circle,draw] {};
	\node (2) at (90-72:1) [circle,draw] {};
	\node (3) at (90-2*72:1) [circle,draw] {};
	\node (4) at (90-3*72:1) [circle,draw] {};
	\node (5) at (90-4*72:1) [circle,draw] {};
	\draw[dark green,thick] (2) to node[right] {$a$} (3);
	\draw[dark green,thick] (3) to node[below] {$b$} (4);
	\draw[dark green,thick] (4) to node[left] {$d$} (5);
	\draw[thick] (5) to node[above left] {$f$} (1);
	\draw[thick] (1) to node[above right] {$g$} (2);
	\draw[dark green,thick] (4) to node[left] {$i$} (1);
	\draw[dark green,thick] (4) to node[below right] {$c$} (2);
\end{tikzpicture}
\hspace{.5cm}
\begin{tikzpicture}[inner sep=0.5mm,scale=1.25,auto]
	\node (1) at (90:1) [circle,draw] {};
	\node (2) at (90-72:1) [circle,draw] {};
	\node (3) at (90-2*72:1) [circle,draw] {};
	\node (4) at (90-3*72:1) [circle,draw] {};
	\node (5) at (90-4*72:1) [circle,draw] {};
	\draw[dark green,thick] (2) to node[right] {$a$} (3);
	\draw[dark green,thick] (3) to node[below] {$b$} (4);
	\draw[dark green,thick] (4) to node[left] {$d$} (5);
	\draw[thick] (5) to node[above left] {$f$} (1);
	\draw[thick] (1) to node[above right] {$g$} (2);
	\draw[dark green,thick] (3) to node[below left] {$h$} (5);
	\draw[dark green,thick] (3) to node[right] {$j$} (1);
\end{tikzpicture}
\subcaption{Two possible modifications, one of\\which must have fewer vanishing arcs.}
\label{fig: second triangulation}
\end{subfigure}
\caption{Modifying a triangulation to reduce the number of vanishing arcs.}
\label{fig: triangulation modification}
\end{figure}
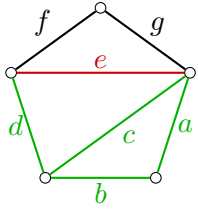
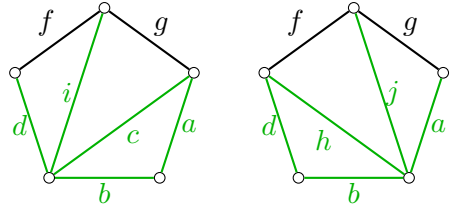

\begin{warn}
Unlike the polygonal case, this lemma is not enough to show that every deep point kills the same set of arcs (as in Lemma \ref{lemma: evenrule}); in fact, different deep points of $\CA(\SS)$ can kill different sets of arcs. 
This results in the deep locus of $\CA(\SS)$ having multiple components (Theorem \ref{thm: deepsurfacediss}).
%
%
\end{warn}

\subsection{Polygonal dissections}


To reduce to the polygonal case, we want to find collections of arcs in $\SS$ whose cutting is a polygon.


\begin{prop}\label{prop: polydiss}
Let $\SS$ be a connected, triangulable marked surface. Then there exists a compatible collection of marked arcs $D$ such that the cutting $\SS\smallsetminus D$ (as defined in Section \ref{section: cutting}) is a polygon.
\end{prop}
Let us call such a set $D$ of arcs a \textbf{polygonal dissection} of $\SS$. 

\begin{proof}
Choose a triangulation of $\SS$. The dual graph of such a triangulation is connected, and therefore it admits a spanning tree. Choose a spanning tree, and let $D$ denote the set of arcs in the triangulation which are not in the spanning tree. Then the triangulation descends to a triangulation of $\SS\smallsetminus D$ whose dual graph is a tree; therefore, $\SS\smallsetminus D$ is a polygon.
\end{proof}


\begin{prop}\label{prop: polydisscounts}
Let $\SS$ be a connected, triangulable marked surface with genus $g$, $b$-many boundary components, $p$-many punctures, and $m$-many boundary marked points. Then every polygonal dissection of $\SS$ cuts along 
\[
d(\SS) \coloneqq 2g + b + p - 1
\]
many arcs and produces a 
\[
\delta(\SS) \coloneqq 4g + 2b + 2p + m - 2 
\]
sided polygon.
\end{prop}

\begin{proof}
Let $D$ be a polygonal dissection of $\SS$ and choose a triangulation $T$ of $\SS$ such that $D\subset T$. We can compute the Euler characteristic of $\SS$ as $\chi(\SS) = V-E+F = (m+p)-|T|+F$. From Proposition \ref{prop: triangulationcount}, a triangulation of $\SS$ consists of $|T|= n(\SS) = 6g + 3b + 3p + 2m - 6$ arcs, so $F=\chi(\SS)+n(\SS)-m-p = \chi(\SS) + (6g + 3b + 2p + m - 6)$. Using the Euler characteristic formula $\chi (\SS) = 2 - 2g - b$, we have $F= (2-2g-b) + (6g + 3b + 2p + m - 6) = 4g+2b+2p +m -4 $. 

Observe that a triangulation of the polygon $\SS\smallsetminus D$ has $n(\SS)+d(\SS)$ arcs which divide $\SS\smallsetminus D$ into $F$ triangles. Hence, the polygon has $\delta(\SS) = F+2 = 4g+2b+2p+m-2$ vertices and sides. A triangulation of a polygon of $F+2$ sides will have $2(F+2)-3$ arcs, so 
solving the equation $n(\SS)+d(\SS) = 2(F+2)-3$, we find that the number of arcs in a polygonal dissection is $d(\SS) = 2g+b+p-1$.
\end{proof}

In what follows, we will need more than the existence of polygonal dissections; we need polygonal dissections on which any given point of $\CA(\SS)$ is non-zero. Since a point in $V(\CA(\SS),\kk)$ can never kill exactly two of the three arcs in a triangle (Lemma \ref{lemma: trianglesurf}), it will suffice to prove the following.

\begin{lemma}\label{lemma: polydissavoid}
Let $\SS$ be a connected, unpunctured, triangulable marked surface, and let $\cV$ be a set of simple, non-boundary marked arcs which does not contain exactly two of the three arcs in any triangle in $\SS$. Then there exists a polygonal dissection $D$ of $\SS$ consisting of arcs not in $\cV$.
\end{lemma}

\begin{proof}
Choose a polygonal dissection $D_0$ of $\SS$, and suppose there is some arc $p \in D_0\cap \cV$. 
Because $\cV$ does not contain boundary arcs in $\SS$, some but not all edges in the boundary of $\SS\smallsetminus D_0$ come from arcs in $\cV$. Choose $\alpha\in \cV$ and $\beta\not\in \cV$ to be arcs in $\SS$ which become adjacent edges in the boundary of $\SS\smallsetminus D_0$, and let $\gamma$ be the arc in $\SS$ that completes a triangle with $\alpha$ and $\beta$ in $\SS\smallsetminus D_0$. By the assumption on triangles, we know that $\gamma\not\in \cV$. 
Then $D_1 \coloneqq (D_0 \cup \{\gamma\})\smallsetminus \{\alpha \}$ is a polygonal dissection of $\SS$ containing strictly fewer arcs in $\cV$ than $D_0$.
Repeated applications of this process will eventually produce a polygonal dissection $D$ disjoint from $\cV$.
\end{proof}

The special case of $\cV\coloneqq\{a \text{ such that }p(a)=0\}$ is the following.

\begin{coro}\label{coro: deeppolydiss}
Let $\SS$ be a connected, unpunctured, triangulable marked surface. For every point $p\in V(\CA(\SS),\kk)$, there is a polygonal dissection $D$ of $\SS$ on which $p$ is non-zero.

\end{coro}

This can be interpreted geometrically as follows.
By Proposition \ref{prop: cutting subvarieties}, each polygonal dissection $D$ determines an open embedding
\begin{equation}\label{eq: polydissinclusion}
\{ \text{points in $V(\CA(\SS\smallsetminus D),\kk)$ gluable along $D$}\}\hookrightarrow V(\CA(\SS),\kk)
\end{equation}
The corollary is equivalent to the fact that these open charts collectively cover the cluster variety.






\subsection{Cutting and gluing deep points}

We now know that every point in the cluster variety $V(\CA(\SS),\kk)$ is non-zero on some polygonal dissection, and is therefore contained in an open neighborhood given by gluing the cluster variety of a polygon (Equation \eqref{eq: polydissinclusion}). Since Theorem \ref{thm: deeppolygon} gives a characterization of the deep points in the latter, we need to understand how cutting and gluing affect deepness.
The immediate result is the following.



\begin{prop}
The cutting isomorphism restricts to an inclusion on deep points; that is,
\[ 
\{ \text{deep points in $V(\CA(\SS),\kk)$ non-zero on $S$}\} 
\hookrightarrow
\{\text{deep points in $V(\CA(\SS\smallsetminus S),\kk)$ gluable along $S$} \}
\]
\end{prop}

\begin{proof}
Let $p$ be a deep point of $V(\CA(\SS),\kk)$ which is non-zero on $S$, and let $p'$ be its image under the cutting isomorphism. Any triangulation of $\SS\smallsetminus S$ is the image of a triangulation of $\SS$ which contains $S$. Since $p$ is deep, it must vanish on at least one arc in the triangulation of $\SS\smallsetminus S$, and so $p'$ must vanish on at least one arc in the triangulation of $\SS$.
\end{proof}



Unfortunately, this restriction is not always a bijection! As the following example shows, the cutting of a non-deep point may become deep.

\begin{ex}\label{ex: badcut}
For any non-zero $a\in \kk$, the values depicted in Figure \ref{fig: badcut} uniquely determine a point in $V(\CA(\Delta_6),\kk)$ which is not deep (since no arc in the given triangulation vanishes), but which becomes deep when cut along a certain diagonal (since any triangulation of the disconnected surface must contain one of the two vanishing diagonals in the right-hand component). 
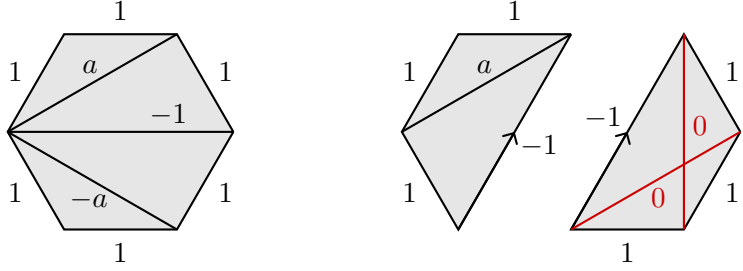
\begin{figure}[h!t]
\[
\begin{tikzpicture}[xscale=1,scale=.75,baseline={(0,0)}]
    \draw[fill=black!10,thick] 
    (0*60:2)
    to (1*60:2)
    to (2*60:2)
    to (3*60:2)
    to (4*60:2)
    to (5*60:2)
    to (6*60:2)
    ;
    \node (a1) at (30+0*60:2.15) {$1$};
    \node (a2) at (30+1*60:2.15) {$1$};
    \node (a3) at (30+2*60:2.15) {$1$};
    \node (a4) at (30+3*60:2.15) {$1$};
    \node (a5) at (30+4*60:2.15) {$1$};
    \node (a6) at (30+5*60:2.15) {$1$};
    \draw[thick] (3*60:2) to (1*60:2);
    \node at (115:1.27) {$a$};
    \draw[thick] (3*60:2) to (0*60:2);
    \node at (20:0.9) {$-1$};
    \draw[thick] (3*60:2) to (5*60:2);
    \node at (245:1.33) {$-a$};
\end{tikzpicture}
\hspace{2cm}
\begin{tikzpicture}[xscale=1,scale=.75,baseline={(0,0)}]
    \begin{scope}[xshift=-1cm]
    \draw[fill=black!10,thick] 
    (1*60:2)
    to (2*60:2)
    to (3*60:2)
    to (4*60:2)
    to (1*60:2)
    ;
    \draw[-angle 90,thick] (4*60:2) to (0,0);
    \node (a6) at (30+5*60:.5) {$-1$};
    \node (a2) at (30+1*60:2.15) {$1$};
    \node (a3) at (30+2*60:2.15) {$1$};
    \node (a4) at (30+3*60:2.15) {$1$};
    \draw[thick] (3*60:2) to (1*60:2);
    \node at (115:1.27) {$a$};
    \end{scope}
    \begin{scope}[xshift=1cm]
    \draw[fill=black!10,thick] 
    (0*60:2)
    to (1*60:2)
    to (4*60:2)
    to (5*60:2)
    to (6*60:2)
    ;
    \draw[-angle 90,thick] (4*60:2) to (0,0);
    \node (a1) at (30+0*60:2.15) {$1$};
    \node (a3) at (30+2*60:.5) {$-1$};
    \node (a5) at (30+4*60:2.15) {$1$};
    \node (a6) at (30+5*60:2.15) {$1$};
    \draw[thick,dark red] (5*60:2) to (1*60:2);
    \node[dark red] at (365:1.29) {$0$};
    \draw[thick,dark red] (4*60:2) to (0*60:2);
    \node[dark red] at (295:1.29) {$0$};
    \end{scope}
\end{tikzpicture}
\]
\caption{A non-deep point (left) which becomes deep when cut (right)}
\label{fig: badcut}
\end{figure}
%
%
%
%
\end{ex}


The problem is that deep points of a disconnected surface behave counterintuitively. If $\SS$ is a disjoint union of two components $\SS_1$ and $\SS_2$, then a point $p$ of $V(\CA(\SS),\kk)$ is equivalent to a pair of points $p_1$ and $p_2$ of $V(\CA(\SS_1),\kk)$ and $V(\CA(\SS_2),\kk)$, respectively; that is,
\[ V(\CA(\SS), \kk) \simeq V(\CA(\SS_1),\kk) \times V(\CA(\SS_2),\kk) \]
Since a triangulation of $\SS$ is a union of triangulations of $\SS_1$ and $\SS_2$, we see that a point $p$ of $V(\CA(\SS),\kk)$ is deep iff \emph{either} of the points $p_1$ or $p_2$ are deep. 

We introduce a definition which behaves better for disconnected surfaces and which won't be used beyond the following lemma. Let us say a point $p$ in $V(\CA(\SS),\kk)$ is \textbf{deep-on-components} if the restriction of $p$ to each connected component of $\SS$ is deep; that is, it kills a marked arc in each triangulation of each connected component. Every deep-on-components point is deep, but the converse fails (for example, the deep point on the right side of Figure \ref{fig: badcut} is not deep-on-components).

Many prior results extend to disconnected surfaces by replacing \emph{deep} with \emph{deep-on-components}. 
For example, applying Lemma \ref{lemma: trianglesurf} to each connected component of $\SS$, we see that a point $p$ of $V(\CA(\SS),\kk)$ is deep-on-components if and only if it kills an odd number of arcs in any triangle in $\SS$.





\begin{lemma}\label{lemma: deepcut}
Let $\SS$ be an unpunctured, triangulable marked surface. For any compatible collection of non-boundary arcs $S$ in $\SS$, the cutting isomorphism \eqref{eq: cutting} restricts to a bijection between:
\begin{enumerate}
    \item The deep-on-components points of $V(\CA(\SS),\kk)$ which are non-zero on $S$.
    \item The deep-on-components points of $V(\CA(\SS\smallsetminus S),\kk)$ which are gluable along $S$.
\end{enumerate}
\end{lemma}


\begin{proof}
Fix a point $p\in V(\CA(\SS),\kk)$ which is non-zero on $S$, and let $p'\in V(\CA(\SS\smallsetminus S),\kk)$ be its image under the cutting isomorphism.

$(1)\Rightarrow (2)$. If $p$ is deep-on-components, then it kills at least one marked arc in every triangle of $\SS$ (by Lemma \ref{lemma: trianglesurf} for each connected component). 
If $\SS'$ is a connected component of $\SS\smallsetminus S$, then $p'$ kills at least one marked arc in every triangle in $\SS'$; in particular, it kills at least one marked arc in every triangulation of $\SS'$. Therefore, $p'$ is deep-on-components.

$(2)\Rightarrow (1)$. Assume that $p'$ is deep-on-components, and consider various cases.

First, assume $\SS$ is a polygon; note that this forces $\SS\smallsetminus S$ to be a disjoint union of polygons. 
Index the elements of $S$ as $\alpha_1,\alpha_2,\dotsc ,\alpha_n$. For any $i$ between $0$ and $n$, let $V_i$ denote the set of points in $V(\CA(\SS\smallsetminus \{\alpha_1,\alpha_2,\dotsc ,\alpha_i\}),\kk)$ which are gluable along $\{\alpha_1,\alpha_2,\dotsc ,\alpha_i\}$ and non-zero on $\{\alpha_{i+1},\alpha_{i+2},\dotsc ,\alpha_n\}$. Then the cutting isomorphism \eqref{eq: cutting} factors as a composition of isomorphisms
\[ V_0 \xrightarrow{\sim} V_1 \xrightarrow{\sim} \dotsm \xrightarrow{\sim} V_{n-1} \xrightarrow{\sim} V_n \]
where each $V_i\rightarrow V_{i+1}$ is (a restriction of) the cutting isomorphism along $\alpha_{i+1}$. Each of these cuts divides a polygon into two smaller polygons; by Corollary \ref{coro: polygoncut}, a point which is deep on both smaller polygons corresponds to a point which is deep on their gluing. As a consequence, a deep-on-components point in $V_{i+1}$ glues to a deep-on-components point in $V_i$. Iterating along the composition above, the deep-on-components point $p'$ glues to a deep-on-components point $p$.

Next, assume that $\SS$ is connected and $\SS\smallsetminus S$ is a disjoint union of polygons. Then the universal cover $\widehat{\SS}$ of $\SS$ can be covered by the lifts of these polygons. Consider three marked arcs $\{\alpha,\beta,\gamma\}$ which bound a triangle in $\SS$, and choose lifts $\{\widehat{\alpha},\widehat{\beta},\widehat{\gamma}\}$ which bound a triangle in $\widehat{\SS}$. Since this triangle is compact, it is contained in the union of finitely many lifts of components in $\SS\smallsetminus S$. Let $\SS'$ denote this union; since it is a finite union of polygons in a simply connected space, it is also a polygon.
Since the point $p'$ is deep-on-components, it gives a deep point on each of the lifts of these components. By the previous case, these deep points glue together to a deep point $\widehat{p}$ on the polygon $\SS'$. By Lemma \ref{lemma: trianglepoly}, an odd number of $\widehat{p}(\widehat{\alpha}),\widehat{p}(\widehat{\beta}),\widehat{p}(\widehat{\gamma})$ must vanish, and so an odd number of $p(\alpha),p(\beta),p(\gamma)$ must vanish. Therefore, $p$ kills an odd number of marked arcs in every triangle, and so it is deep-on-components.

Next, assume that $\SS$ is connected. By Corollary \ref{coro: deeppolydiss}, there are polygonal dissections of each connected component of $\SS\smallsetminus S$ on which $p'$ is non-zero. 
By redefining $S$ to be the union of the original $S$ together with the arcs in these polygonal dissections, we reduce to the previous case.

Finally, if $\SS$ is disconnected, then we may apply the previous case to each component of $\SS$ and deduce that $p$ is deep-on-components.
\end{proof}

In what follows, we only need the specialization of Lemma \ref{lemma: deepcut} in the case of polygonal dissections.

\begin{coro}\label{coro: deepcutpoly}
Let $\SS$ be a connected, unpunctured, triangulable marked surface. For any polygonal dissection $D$ in $\SS$, the cutting isomorphism \eqref{eq: cutting} restricts to a bijection between:
\begin{enumerate}
    \item The deep points in $V(\CA(\SS),\kk)$ which are non-zero on $D$.
    \item The deep points in $V(\CA(\SS\smallsetminus D),\kk)$ which are gluable along $D$.
\end{enumerate}
\end{coro}

\subsection{Comparing polygonal dissections}

While Corollary \ref{coro: deepcutpoly} allows us to construct deep points of $\CA(\SS)$ by gluing deep points of polygons, a single deep point of $\CA(\SS)$ may be non-zero on many polygonal dissections, and can therefore be constructed via many different gluings. 

In this section, we use relative cohomology classes in $H^1(\SS,\MM;\mathbb{Z}_2)$ to characterize when a deep point is non-zero on a polygonal dissection, and when two polygonal dissections determine the same set of deep points.
Note that every marked arc in $\SS$ is a 1-cycle relative to the marked points $\MM$, and so elements of $H^1(\Sigma,\MM;\mathbb{Z}_2)$ have a well-defined value on each marked arc in $\SS$.

\begin{prop}\label{prop: cohomologyclass}
Let $\SS$ be a connected, unpunctured, triangulable marked surface with an even number of marked points.
\begin{enumerate}
    \item Given a deep point $p$ of $V(\CA(\SS),\kk)$, there is a unique cohomology class
    $
    \chi_p\in H^1(\SS,\MM;\mathbb{Z}_2)
    $
    such that the value of $\chi_p$ on each simple marked arc $a$ in $\SS$ is
    \begin{equation}\label{eq: cohomologyclass}
    \chi_p(a) \coloneqq \left\{
    \begin{array}{cc}
    0 & \text{if $p(a)=0$} \\
    1 & \text{if $p(a)\neq0$}
    \end{array}
    \right\}
    \end{equation}
    \item Given a polygonal dissection $D$ of $\Sigma$, there is a unique cohomology class 
    $
    \chi_D\in H^1(\SS,\MM;\mathbb{Z}_2)
    $
    such that $\chi_D$ is non-zero on each marked arc in $D$ and each boundary arc of $\SS$.
    \item A deep point $p$ is non-zero on a polygonal dissection $D$ if and only if $\chi_p=\chi_D$.
\end{enumerate}
\end{prop}

\begin{proof}
A (marked) triangulation $T$ of $\SS$ determines a simplicial decomposition of $\SS$ whose 0-cells are the marked points $\MM$, whose 1-cells are the marked arcs in $T$, and whose 2-cells are the triangles in $\SS\smallsetminus T$. Since singular and simplicial cohomology are isomorphic in this case, we may compute the cohomology groups $H^\bullet(\SS,\MM;\mathbb{Z}_2)$ as the cohomology of the reduced simplicial cochain complex 
\[
\underbrace{\mathbb{Z}_2^\MM / \mathbb{Z}_2^\MM}_{\simeq0} \rightarrow \mathbb{Z}_2^T \rightarrow \mathbb{Z}_2^{\SS\smallsetminus T}
\]
As a consequence, elements of $H^1(\SS,\MM;\mathbb{Z}_2)$ may be realized as functions $T\rightarrow \mathbb{Z}_2$ such that the sum of the values around each triangle is 0 (the 1-cocycle condition). In particular:

\begin{itemize}
    \item Given a deep point $p$ of $V(\CA(\SS),\kk)$, choose a triangulation $T$ and define a function $\chi_p: T\rightarrow\mathbb{Z}_2$ by applying Formula \eqref{eq: cohomologyclass} to each marked arc in $T$. The sum around each triangle in $T$ is $0$ by Lemma \ref{lemma: trianglesurf}, so this defines an element $\chi_p \in H^1(\SS,\MM;\mathbb{Z}_2)$. 

    Given a non-boundary arc $a$ in $T$, let $a'$ be the \emph{flip} of $a$; that is, the unique-up-to-homotopy marked arc such that $T\smallsetminus \{a\} \cup \{a'\}$ is a triangulation and $a$ is not homotopic to $a'$. Choose arcs $b,c\in T$ such that $a',b,c$ bound a triangle. By Lemma \ref{lemma: trianglesurf}, $p(a')=0$ iff an even number of $p(b)$ and $p(c)$ are $0$; similarly, by the cocycle condition, $\chi_p(a')=0$ iff an even number of $\chi_p(b)$ and $\chi_p(c)$ are $0$. Therefore, $\chi_p(a')=0$ iff $p(a)=0$, so Formula \eqref{eq: cohomologyclass} holds for $a'$. Since any simple marked arc in $\SS$ can be realized by a sequence of flips, Formula \eqref{eq: cohomologyclass} holds for all simple marked arcs in $\SS$.

    \item Consider a polygonal dissection $D$ of $\SS$. For any triangulation $T$ of $\SS$ which contains $D$, define a function $\chi_D: T\rightarrow \mathbb{Z}_2$ by the rule that, for each $a$ in $T$,
    \[
    \chi_D(a) \coloneqq 
    \left\{
    \begin{array}{cc}
    0 & \text{if $\SS\smallsetminus (D\cup \{a\})$ consists of two odd-sided polygons} \\
    1 & \text{if $\SS\smallsetminus (D\cup \{a\})$ consists of two even-sided polygons} \\
    1 & \text{if $a$ is in $D$ or $\partial \SS$} \\
    \end{array}
    \right\}
    \]
    It is straight-forward to show this satisfies the 1-cocycle condition (for example, by induction on the size of the smaller polygon in $\SS\smallsetminus (D\cup \{a\})$), so it defines an element in $H^1(\SS,\MM;\mathbb{Z}_2)$.
\end{itemize}
Since each $\chi_p$ is non-zero on every boundary arc in $\SS$, the uniqueness condition in (2) implies that $\chi_p=\chi_D$ if and only if $\chi_p$ is non-zero on the arcs in $D$; equivalently, if $p$ is non-zero on $D$.
\end{proof}

As a consequence, two polygonal dissections can be used to construct the same deep points if and only if they have the same cohomology class. We use this to define an equivalence relation.

\begin{defn}\label{defn: congruent}
Let $\SS$ be as in Proposition \ref{prop: cohomologyclass}.
Then two polygonal dissections of $\SS$ are \textbf{congruent} if their relative cohomology classes are equal.
\end{defn}

\begin{prop}\label{prop: polydisscount}
Let $\SS$ be a connected, unpunctured, triangulable marked surface with an even number of marked points. Then the function $D\mapsto \chi_D$ defines a bijection between
\begin{enumerate}
    \item polygonal dissections $D$ of $\SS$ up to congruence, and
    \item elements $\chi$ in $H^1(\SS,\MM;\mathbb{Z}_2)$ with value $1$ on every boundary marked arc in $\SS$.
\end{enumerate}
The size of each set is $2^{2g+b-1}$, where $g$ is the genus and $b$ is the number of boundary components.
\end{prop}

\begin{proof}
The function $D\mapsto \chi_D$ is injective by the definition of congruence. Given $\chi$ satisfying (2), the set of simple marked arcs killed by $\chi$ contains an odd number of arcs in each triangle. By Lemma \ref{lemma: polydissavoid}, there exists a polygonal dissection $D$ on which $\chi$ is non-zero; by the uniqueness in Proposition \ref{prop: cohomologyclass}.2, $\chi_D=\chi$. Therefore, the function $D\mapsto \chi_D$ is surjective.


For cardinality, consider (part of) the long exact sequence induced by the triple $(\SS,\partial \SS,\MM)$.
\[
\dotsb
\rightarrow
H^0(\partial\SS,\MM;\mathbb{Z}_2)
\rightarrow
H^1(\SS,\partial \SS;\mathbb{Z}_2)
\xrightarrow{f}
H^1(\SS,\MM;\mathbb{Z}_2)
\xrightarrow{g}
H^1(\partial \SS,\MM;\mathbb{Z}_2)
\rightarrow
\dotsb
\]
Since every component of $\partial \SS$ contains a marked point, $H^0(\partial\SS,\MM;\mathbb{Z}_2)=0$ and so the map labeled $f$ above is injective.
The map labeled $g$ above sends $\chi$ to its values on the boundary arcs, and so the set of relative cohomology classes coming from polygonal dissections is precisely the preimage of the element in $H^1(\partial \SS,\MM;\mathbb{Z}_2)$ which sends each boundary arc to $1$. Since this preimage is non-empty, it has the same cardinality of as the preimage of $0$; that is, the kernel of $g$. Since $f$ is injective and the sequence is exact, the kernel of $g$ has the same cardinality as $H^1(\SS,\partial \SS;\mathbb{Z}_2)$.

The dimension of this relative cohomology group can be computed via Euler characteristics.
\[
\begin{aligned}
\sum_{i=0}^2 (-1)^i\dim_{\mathbb{Z}_2}(H^i(\SS,\partial \SS;\mathbb{Z}_2))
&=
\sum_{i=0}^2 (-1)^i\dim_{\mathbb{Z}_2}(H^i(\SS;\mathbb{Z}_2))
-
\sum_{i=0}^2 (-1)^i\dim_{\mathbb{Z}_2}(H^i(\partial\SS;\mathbb{Z}_2)) \\
=
(2-2g-b) + 0 
\end{aligned}
\]
Since $\SS$ is connected, $\dim(H^0(\SS,\partial \SS;\mathbb{Z}_2))=0$ and $\dim(H^2(\SS,\partial \SS;\mathbb{Z}_2))=1$. Therefore, 
\[
\dim_{\mathbb{Z}_2}(H^1(\SS,\partial \SS;\mathbb{Z}_2)) = 2g+b-1
\] 
and so the cardinality of $H^1(\SS,\partial\SS;\mathbb{Z}_2)$ is $2^{2g+b-1}$.
\end{proof}

\begin{rem}
Congruence classes of polygonal dissections are in bijection with several other sets of potential interest; we explore these different interpretations in Appendix \ref{appendix: indexingdeepcomponents}.
\end{rem}

\subsection{Deep points of marked surfaces}



We can now combine our results to characterize and parameterize the deep locus of $V(\CA(\SS),\kk)$.

\begin{thm}\label{thm: deepsurfacediss}
Let $\SS$ be a connected, unpunctured, triangulable marked surface with genus $g$, $b$-many boundary components, and $m$-many marked points.
\begin{enumerate}
    \item If $m$ is odd, then $V(\CA(\SS),\kk)$ has no deep points.
    \item If $m$ is even, then each deep point of $V(\CA(\SS),\kk)$ is non-zero on some polygonal dissection. 
    \begin{enumerate}
       \item 
        Given a polygonal dissection $D$ of $\SS$, a function
        \[ p:\{\text{boundary arcs in $\SS$}\} \cup D \rightarrow \kk^\times\]
        extends to a (necessarily unique) deep point of $V(\CA(\SS),\kk)$ which is non-zero on $D$ iff the alternating product of the values of $p$ around the boundary of $\SS\smallsetminus D$ is $(-1)^{b+\frac{m}{2}}$.
        \item Two polygonal dissections of $\SS$ parameterize the same deep points iff they are congruent (see Definition \ref{defn: congruent}); otherwise they parameterize disjoint sets of deep points.
    \end{enumerate}
     Therefore, the deep locus of $V(\CA(\SS),\kk)$ consists of $2^{2g+b-1}$-many disjoint copies of \\$(\kk^\times)^{2g+b+m-2}$.
\end{enumerate}
\end{thm}

\begin{proof}
Let $D$ be a polygonal dissection of $\SS$. By Corollary \ref{coro: deepcutpoly}, deep points of $V(\CA(\SS),\kk)$ which are non-zero on $D$ are in bijection with deep points of $V(\CA(\SS\smallsetminus D),\kk)$ which are gluable along $D$. 
By Proposition \ref{prop: polydisscounts}, $D$ contains $(2g+b-1)$-many arcs, and $\SS\smallsetminus D$ is a polygon with $(4g+2b+m-2)$-many sides. By Theorem \ref{thm: deeppolygon}, $V(\CA(\SS\smallsetminus D),\kk)$ has no deep points if $(4g+2b+m-2)$ is odd, or equivalently, if $m$ is odd. This verifies Part (1) of the theorem.

If $m$ is even, then $V(\CA(\SS\smallsetminus D),\kk)$ has a unique deep point for each choice of values on the boundary of $\SS\smallsetminus D$, subject to the product in Condition \eqref{eq: alternating product}. 
If we fix a boundary arc $\alpha$ in $\SS$, then we may freely choose the values of $p$ on the $(m-1)$-many other boundary arcs and the $(2g+b-1)$-many arcs in $D$ to be any non-zero numbers in $\kk$. Since $p(\alpha)$ will appear once in the alternating product, there is a unique non-zero value which makes Condition \eqref{eq: alternating product} true.
Therefore, a deep point in $V(\CA(\SS),\kk)$ which is non-zero on $D$ is freely determined by the choice of $(2g+b+m-2)$-many non-zero elements in $\kk$. This verifies part (2a) and the shape of each deep component.

If two polygonal dissections have any deep points in common, they must have the same cohomology class by Proposition \ref{prop: cohomologyclass}, in which case all of their deep points are the same. This verifies Part (2b). Since there are $2^{2g+b-1}$-many equivalence classes of polygonal dissections (Proposition \ref{prop: polydisscount}), this will be the number of components of the deep locus.
%
\end{proof}

\begin{ex}\label{ex: 2 2 annulus 1}
Let $\SS$ be the $(2,2)$-annulus, the genus $g=0$ surface with $b=2$ boundary components and $2$ marked points on each boundary, for a total of $m=4$ marked points. 
Figure \ref{fig: annulus 2 2 identified sides deep} depicts $\SS$ as a rectangle with opposite sides identified (drawn dashed).
The marked points divide each boundary component into two boundary arcs, which we label $a_1, a_2$ and $c_1,c_2$ as in Figure \ref{fig: annulus 2 2 deep}.

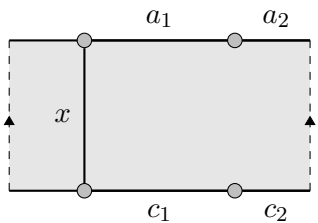
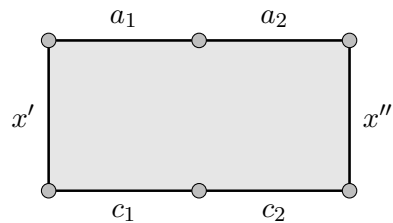
\begin{figure}[h!tb]
\captionsetup[subfigure]{justification=centering}
\centering
\begin{subfigure}[c]{0.45\linewidth}
\centering
\begin{tikzpicture}
    \path[fill=black!10] (-2,1) to (2,1) to (2,-1) to (-2,-1) to (-2,1);
    \draw[thick] (-2,1) to (2,1) (-2,-1) to (2,-1);
    \draw[dashed] (-2,-1) to (-2,1) (2,-1) to (2,1);
    \draw[-Triangle,thin] (-2,-.1) to (-2,0);
    \draw[-Triangle,thin] (2,-.1) to (2,0);
    
    \node[dot] (1) at (-1,1) {};
    \node[dot] (2) at (1,1) {};
    \node[dot] (3) at (1,-1) {};
    \node[dot] (4) at (-1,-1) {};
    \draw[thick] (1) to node [midway, left=1pt] {$x$} (4);
    \draw[thick] (1) to node [midway, above=1pt] {$a_1$} (2);
    \draw[thick] (2) to node [midway, above=1pt] {$a_2$} (2,1);
    \draw[thick] (3) to node [midway, below=1pt] {$c_1$} (4);
    \draw[thick] (3) to node [midway, below=1pt] {$c_2$} (2,-1);

\end{tikzpicture}
\subcaption{The $(2,2)$-annulus $\SS$\\with marked arc $x$}
\label{fig: annulus 2 2 identified sides deep}
\end{subfigure}
\begin{subfigure}[c]{0.45\linewidth}
\centering
\begin{tikzpicture}
    \draw[fill=black!10,thick] (-3,1) to (1,1) to (1,-1) to (-3,-1) to (-3,1);
    \node[dot] (1) at (-3,1) {};
    \node[dot] (2) at (-1,1) {};
    \node[dot] (3) at (1,1) {};
    \node[dot] (4) at (1,-1) {};
    \node[dot] (5) at (-1,-1) {};
    \node[dot] (6) at (-3,-1) {};
    \draw[thick] (2) to node [midway, above=1pt] {$a_2$} (3);
    \draw[thick] (1) to node [midway, above=1pt] {$a_1$} (2);
    \draw[thick] (4) to node [midway, below=1pt] {$c_2$} (5);
    \draw[thick] (6) to node [midway, below=1pt] {$c_1$} (5);
    \draw[thick] (6) to node [midway, left=1pt]  {$x'$} (1);
    \draw[thick] (4) to node [midway, right=1pt] {$x''$} (3);
\end{tikzpicture}
\subcaption{A hexagon $\Delta$ formed by cutting $\SS$ \\along the polygonal dissection $D=\{x\}$.}
\label{fig: annulus 2 2 poly diss deep}
\end{subfigure}
\caption{A polygonal dissection of the (2,2)-annulus.}
   \label{fig: annulus 2 2 deep}
\end{figure}

The theorem predicts that the deep locus of $V(\CA(\SS),\kk)$ consists of $2$ copies of $(\kk^{\times})^4$; we show this now.
Consider the polygonal dissection of $\SS$ consisting of the marked arc $x$ in Figure \ref{fig: annulus 2 2 identified sides deep}, whose cutting is the hexagon $\Delta$ in Figure \ref{fig: annulus 2 2 poly diss deep} (note that the arc $x$ becomes boundary arcs $x'$ and $x''$ in $\Delta$). By Theorem \ref{thm: deeppolygon}, a function 
\[ p:\{a_1,a_2,c_1,c_2,x',x''\} \rightarrow \kk^\times \]
extends to a (necessarily unique) deep point in $V(\CA(\Delta),\kk)$ iff Condition \eqref{eq: alternating product} holds; that is,
\[
\frac{
p(a_1)p(x'')p(c_1)
}{
p(a_2)p(c_2)p(x')
}=
1 
=
(-1)^{\frac{6+2}{2}} 
\]

Such a deep point is gluable along $x$ iff $p(x')=p(x'')$; therefore, a function 
\[ p:\{a_1,a_2,c_1,c_2,x\} \rightarrow \kk^\times \]
extends to a (necessarily unique) deep point in $V(\CA(\SS),\kk)$ iff
\[
\frac{
p(a_1)p(x)p(c_1)
}{
p(a_2)p(c_2)p(x)
}=
\frac{
p(a_1)p(c_1)
}{
p(a_2)p(c_2)
}=
1 = (-1)^{2+\frac{4}{2}}
\]
Note that this condition can be used to determine one value in terms of the other three.
Therefore, a deep point in $V(\CA(\SS),\kk)$ which is non-zero on $\{x\}$ is freely determined by a function 
\[p:\{a_1,a_2,c_1,x\}\rightarrow \kk^\times\]
So, the set of such deep points forms a copy of $(\kk^\times)^4$ inside $V(\CA(\SS),\kk)$.

The other component of the deep locus can be parametrized by finding a polygonal dissection which is not congruent to the first; for example, the marked arc $y$ in Figure \ref{fig: annulus 2 2 identified sides other}. By the same logic as before, a function 
\[ 
p:\{a_1,a_2,c_1,c_2,y\}\rightarrow \kk^\times
\]
extends to a (necessarily unique) deep point of $\CA(\SS)$ if and only if
\[ \frac{p(a_1)p(c_2)}{p(a_2)p(c_1)} = 1 \]
The set of such deep points forms another copy of $(\kk^\times)^4$ inside $V(\CA(\SS),\kk)$.
\end{ex}

\begin{figure}[h!tb]
\captionsetup[subfigure]{justification=centering}
\centering
\begin{subfigure}[c]{0.45\linewidth}
\centering
\begin{tikzpicture}
    \path[fill=black!10] (-2,1) to (2,1) to (2,-1) to (-2,-1) to (-2,1);
    \draw[thick] (-2,1) to (2,1) (-2,-1) to (2,-1);
    \draw[dashed] (-2,-1) to (-2,1) (2,-1) to (2,1);
    \draw[-Triangle,thin] (-2,-.1) to (-2,0);
    \draw[-Triangle,thin] (2,-.1) to (2,0);
    
    \node[dot] (1) at (-1,1) {};
    \node[dot] (2) at (1,1) {};
    \node[dot] (3) at (1,-1) {};
    \node[dot] (4) at (-1,-1) {};
    \draw[thick] (1) to node [midway, left=1pt] {$y$} (3);
    \draw[thick] (1) to node [midway, above=1pt] {$a_1$} (2);
    \draw[thick] (2) to node [midway, above=1pt] {$a_2$} (2,1);
    \draw[thick] (3) to node [midway, below=1pt] {$c_1$} (4);
    \draw[thick] (3) to node [midway, below=1pt] {$c_2$} (2,-1);

\end{tikzpicture}
\subcaption{The $(2,2)$-annulus $\SS$,\\with marked arc $y$}
\label{fig: annulus 2 2 identified sides other}
\end{subfigure}
\begin{subfigure}[c]{0.45\linewidth}
\centering
\begin{tikzpicture}
    \draw[fill=black!10,thick] (-3,1) to (1,1) to (3,-1) to (-1,-1) to (-3,1);
    \node[dot] (1) at (-3,1) {};
    \node[dot] (2) at (-1,1) {};
    \node[dot] (3) at (1,1) {};
    \node[dot] (4) at (3,-1) {};
    \node[dot] (5) at (1,-1) {};
    \node[dot] (6) at (-1,-1) {};
    \draw[thick] (2) to node [midway, above=1pt] {$a_2$} (3);
    \draw[thick] (1) to node [midway, above=1pt] {$a_1$} (2);
    \draw[thick] (4) to node [midway, below=1pt] {$c_1$} (5);
    \draw[thick] (6) to node [midway, below=1pt] {$c_2$} (5);
    \draw[thick] (6) to node [midway, left=1pt]  {$y'$} (1);
    \draw[thick] (4) to node [midway, right=1pt] {$y''$} (3);
\end{tikzpicture}
\subcaption{A hexagon $\Delta$ formed by cutting $\SS$ \\along the polygonal dissection $D=\{y\}$.}
\label{fig: annulus 2 2 poly diss other}
\end{subfigure}
\caption{Another polygonal dissection of the (2,2)-annulus.}
   \label{fig: annulus 2 2 other}
\end{figure}
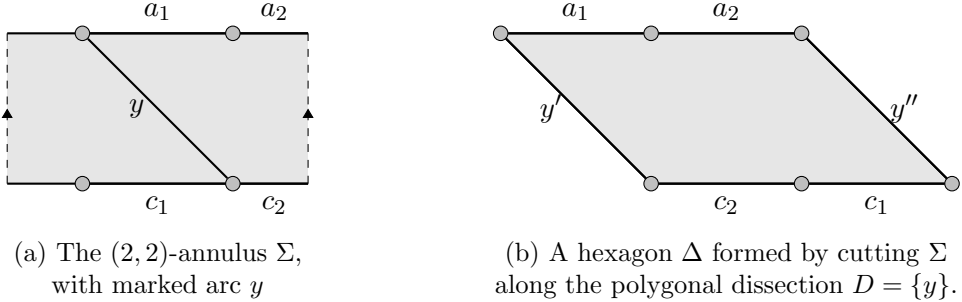

\draftnewpage

\appendix

\section{The deep components of unpunctured marked surfaces} \label{appendix: indexingdeepcomponents}

By Theorem \ref{thm: deepsurfacediss}, the deep locus of an unpunctured marked surface with an even number of marked points consists of $2^{2g+b-1}$-many disjoint algebraic tori. Proposition \ref{prop: polydisscount} showed these tori may be indexed by either (a) polygonal dissections of $\SS$ up to congruence, or (b) relative cohomology classes in $H^1(\SS,\MM;\mathbb{Z}_2)$ with value 1 on the boundary arcs. 

In this appendix, we construct several other sets which index the components of the deep locus. The results of this appendix are summarized by the following theorem.

%

\begin{thm}\label{thm: deepcomponents}
Let $\SS$ be a connected, unpunctured, triangulable marked surface with an even number of marked points. Then the following sets are in canonical bijection.
\begin{enumerate}
    \item Components of the deep locus of $V(\CA(\SS),\kk)$, for a fixed field $\kk$.
    \item Polygonal dissections of $\SS$, up to congruence.
    \item Relative cohomology classes in $H^1(\SS,\MM;\mathbb{Z}_2)$ with value 1 on the boundary arcs.
    \item \emph{Vanishing classes} of $\SS$; that is, sets of simple, non-boundary marked arcs in $\SS$ which contain an odd number of arcs in every triangle.
    \item Deep points in $V(\CA(\SS),\mathbb{Z}_2)$.
    \item Subsets of $D$, for a fixed polygonal dissection $D$ of $\SS$.
    \item \emph{Alternating double covers} of $\SS$; that is, double covers of $\SS$ with an alternating coloring of their marked points, such that the non-trivial deck transformation swaps colors.
\end{enumerate}
While the set of spin structures on $\SS$ has the same number of elements as each of these sets, in general there is no bijection which commutes with the action of the marked mapping class group.
\end{thm}




\begin{conv}
In this section, $\SS$ will always denote a connected, unpunctured, triangulable marked surface with genus $g$, $b$-many marked points, and an even number $m$ of marked points.
\end{conv}



\subsection{Vanishing classes}

The elementary tool in Section \ref{section: deepmarkedsurfaces} is Lemma \ref{lemma: trianglesurf}, which states that a deep point of $V(\CA(\SS),\kk)$ must vanish on an odd number of arcs in each triangle in $\SS$. In contrast with the polygonal case (Lemma \ref{lemma: trianglepoly}), this does not uniquely determine the set of simple arcs (i.e.~cluster variables) killed by a given deep point. We give such a set of arcs a name.

\begin{defn}
A \textbf{vanishing class} of $\SS$ is a subset $\cV$ of the set of simple, non-boundary marked arcs in $\SS$ which contains an odd number of arcs in each triangle in $\SS$.
\end{defn}

By Lemma \ref{lemma: trianglesurf}, each deep point $p$ of $V(\CA(\SS),\kk)$ determines a vanishing class
%
\begin{equation}\label{eq: vanishingclass}
\cV(p) \coloneqq \{ \text{simple marked arcs $a$ in $\SS$ such that $p(a)=0$}\}
\end{equation}
Note that a vanishing class is usually infinite.


\begin{ex}
A (unpunctured) polygon has a unique vanishing class, consisting of the set of arcs whose two indices differ by an even integer (c.f.~Lemma \ref{lemma: evenrule}).
\end{ex}

\begin{ex}
The $(2,2)$-annulus has 2 vanishing classes. One way to describe these classes is via the two colorings of the marked points in Figure \ref{fig: annulus 2 2 alternating colorings}. For each, the set of simple marked arcs whose endpoints have different colors is a vanishing class, and these are the only two vanishing classes.
\end{ex}

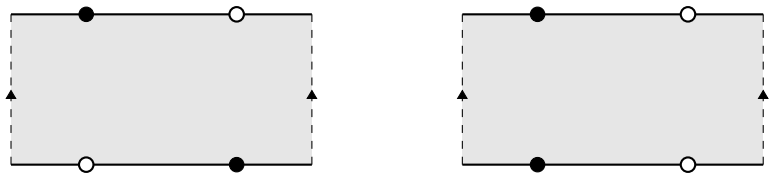
\begin{figure}[h!tb]
\begin{tikzpicture}
	\node (a) at (-3,0) {
		\begin{tikzpicture}
			\path[fill=black!10] (-2,1) to (2,1) to (2,-1) to (-2,-1) to (-2,1);
			\draw[thick] (-2,1) to (2,1) (-2,-1) to (2,-1);
			\draw[dashed] (-2,-1) to (-2,1) (2,-1) to (2,1);
			\draw[-Triangle,thin] (-2,-.1) to (-2,0);
			\draw[-Triangle,thin] (2,-.1) to (2,0);
		    
			\node[bpoint] (1) at (-1,1) {};
			\node[wpoint] (2) at (1,1) {};
			\node[bpoint] (3) at (1,-1) {};
			\node[wpoint] (4) at (-1,-1) {};
		
		\end{tikzpicture}
	};
	\node (c) at (3,0) {
		\begin{tikzpicture}
			\path[fill=black!10] (-2,1) to (2,1) to (2,-1) to (-2,-1) to (-2,1);
			\draw[thick] (-2,1) to (2,1) (-2,-1) to (2,-1);
			\draw[dashed] (-2,-1) to (-2,1) (2,-1) to (2,1);
			\draw[-Triangle,thin] (-2,-.1) to (-2,0);
			\draw[-Triangle,thin] (2,-.1) to (2,0);
		    
			\node[bpoint] (1) at (-1,1) {};
			\node[wpoint] (2) at (1,1) {};
			\node[wpoint] (3) at (1,-1) {};
			\node[bpoint] (4) at (-1,-1) {};
		
		\end{tikzpicture}
	};
\end{tikzpicture}
\caption{Coloring vertices to describe vanishing classes of the $(2,2)$-annulus}
\label{fig: annulus 2 2 alternating colorings} 
\end{figure}

\begin{warn}
It is not always possible to describe the vanishing classes by coloring the vertices of $\SS$, as in Figure \ref{fig: annulus 2 2 alternating colorings}! More generally, one needs to pass to a double cover (see Section \ref{section: alternatingdoublecover}).
\end{warn}

\begin{prop}\label{prop: vanishingclasses}
The vanishing class of a deep point (Equation \eqref{eq: vanishingclass})
%
is constant on components of the deep locus of $V(\CA(\SS),\kk)$, and induces a bijection between components of the deep locus of $V(\CA(\SS),\kk)$ and vanishing classes of $\SS$.
\end{prop}

\begin{proof}
By Proposition \ref{prop: cohomologyclass}.1, the cohomology class $\chi_p$ of a deep point $p$ is uniquely determined by the set of simple non-boundary arcs it sends to $0$; that is, its vanishing class $\cV(p)$. Therefore, $\cV(p)=\cV(p')$ iff $\chi_p=\chi_{p'}$ iff $p$ and $p'$ are in the same component of the deep locus.

To show surjectivity, consider a vanishing class $\cV$. 
By Lemma \ref{lemma: polydissavoid}, we can find a polygonal dissection $D$ disjoint from $\cV$. By Theorem \ref{thm: deepsurfacediss}, we can find a deep point $p$ of $V(\CA(\SS),\kk)$ which is non-zero on $D$. By Proposition \ref{prop: cohomologyclass}, $\chi_p=\chi_D$; therefore, $\cV(p)=\cV$.
\end{proof}

\begin{coro}\label{coro: uniquepolydiss}
Every polygonal dissection of $\SS$ is disjoint from a unique vanishing class of $\SS$.
\end{coro}

\begin{proof}
Existence follows from Lemma \ref{lemma: polydissavoid}.
Let $\cV$ and $\cV'$ be vanishing classes disjoint from a polygonal dissection $D$. By Proposition \ref{prop: vanishingclasses}, there are deep points $p,p'$ such that $\cV=\cV(p)$ and $\cV'=\cV(p')$. Since $p$ and $p'$ are non-zero on $D$, $\chi_p=\chi_D=\chi_{p'}$; therefore, $\cV=\cV'$.
\end{proof}





\subsection{Deep points over \texorpdfstring{$\mathbb{Z}_2$}{Z2}}

When $\kk=\mathbb{Z}_2$ is the field with two elements, the algebraic torus $(\kk^\times)^n$ consists of a single point for any $n$. Therefore, the deep points of $V(\CA(\SS),\mathbb{Z}_2)$ are in bijection with polygonal dissections of $D$ up to congruence by Theorem \ref{thm: deepsurfacediss}, and so they may be used to index the deep components of $V(\CA(\SS),\kk)$ for any field $\kk$.

This correspondence may be realized more directly as follows.

\begin{prop}
Let $\cV$ be a vanishing class of $\SS$. Then the function from the set of simple marked arcs in $\SS$ to $\mathbb{Z}_2$ defined by 
    \begin{equation}\label{eq: functionofvclass}
    a \mapsto  \left\{
    \begin{array}{cc}
    0 & \text{if $a\in \cV$} \\
    1 & \text{if $a\not\in\cV$}
    \end{array}
    \right\}
    \end{equation}
\begin{enumerate}
    \item extends uniquely to a relative cohomology class $\chi_\cV\in H^1(\SS,\MM;\mathbb{Z}_2)$ with value $1$ on the boundary arcs, and 
    \item extends uniquely to a deep point $p_\cV\in V(\CA(\SS),\mathbb{Z}_2)$; that is, a ring homomorphism $\CA(\SS)\rightarrow \mathbb{Z}_2$ which kills at least one marked arc in each triangulation.
\end{enumerate}
Furthermore, every such relative cohomology class and deep point arises in this way.
\end{prop}

\begin{proof}
Part (1) is a restatement of Proposition \ref{prop: cohomologyclass}.1. To show Part (2), let $\cV$ be a vanishing class of $\SS$. By Proposition \ref{prop: vanishingclasses} and the fact that $|\mathbb{Z}_2^\times|=1$, there is a unique deep point $p\in V(\CA(\SS),\mathbb{Z}_2)$ such that $\cV(p)=\cV$. 
For any simple marked arc $a$, $p(a) = 0$ if and only if $a\in \cV(p)$ by Equation \eqref{eq: vanishingclass}. Since there is only one other element in $\mathbb{Z}_2$, $p(a) =1$ if and only if $a\not\in \cV(p)$. Therefore, $p$ extends the function given in Equation \eqref{eq: functionofvclass}. 

Therefore, these constructions define injections from the set of vanishing classes to (1) the sets of relative cohomology classes with value $1$ on the boundary and (2) deep points in $V(\CA(\SS),\mathbb{Z}_2)$. Since all three sets have cardinality $2^{2g+b-1}$, these constructions are bijections.
\end{proof}

\subsection{Subsets of a polygonal dissection}

One of the simpler characterizations of the deep components is by the subset of arcs they kill in a fixed polygonal dissection.

\begin{prop} \label{prop: uniquevanishingclassforeachsubset}
Let $D$ be a polygonal dissection of $\SS$. For every subset $S\subset D$, there is a unique vanishing class $\cV$ such that $D\cap \cV=S$.
\end{prop}

\begin{proof}
By Corollary \ref{coro: uniquepolydiss}, there exists a unique vanishing class of $\SS$ disjoint from $D$. 


For our inductive hypothesis, let $S \subseteq D$ be nonempty and assume that for each proper subset $S^\prime \subsetneq S$ and every polygonal dissection $D^\prime$ containing $S^\prime$, there exists a unique vanishing class $\cV^\prime$ such that $D^\prime \cap \cV^\prime = S^\prime$. 
Because $\SS$ has nonempty boundary, we can choose arcs $\alpha \in S$ and $\beta \not\in S$ which are adjacent edges in the boundary of $\SS \smallsetminus D$. Let $\gamma$ be the arc in $\SS$ that completes a triangle with $\alpha$ and $\beta$ in $\SS\smallsetminus D$. Then $D_0 \coloneqq (D \cup \{\gamma\})\smallsetminus \{\alpha \}$ is a polygonal dissection, and $S_0 \coloneqq S\smallsetminus \{a\} $ is a proper subset of $S$ that is contained in $D_0$. By the inductive hypothesis, there exists a unique vanishing class $\cV_0$ such that $D_0 \cap \cV_0 = S_0$. Since $\{\beta, \gamma\} \subset (D_0 \cup \partial \SS)$ and $\{\beta,\gamma\} \cap S_0 = \emptyset$, Lemma \ref{lemma: trianglesurf} implies that $\alpha \in \cV_0$. Hence, $\cV_0$ is the unique vanishing class such that $D \cap \cV_0 = S$.
\end{proof}

Note that the number of arcs in $D$ is $2g+b-1$ by Proposition \ref{prop: polydisscounts}, so the number of subsets of $D$ is $2^{2g+b-1}$. By Theorem \ref{thm: deepsurfacediss} and Proposition \ref{prop: vanishingclasses}, the number of vanishing classes is also $2^{2g+b-1}$. Since the sets have the same cardinality, Proposition \ref{prop: uniquevanishingclassforeachsubset} implies that these sets are in bijection.

\subsection{Alternating double covers}\label{section: alternatingdoublecover}

It is well-known that elements of $H^1(\SS;\mathbb{Z}_2)$ may be identified with double covers of $\SS$ up to equivalence. Over the next several sections, we relate elements of $H^1(\SS,\MM;\mathbb{Z}_2)$ with value $1$ on the boundary arcs to double coverings of $\SS$ endowed with an \emph{alternating coloring} of their marked points.


\begin{defn}
An \textbf{alternating double cover} of $\SS$ is a (not necessarily connected) 2-fold covering space\footnote{A \textbf{covering space} of a marked surface is a covering space of the underlying surface whose marked points are the points which map to marked points in the original surface.} $\SS'\rightarrow \SS$, with a black and white coloring of the marked points of $\SS'$, such that
\begin{itemize}
    \item every pair of adjacent marked points in $\SS'$ has opposite colors, and
    \item the preimages of each marked point in $\SS$ consists of one point of each color.
\end{itemize}
An \textbf{equivalence} of alternating double covers of $\SS$ is a diffeomorphism which commutes with the covering map to $\SS$ and either sends every marked point to a marked point of the same color, or sends every marked point to a marked point of the opposite color.
\end{defn}

Swapping the colors in an alternating double cover yields an equivalent double cover; for this reason, the specific colors do not matter as much as the underlying partition of the marked points.
We say an alternating double cover is \textbf{trivial} if the underlying covering space $\SS'$ is a trivial double cover of $\SS$ (that is, two copies of $\SS$).

\begin{figure}[h!tb]
\begin{tikzpicture}
	\node (a) at (-5,0) {
		\begin{tikzpicture}
			\path[fill=black!10] (-2,1) to (2,1) to (2,-1) to (-2,-1) to (-2,1);
			\draw[thick] (-2,1) to (2,1) (-2,-1) to (2,-1);
			\draw[dashed] (-2,-1) to (-2,1) (2,-1) to (2,1);
			\draw[-Triangle,thin] (-2,-.1) to (-2,0);
			\draw[-Triangle,thin] (2,-.1) to (2,0);
		    
			\node[bpoint] (1) at (-1,1) {};
			\node[wpoint] (2) at (1,1) {};
			\node[bpoint] (3) at (1,-1) {};
			\node[wpoint] (4) at (-1,-1) {};
		
		\end{tikzpicture}
	};
	\node (b) at (0,0) {
		\begin{tikzpicture}
			\path[fill=black!10] (-1,1) to (1,1) to (1,-1) to (-1,-1) to (-1,1);
			\draw[thick] (-1,1) to (1,1) (-1,-1) to (1,-1);
			\draw[dashed] (-1,-1) to (-1,1) (1,-1) to (1,1);
			\draw[-Triangle,thin] (-1,-.1) to (-1,0);
			\draw[-Triangle,thin] (1,-.1) to (1,0);
		    
			\node[dot] (1) at (0,1) {};
			\node[dot] (2) at (0,-1) {};
		
		\end{tikzpicture}
	};
	\node (c) at (5,0) {
		\begin{tikzpicture}
			\path[fill=black!10] (-2,1) to (2,1) to (2,-1) to (-2,-1) to (-2,1);
			\draw[thick] (-2,1) to (2,1) (-2,-1) to (2,-1);
			\draw[dashed] (-2,-1) to (-2,1) (2,-1) to (2,1);
			\draw[-Triangle,thin] (-2,-.1) to (-2,0);
			\draw[-Triangle,thin] (2,-.1) to (2,0);
		    
			\node[bpoint] (1) at (-1,1) {};
			\node[wpoint] (2) at (1,1) {};
			\node[wpoint] (3) at (1,-1) {};
			\node[bpoint] (4) at (-1,-1) {};
		
		\end{tikzpicture}
	};
	\draw[-angle 90] (a) to node[above] {$\pi_1$} (b);
	\draw[-angle 90] (c) to node[above] {$\pi_2$} (b);
\end{tikzpicture}
\caption{Alternating double covers of the $(1,1)$-annulus (up to equivalence)}
\label{fig: annulus 1 1 covers inline} 
\end{figure}
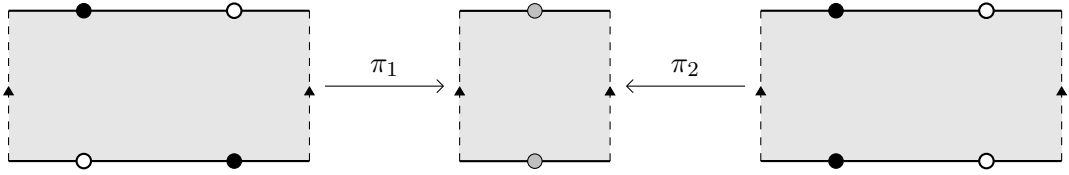

\begin{ex}
The $(1,1)$-annulus (that is, the annulus with 1 marked point on each boundary component) has two alternating double covers up to equivalence (Figure \ref{fig: annulus 1 1 covers inline}). Note that a trivial double cover consisting of two copies of $\SS$ would have been allowed, except there is no way to color the marked points in such a double cover so as to satisfy the alternating condition.
\end{ex}

\subsection{Constructing alternating double covers}

The natural way to construct an alternating double cover is to first choose a (not necessarily connected) double cover of $\SS$, and then attempt to color the marked points in $\SS'$. The following proposition summarizes when and how this works.

\begin{prop}
A double cover $\SS'\rightarrow\SS$ admits an alternating coloring if and only if 
\begin{itemize}
    \item each component in $\partial\SS$ with an odd number of marked points has connected preimage, and
    \item each component in $\partial \SS$ with an even number of marked points has disconnected preimage.
\end{itemize}
Such a double cover exists iff the total number of marked points in $\SS$ is even, in which case there are $2^{2g}$-many such double covers and each admit $2^{b-1}$-many alternating colorings up to equivalence.
\end{prop}

\begin{rem}
So, $\SS$ admits a trivial alternating double cover iff each component of $\partial \SS$ has an even number of marked points, in which case there are $2^{b-1}$-many trivial alternating double covers.
\end{rem}

\begin{proof}
Consider a double cover $\SS'\rightarrow \SS$, and let $S_i$ be a boundary component of $\SS$ with $m_i$-many marked points.
\begin{enumerate}
    \item If the preimage of $S_i$ in $\SS'$ is connected, then this preimage is a single circle with $2m_i$-many marked points. There are always exactly two ways to color these points so that adjacent points have opposite colors. However, the two preimages of a marked point in $S_i$ are $m_i$-many steps apart, so they will receive opposite colors if and only if $m_i$ is odd.
    
    \item If the preimage of $S_i$ in $\SS'$ is disconnected, then this preimage consists of two circles, each with $m_i$-many marked points. Each of these circles may be given an alternating coloring if and only if $m_i$ is even. Once such a coloring is chosen, the alternating coloring of the other circle is determined by the condition that preimages have opposite colors.
\end{enumerate}
This confirms that the parity conditions on boundary components of $\SS$ are a necessary and sufficient for the existence of alternating colorings. Furthermore, when they exist, we see that there are always exactly 2 choices of alternating coloring for each boundary component in $\SS$. Since swapping colors is an equivalence, this gives $2^{b-1}$-many alternating colorings of $\SS'$ up to equivalence.

To count double covers of $\SS$ satisfying the parity conditions, recall that double covers of a manifold $\SS$ are in bijection with $H^1(\SS,\mathbb{Z}_2)$, where a closed loop in $\SS$ lifts along the double cover iff it is killed by the corresponding cohomology class. 

Define $\chi_0\in H^1(\partial\SS;\mathbb{Z}_2)$ to be the class which sends each boundary component in $\SS$ to the parity of the number of marked points on that boundary component. Then double covers of $\SS$ which satisfy the parity conditions are in bijection the preimage of $\chi_0$ under the map
\[ H^1(\SS;\mathbb{Z}_2)\rightarrow H^1(\partial\SS;\mathbb{Z}_2) \]
%

Consider the long exact sequence (with dimensions computed in the proof of Proposition \ref{prop: polydisscount}) 
\[
0\!
\rightarrow
\underbrace{
H^0(\SS;\mathbb{Z}_2)
}_{\dim=1}
\rightarrow
\underbrace{
H^0(\partial\SS;\mathbb{Z}_2)
}_{\dim=b}
\rightarrow
\underbrace{
H^1(\SS,\partial \SS;\mathbb{Z}_2)
}_{\dim=2g+b-1}
\xrightarrow{f}
\underbrace{
H^1(\SS;\mathbb{Z}_2)
}_{\dim=2g+b-1}
\xrightarrow{g}
\underbrace{
H^1(\partial \SS;\mathbb{Z}_2)
}_{\dim=b}
 \xrightarrow{h}
\underbrace{
 H^2(\SS,\partial \SS;\mathbb{Z}_2)
}_{\dim=1}
\rightarrow
\!0
\]
Under the identification $H^2(\SS,\partial \SS;\mathbb{Z}_2)\simeq \mathbb{Z}_2$, $h$ sends a class in $H^1(\partial \SS;\mathbb{Z}_2)$ to the sum of its values on the boundary components. 
In particular, $\SS$ admits an alternating double cover iff it admits a double cover satisfying the parity conditions iff the preimage of $\chi_0$ is non-empty iff $h(\chi_0)=0$ iff the total number of marked points $m$ is even.

When $h(\chi_0)=0$, the cardinality of the preimage $g^{-1}(\chi_0)$ is equal to the cardinality of the preimage $g^{-1}(0)$ of $0$; that is, the kernel of $g$. Using the above dimensions, we compute that
\[ \dim(\mathrm{ker}(g)) = \dim(\mathrm{im}(f)) = \dim (H^1(\SS,\partial\SS;\mathbb{Z}_2)) -\dim(H^0(\partial\SS;\mathbb{Z}_2) ) +\dim(H^0(\SS;\mathbb{Z}_2)) = 2g \]
Therefore, the number of preimages of $\chi_0$ is $2^{2g}$ when $m$ is even and $0$ otherwise.
\end{proof}

\subsection{Alternating double covers and deep components}

Assuming that the total number of marked points $m$ is even, each polygonal dissection of $\SS$ can be used to construct an alternating double cover of $\SS$, as follows.


\begin{cons}\label{cons: polydiss}
Let $D$ be a polygonal dissection of $\SS$. By Proposition \ref{prop: polydisscounts}, $\SS\smallsetminus D$ has an even number of sides. As a result, there are two ways to color the marked points in $\SS\smallsetminus D$ alternating between black and white (differing by swapping the colors). Take two copies of the polygon $\SS \smallsetminus D$, one with each of the two possible alternating colorings. For each arc in $D$, glue the two pairs of preimages of that arc back together in such a way that the colorings at the endpoints match. The resulting surface $\SS'$ is an alternating double cover of $\SS$ (e.g.~ Figure \ref{fig: adcfrompolydiss}).
\end{cons}

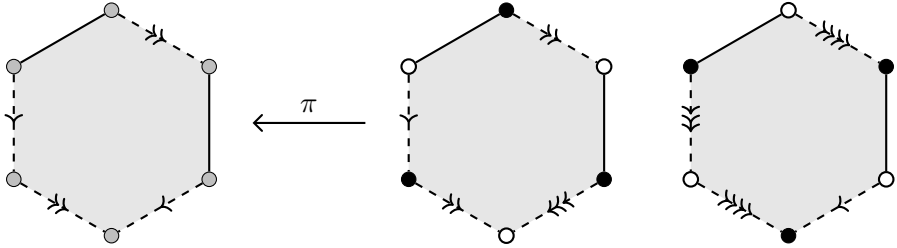
\begin{figure}[h!tb]
\centering
\begin{tikzpicture}[xscale=1,scale=.75,baseline={(0,0)}]
    \path[use as bounding box] (-2,-2.5) rectangle (14,2.5);
    \begin{scope}
        \path[fill=black!10] 
        (90-0*60:2)
        to (90-1*60:2)
        to (90-2*60:2)
        to (90-3*60:2)
        to (90-4*60:2)
        to (90-5*60:2)
        to (90-6*60:2)
        ;
        \node[dot] (1) at (90-0*60:2) {};
        \node[dot] (2) at (90-1*60:2) {};
        \node[dot] (3) at (90-2*60:2) {};
        \node[dot] (4) at (90-3*60:2) {};
        \node[dot] (5) at (90-4*60:2) {};
        \node[dot] (6) at (90-5*60:2) {};
        \draw[thick] (6) to (1) (2) to (3);
        \begin{scope}[decoration={markings,mark=at position 0.5 with {\arrow{>}}}]
            \draw[thick,dashed,postaction={decorate}] (3) to (4);
            \draw[thick,dashed,postaction={decorate}] (6) to (5);
        \end{scope}
        \begin{scope}[decoration={markings,mark=at position 0.55 with {\arrow{>>}}}]
            \draw[thick,dashed,postaction={decorate}] (1) to (2);
            \draw[thick,dashed,postaction={decorate}] (5) to (4);
        \end{scope}
    \end{scope}
    \draw[-angle 90,thick] (4.5,0) to node[above] {$\pi$} (2.5,0);
    \begin{scope}[xshift=7cm]
    \path[fill=black!10] 
    (90-0*60:2)
    to (90-1*60:2)
    to (90-2*60:2)
    to (90-3*60:2)
    to (90-4*60:2)
    to (90-5*60:2)
    to (90-6*60:2)
    ;
    \node[bpoint] (1) at (90-0*60:2) {};
    \node[wpoint] (2) at (90-1*60:2) {};
    \node[bpoint] (3) at (90-2*60:2) {};
    \node[wpoint] (4) at (90-3*60:2) {};
    \node[bpoint] (5) at (90-4*60:2) {};
    \node[wpoint] (6) at (90-5*60:2) {};
        \draw[thick] (6) to (1) (2) to (3);
        \begin{scope}[decoration={markings,mark=at position 0.6 with {\arrow{>>>}}}]
            \draw[thick,dashed,postaction={decorate}] (3) to (4);
        \end{scope}
        \begin{scope}[decoration={markings,mark=at position 0.5 with {\arrow{>}}}]
            \draw[thick,dashed,postaction={decorate}] (6) to (5);
        \end{scope}
        \begin{scope}[decoration={markings,mark=at position 0.55 with {\arrow{>>}}}]
            \draw[thick,dashed,postaction={decorate}] (1) to (2);
        \end{scope}
        \begin{scope}[decoration={markings,mark=at position 0.55 with {\arrow{>>}}}]
            \draw[thick,dashed,postaction={decorate}] (5) to (4);
        \end{scope}
    \end{scope}
    \begin{scope}[xshift=12cm]
    \path[fill=black!10] 
    (90-0*60:2)
    to (90-1*60:2)
    to (90-2*60:2)
    to (90-3*60:2)
    to (90-4*60:2)
    to (90-5*60:2)
    to (90-6*60:2)
    ;
    \node[wpoint] (1) at (90-0*60:2) {};
    \node[bpoint] (2) at (90-1*60:2) {};
    \node[wpoint] (3) at (90-2*60:2) {};
    \node[bpoint] (4) at (90-3*60:2) {};
    \node[wpoint] (5) at (90-4*60:2) {};
    \node[bpoint] (6) at (90-5*60:2) {};
    \draw[thick] (6) to (1) (2) to (3);
        \begin{scope}[decoration={markings,mark=at position 0.5 with {\arrow{>}}}]
            \draw[thick,dashed,postaction={decorate}] (3) to (4);
        \end{scope}
        \begin{scope}[decoration={markings,mark=at position 0.6 with {\arrow{>>>}}}]
            \draw[thick,dashed,postaction={decorate}] (6) to (5);
        \end{scope}
        \begin{scope}[decoration={markings,mark=at position 0.65 with {\arrow{>>>>}}}]
            \draw[thick,dashed,postaction={decorate}] (1) to (2);
        \end{scope}
        \begin{scope}[decoration={markings,mark=at position 0.65 with {\arrow{>>>>}}}]
            \draw[thick,dashed,postaction={decorate}] (5) to (4);
        \end{scope}
    \end{scope}
\end{tikzpicture}
\caption{Constructing the alternating double covering of a polygonal dissection}
\label{fig: adcfrompolydiss}
\end{figure}

Conversely, alternating double covers determine vanishing classes via the following observation. Given a simple marked arc $a$ in $\SS$, there are two lifts of $a$ along an alternating double cover $\SS'\rightarrow \SS$. Since these two lifts are related by the unique non-trivial deck transformation, either (a) both lifts of $a$ have same colored endpoints, or (b) both lifts of $a$ have opposite colored endpoints.

\begin{prop}\label{prop: vanclassaltdouble}
Given an alternating double cover $\SS'\rightarrow \SS$, the set 
\[ \cV(\SS') \coloneqq \{\text{simple marked arcs $a$ in $\SS$ such that either lift of $a$ to $\SS$ has same colored ends}\} \]
is a vanishing class of $\SS$. If $D$ is a polygonal dissection of $\SS$ with associated alternating double cover $\SS'$ (Construction \ref{cons: polydiss}), then $\chi_{\cV(\SS')} = \chi_D$.
\end{prop}

\begin{proof}
If $a$ is a boundary arc in $\SS$, then both lifts of $a$ to $\SS'$ are boundary arcs, so the endpoints have opposite colors and $a\not\in \cV(\SS')$. Consider a triangle $\{a,b,c\}$ in $\SS$, and let $\{a',b',c'\}$ be a lift to $\SS'$. If all three vertices of the triangle $\{a',b',c'\}$ have the same color, then all three of $\{a,b,c\}$ are in $\cV(\SS')$. If one vertex of the triangle $\{a',b',c'\}$ has a different color than the other two, then exactly one of $\{a,b,c\}$ is in $\cV(\SS')$. Therefore, $\cV(\SS')$ is a vanishing class of $\SS$.

Let $D$ be a polygonal dissection of $\SS$ with associated alternating double cover $\SS'$. 
Since the coloring of marked points in $\SS'$ alternates around the boundary of each copy of $\SS\smallsetminus D$ in Construction \ref{cons: polydiss}, both lifts of an arc in $D$ to an arc in $\SS'$ have opposite colored ends. As a consequence, $D$ is disjoint from $\cV(\SS')$.
By Proposition \ref{prop: cohomologyclass}.2, $\chi_{\cV(\SS')} (a) =1$ for each arc $a\in D$ and each boundary arc $a$. By the uniqueness in Proposition \ref{prop: cohomologyclass}.3, $\chi_{\cV(\SS')} = \chi_D$.
%
%
%
\end{proof}




\begin{coro}
Construction \ref{cons: polydiss} and Proposition \ref{prop: vanclassaltdouble} define bijections between
\begin{itemize}
    \item polygonal dissections of $\SS$ up to congruence,
    \item alternating double covers of $\SS$ up to equivalence, and
    \item vanishing classes of $\SS$.
\end{itemize}  
\end{coro}

This implies that alternating double covers of $\SS$ are in bijection with the other sets in Theorem \ref{thm: deepcomponents}. Many of these bijections can be formulated directly; for example, if $p$ is a deep point of $V(\CA(\SS),\kk)$ with vanishing class $\cV(p)$, then the alternating double cover $\SS'\rightarrow \SS$ associated to $\cV(p)$ is the unique alternating double cover of $\SS$ such that the pullback of $p$ to $V(\CA(\SS'),\kk)$ kills precisely the marked arcs in $\SS'$ whose ends have the same color.

\subsection{Spin structures}

Finally, we dispel a tempting misconception: that components of the deep locus of $V(\CA(\SS),\kk)$ may be identified with spin structures on $\SS$.

There are many equivalent ways to define spin structures on a surface-with-boundary; here is one. Define the \textbf{direction bundle} $\mathcal{D}\SS$ to $\SS$ to be the quotient of the punctured tangent bundle $\mathcal{T}^\times\SS$ by scaling by $\mathbb{R}_+$.\footnote{Given a Riemannian metric of $\SS$, the direction bundle may be identified with the unit tangent bundle.} Elements of $\mathcal{D}\SS$ are non-zero tangent vectors to $\SS$ up to a positive scalar. The direction bundle is a bundle of circles over the base $\SS$.
A \textbf{spin structure} on $\SS$ is a choice of double cover $\mathcal{S}\SS\rightarrow \mathcal{D}\SS$ such that the preimage of any circular fiber in $\mathcal{D}\SS$ is connected. 

The set of spin structures on $\SS$ is a torsor over $H^1(\SS;\mathbb{Z}_2)$ \cite[Theorem 3A]{Joh80}, and so there are $2^{2g+b-1}$-many spin structures on $\SS$, which coincides with the number of deep components of $V(\CA(\SS),\kk)$. Therefore, bijections between the two sets exist; however, the following example shows that no such bijection can be `natural'.

\begin{prop}
Let $\SS$ be the $(1,1)$-annulus (so $2^{2g+b-1}=2$), and let $\gamma:\SS\rightarrow\SS$ be the Dehn twist around the unique simple loop in $\SS$. Then 
\begin{enumerate}
    \item the induced action of $\gamma$ on the set of spin structures of $\SS$ is trivial, but
    \item the induced action of $\gamma$ on the set of components of the deep locus of $\CA(\SS)$ is non-trivial.
\end{enumerate}
As a consequence, there can be no bijection between the two sets which commutes with the action of the mapping class group of $\SS$.
\end{prop}

\begin{proof}
Let $\ell$ be the (homotopy class of) the simple loop in $\SS$ (Figure \ref{fig: annulus 2 2 loop}). The derivative of a parametrization of $\ell$ lifts to a loop $d\ell$ in $\mathcal{D}\SS$, and the homotopy class of $d\ell$ is independent of the choice of parametrization. The two spin structures on $\SS$ differ in whether $d\ell$ has connected or disconnected preimage along the double cover $\mathcal{S}\SS\rightarrow \mathcal{D}\SS$. Since the Dehn twist around $\ell$ fixes $d\ell$, the Dehn twist does not exchange these two spin structures.

Conversely, the Dehn twist around $\ell$ exchanges the two alternating double covers of $\SS$ (Figure \ref{fig: annulus 1 1 covers inline}). Therefore, the action of the Dehn twist on the other sets in Theorem \ref{thm: deepcomponents} is also non-trivial.
\end{proof}

\draftnewpage

\bibliographystyle{alpha}
\bibliography{GlobalBib}

\end{document}